\documentclass[12pt]{UOthesis}

\usepackage{fancybox}
\usepackage{shadow}
\usepackage[all]{xy}
\usepackage{url}

\allowdisplaybreaks

\NoListoftables
\NoListoffigures

\setUOname{Caroline El-Cha\^ar}         
\setUOcpryear{2010}               
\setUOtitle{The Onsager Algebra}  

\msc               

\setUOabstract{In this thesis, four realizations of the Onsager algebra are explored. We begin with its original definition as introduced by Lars Onsager. We then examine how the Onsager algebra can be presented as a Lie algebra with two generators and two relations. The third realization of the Onsager algebra consists of viewing it as an equivariant map algebra which then gives us the tools to classify its closed ideals. Finally, we examine the Onsager algebra as a subalgebra of the tetrahedron algebra. Using this fourth realization, we explicitly describe all its ideals.  
}

\setUOresume{Dans cette th\`ese, quatre r\'ealisations de l'alg\`ebre d'Onsager sont explor\'ees. Nous commen\c cons avec sa d\'efinition d'origine introduite par Lars Onsager. Ensuite, nous examinons comment l'alg\`ebre d'Onsager peut \^etre pr\'esent\'ee comme une alg\`ebre de Lie engendr\'ee par deux g\'en\'erateurs avec deux relations. La troisi\`eme r\'ealisation de l'alg\`ebre d'Onsager consiste à l'examiner sous la forme d'une alg\`ebre d'applications \'equivariantes ce qui nous permet ensuite de classifier ses id\'eaux ferm\'es. Finalement, nous examinons l'alg\`ebre d'Onsager comme une sous-alg\`ebre de l'alg\`ebre t\'etra\`edre. En exploitant cette quatri\`eme r\'ealisation, nous d\'ecrivons explicitement tous les id\'eaux de l'alg\`ebre d'Onsager.
}

\setUOthanks{I would first of all like to acknowledge my supervisors, Professor Erhard Neher and Professor Alistair Savage. I am honored to have worked with them and will forever be grateful for their support throughout the research, writing and editing process. They demonstrated immense patience and were unfailing when it came to pointing out my various mathematical and grammatical errors. It was a great learning experience working with this great tag-team. 
 
I would also like to acknowledge the financial support received from the Department of Mathematics and Statistics and the University of Ottawa.
 
On a personal note, I would like to thank my dear friends for their continued help, guidance and most importantly friendship throughout my endeavors. Last, but certainly not least, I would like to thank my family for their undying love and support. 
}

\setUOdedicationsText{This thesis is dedicated to my parents who have sacrificed immensely in order to give their children a multitude of opportunities. They have always encouraged and supported us to pursue not only success but also our dreams. For their undying love and support, I am eternally grateful.}

\DeclareMathOperator{\Span}{span}
\DeclareMathOperator{\ad}{ad}
\DeclareMathOperator{\lcm}{lcm}
\DeclareMathOperator{\Mod}{mod}
\DeclareMathOperator{\Spec}{Spec}
\DeclareMathOperator{\Int}{Int}

\begin{document}


\chapter{Introduction}
\label{sec:Introduction}

\section{Outline}
\label{sec:Outline}

The purpose of this thesis is to explore various realizations of an algebra which took on the name of Nobel Prize winning Norwegian-American physical chemist and theoretical physicist, Lars Onsager (1903-1976). Four realizations of the Onsager algebra will be studied, following the order in which they were discovered. 

We begin, in this chapter, by chronologically going over the key events pertinent to this thesis including the birth of the Onsager algebra and its original definition, \cite{OnsagerL44}. In Chapter 2, we examine how the Onsager algebra can be presented as a Lie algebra with two generators and two relations. This is done with the help of the works \cite{DolanLGradyM82}, \cite{PerkJ89}, \cite{DaviesB91} and \cite{RoanS91}. 

In Chapter 3, we consider the realization of the Onsager algebra as an equivariant map algebra (a Lie algebra of regular maps equivariant under a group action). Using this correspondence, one can then describe the closed ideals of the Onsager algebra. This is done in reference to the works \cite{RoanSDateE00} and \cite{DateERoanS00}.

Finally, in Chapter 4, we examine the tetrahedron algebra and its realization as a map algebra and show that the Onsager algebra is isomorphic to a subalgebra of the tetrahedron algebra. Using this realization, we can then explicitly describe all ideals of the Onsager algebra. This is done with the help of the works \cite{HartwigBTerwilligerP07} and \cite{ElduqueA07}.

The literature presents these results while omitting many steps and details. Therefore, the goal of this thesis is to fill in these gaps in order to have complete and concise arguments to support them.

\section{Background}
\label{sec:Background}

Throughout this thesis, unless otherwise specified, it is assumed that $k$ is an arbitrary field of characteristic zero, $t$ is an indeterminate and all tensor products are over $k$. Furthermore, we will denote the set of natural numbers, $\{0,1,2,\dots\}$, by $\mathbb{N}$ and the set of positive natural numbers, $\{1,2,\dots\}$, by $\mathbb{N}_+$.

In this section, we give a quick review of some basic definitions that will be used throughout the thesis. More details can be found in \cite{HumphreysJ72} and \cite{ErdmannKWildonM06}.

\begin{defn}[Lie Algebra] A \emph{Lie algebra} over $k$ is a vector space $L$, together with a bilinear map, the \emph{Lie bracket},
\[ L \times L \rightarrow L, \quad (x,y) \mapsto [x,y], \]
satisfying the following properties:
\begin{align}
    [x,x] &= 0, \quad \text{for all } x \in L, \\
    [x,[y,z]] + [y,[z,x]] + [z,[x,y]] &= 0, \quad \text{for all } x,y,z \in L. \tag{J.I.} \label{JacobiId}
\end{align}
The Lie bracket $[x,y]$ is often referred to as the \emph{commutator} of $x$ and $y$. Condition \eqref{JacobiId} is known as the \emph{Jacobi identity} .
\end{defn}

Let $\mathfrak{gl}(L)$ denote the set of endomorphisms of the vector space $L$. This is also a vector space over $k$ and becomes a Lie algebra, known as the \emph{general linear algebra}, if we define the Lie bracket by \[ [x,y] = x \circ y - y \circ x, \qquad \text{for } x,y \in \mathfrak{gl}(L),\] where $\circ$ denotes the composition of maps. 

\begin{defn}[Lie Algebra Homomorphism] Let $L_1$ and $L_2$ be Lie algebras over $k$. The map $\varphi : L_1 \rightarrow L_2$ is said to be a \emph{Lie algebra homomorphism} if $\varphi$ is a linear map and satisfies $\varphi ([x,y]) = [ \varphi(x),\varphi(y) ]$ for all $x$, $y \in L_1$. If $\varphi$ is also bijective, it is said to be an \emph{isomorphism}.
\end{defn}

An important homomorphism is the \emph{adjoint homomorphism}. If $L$ is a Lie algebra, it is defined as the map
\begin{align*}
	\ad : L \rightarrow \mathfrak{gl}(L)
\end{align*}
such that $(\ad_x)(y) := [x,y]$ for $x$, $y \in L$.

\begin{defn}[Lie Subalgebra] Given a Lie algebra $L$, a \emph{Lie subalgebra} of $L$ is a vector subspace $K \subseteq L$ such that $[x,y] \in K$ for all $x, y \in K$.
\end{defn}

\begin{rmk} Lie subalgebras are easily seen to be Lie algebras in their own right. 
\end{rmk} 

\begin{defn}[Ideal] An \emph{ideal} of a Lie algebra $L$ is a subspace $I$ of $L$ such that $[x,y] \in I$ for all $x \in L$ and $y \in I$.
\end{defn}

An important example of an ideal of a Lie algebra $L$ is the \emph{center} of $L$, defined by \[ \mathcal{Z}(L) \stackrel{\text{def}}{=} \{x \in L | [x,y] = 0 \text{ for all } y \in L\}. \]

\begin{defn}[Closed Ideal] A (non-trivial) ideal $I$ of a Lie algebra $L$ is called \emph{closed} (in the sense of Date and Roan) if $Z(I) \stackrel{\text{def}}{=} \{x \in L | [x,L] \subseteq I\} = I$ or equivalently $L/I$ has a trivial center.
\end{defn}

\begin{rmk} Note that $I$ is always included in $Z(I)$. \end{rmk}

\begin{defn}[Derived Series / Solvable] The \emph{derived series} of a Lie algebra $L$ is the series of ideals
\begin{align*}
	L^{(1)} = [L,L] \quad \text{and} \quad L^{(k)} = [L^{(k-1)},L^{(k-1)}] \quad \text{for } k \geq 2.
\end{align*}
The Lie algebra $L$ is said to be \emph{solvable} if for some $m \geq 1$ we have $L^{(m)} = 0$.
\end{defn}

\begin{defn}[Lower Central Series / Nilpotent] The \emph{lower central series} of a Lie algebra $L$ is the series of ideals 
\begin{align*}
	L^{1} = [L,L] \quad \text{and} \quad L^{k} = [L,L^{k-1}] \quad \text{for } k \geq 2.
\end{align*}
The Lie algebra $L$ is said to be \emph{nilpotent} if for some $m \geq 1$ we have $L^{m} = 0$.
\end{defn}

\section{Timeline}
\label{sec:Timeline}

The Onsager algebra appeared for the first time in Onsager's solution of the two-dimensional Ising model in a zero magnetic field \cite{OnsagerL44}, which is a landmark in the theory of exactly solved models in statistical mechanics. In \cite{OnsagerL44}, Onsager determined that the so-called transfer matrix associated to this model was proportional to $\exp(A_0) \exp(A_1)$, where $A_0$ and $A_1$ are non-commuting matrices. By analyzing the structure of the algebra generated by $A_0$ and $A_1$ in detail, Onsager derived a Lie algebra (denoted $\mathcal{O}$) with basis $\{A_m, G_l ~ | ~ m \in \mathbb{Z}, l \in \mathbb{N}_+\}$ and antisymmetric product given by:
\begin{align}
	[A_l,A_m] &= 2G_{l-m}, &l > m,\label{Onsrel1} \\
  [G_l,A_m] &= A_{m+l} - A_{m-l},  \label{Onsrel2} \\
  [G_l,G_m] &= 0 \label{Onsrel3}.
\end{align}
Note that the fact that this is a Lie algebra is not obvious. In Lemma \ref{OisLie}, we show that this is indeed the case.

\begin{rmk} In the above definition, the notation is slightly nonstandard. The elements $A_m$, $G_l$ in our definition correspond to the elements $\frac{A_m}{2}$, $\frac{G_l}{2}$ in most references, including \cite{OnsagerL44}. This adjustment is made for notational purposes. 
\end{rmk}

Onsager's work on the two-dimensional Ising model was soon after simplified by using Clifford algebras \cite{KaufmanBOnsagerL49} and fermion techniques \cite{SchultzTMattisDLiebE64}. These new methods of solving the two-dimensional Ising model did not involve the Onsager algebra. This caused the Onsager algebra to lose prominence in the literature until the early 1980s, at which point it began to re-emerge in various papers but in a not so apparent manner. Namely, in Dolan and Grady's 1982 paper \cite{DolanLGradyM82} on the role of self-duality in producing integrability, the Onsager algebra appeared implicitly. The essence of the main result of \cite{DolanLGradyM82} was that when we have a pair of dual operators $A$, $B$ satisfying what we will refer to as the \emph{Dolan-Grady relation} \[ [A,[A,[A,B]]] = 4[A,B],\] then there exists an infinite set of conserved commuting self-dual charges. Perk made the connection between the Dolan-Grady algebra (generated by $A$ and $B$) and the Onsager algebra when he considered the chiral Potts model in \cite{PerkJ89}, which contains the two-dimensional Ising model as a special case. He observed that Dolan and Grady's construction was identical with that of Onsager, but in a different notation. Davies then, after receiving some inspiration from the work of von Gehlen and Rittenberg \cite{GehlenGRittenbergV85}, aimed to prove that the Onsager algebra and the related Dolan-Grady relation did not need self-duality. He managed to do so in \cite{DaviesB91} by taking the one Dolan-Grady relation \[[A,[A,[A,B]]] = 4[A,B]\]
and replacing it by the two relations \[[A,[A,[A,B]]] = 4[A,B], \quad [B,[B,[B,A]]] = 4[B,A]. \]
Doing this eliminated the need for the self-duality condition. He then gave a mathematical treatment of the Onsager algebra using only these two relations. In the same year, Roan took it upon himself to attempt to add rigour to Davies' arguments from the mathematical point of view. In the first section of his paper \cite{RoanS91}, Roan followed Davies' arguments to derive the important relations between generators of the Onsager algebra from the Dolan-Grady relations. In Chapter 2 of this thesis, we will follow their procedure and add some additional arguments in order to demonstrate that the Onsager algebra and the Dolan-Grady algebra are in fact isomorphic. Following this result, the literature uses the two realizations interchangeably as the definition of the Onsager algebra. 

We then fast forward to Date and Roan's work \cite{RoanSDateE00, DateERoanS00} in which they studied the Onsager algebra from the ideal theoretic point of view. They managed to obtain a complete classification of closed ideals and the structure of quotient algebras. They did so by first showing that the Onsager algebra can be seen as a subalgebra of an $\mathfrak{sl}_2$-loop algebra fixed by the Chevalley involution. In Chapter 3 of this thesis, we will imitate their procedure to show that the Onsager algebra can be described as an equivariant map algebra, a realization which is then used to classify its closed ideals. Another motivation for looking at the Onsager algebra as an equivariant map algebra is examined in \cite{NeherESavageASenesiP09}. In this paper, Neher, Savage and Senesi provided a complete classification of the irreducible finite-dimensional representations of an arbitrary equivariant map algebra. They applied their results to the Onsager algebra to recover a classification of its irreducible finite-dimensional representations. Further motivation for looking at the Onsager algebra as an equivariant map algebra is that one can attempt to describe its indecomposable representations in hopes of potentially generalizing the result to an arbitrary equivariant map algebra. 

In 2007, Hartwig and Terwilliger continued to investigate the relationship between the Onsager algebra and $\mathfrak{sl}_2$-loop algebras in their joint work \cite{HartwigBTerwilligerP07}, except they worked with a three-point $\mathfrak{sl}_2$-loop algebra. Their main result showed that this three-point loop algebra is isomorphic to the tetrahedron algebra, $\boxtimes$, hence giving it a presentation via generators and relations. Their second main result was that $\boxtimes$ can be decomposed into a direct sum of three subalgebras (not ideals), each of which turn out to be isomorphic to the Onsager algebra. In the same year, Elduque began his paper \cite{ElduqueA07} by reviewing their work and presenting it in a simplified way. The rest of Elduque's paper consisted of solving some of the problems suggested at the end of \cite{HartwigBTerwilligerP07}, one of which was to describe the ideals of the tetrahedron algebra. Elduque managed to describe these explicitly along with the ideals of the Onsager algebra. In the fourth chapter of this thesis, we will add some arguments to the works of Hartwig, Terwiliger and Elduque in order to present a complete proof of the fact that the Onsager algebra is isomorphic to a subalgebra of the tetrahedron algebra. Using this realization of the Onsager algebra, we then proceed to explicitly describing all its ideals. 
\cleardoublepage

\chapter{Dolan-Grady Relations}
\label{sec:DolanAndGradyRelations}

\section{The Onsager Algebra}
\label{sec:TheOnsagerAlgebra}

We begin this chapter by recalling the original definition of the Onsager algebra.

\begin{defn}[The Onsager algebra] \label{OnsDefOrig} The Onsager algebra, denoted $\mathcal{O}$, is the (nonassociative) algebra over $k$ with a basis $\{A_m, G_l ~ | ~ m \in \mathbb{Z}, l \in \mathbb{N}_+\}$ and antisymmetric product given by:
\begin{align*}
	[A_l,A_m] &= 2G_{l-m}, &l > m, \tag{\ref{Onsrel1}} \\
  [G_l,A_m] &= A_{m+l} - A_{m-l},  \tag{\ref{Onsrel2}} \\
  [G_l,G_m] &= 0. \tag{\ref{Onsrel3}}
\end{align*}
\end{defn}

\begin{rmk} \label{rmkG-m} Note that it follows from (\ref{Onsrel1}) that $G_{-m} = -G_m$ ($m \in \mathbb{Z}$).\end{rmk}

\begin{lem} \label{OisLie} The Onsager algebra is a Lie algebra.
\end{lem}

\begin{proof} All there is to verify here is that our product indeed satisfies the Jacobi identity \eqref{JacobiId}. We have four cases to consider, the first being trivial: \[[G_l,[G_m,G_n]] + [G_m,[G_n,G_l]] + [G_n,[G_l,G_m]] = 0 + 0 + 0 = 0,\] for any $l,m,n \in \mathbb{N}_+$. Next we consider the three remaining cases.

\begin{align*}
&[A_l,[A_m,A_n]] + [A_m,[A_n,A_l]] + [A_n,[A_l,A_m]] \\
&= 2[A_l,G_{m-n}] + 2[A_m,G_{n-l}] + 2[A_n,G_{l-m}] \tag*{by (\ref{Onsrel1})}\\
&= 2A_{l-m+n} - 2A_{l+m-n} + 2A_{m-n+l} - 2A_{m+n-l} + 2A_{n-l+m} - 2A_{n+l-m} \tag*{by (\ref{Onsrel2})} \\
&= 0,
\end{align*}
for any $l,m,n \in \mathbb{Z}$. Also,

\begin{align*}
&[A_l,[A_m,G_n]] + [A_m,[G_n,A_l]] + [G_n,[A_l,A_m]] \\
&= 2[A_l,A_{m-n}-A_{m+n}] + 2[A_m,A_{n+l}-A_{l-n}] + 2[G_n,G_{l-m}] \tag*{by (\ref{Onsrel1}), (\ref{Onsrel2})}\\
&= 2G_{l-m+n} - 2G_{l-m-n} + 2G_{m-n-l} - 2G_{m-l+n} \tag*{by (\ref{Onsrel1}), (\ref{Onsrel3})} \\
&= 2G_{l-m+n} - 2G_{l-m-n} + 2G_{l-m-n} - 2G_{l-m+n} \tag*{by Remark \ref{rmkG-m}} \\
&= 0,
\end{align*}
for any $l,m \in \mathbb{Z}$ and $n \in \mathbb{N}_+$. Finally,

\begin{align*}
\quad &[A_l,[G_m,G_n]] + [G_m,[G_n,A_l]] + [G_n,[A_l,G_m]]\\
={}&0 + [G_m,A_{n+l}-A_{l-n}] + [G_n,A_{l-m}-A_{m+l}] \tag*{by (\ref{Onsrel1}), (\ref{Onsrel3})}\\
={}&A_{m+n+l} - A_{n+l-m} + A_{l-n-m} - A_{m+l-n} \\
	&+ A_{n+l-m} - A_{l-n-m} + A_{m+l-n} - A_{m+n+l} \tag*{by (\ref{Onsrel2})} \\
={}&0, 
\end{align*}
for any $l \in \mathbb{Z}$ and $m,n \in \mathbb{N}_+$.
\end{proof}

\section{The Dolan-Grady Algebra}
\label{sec:DolanAndGradySAlgebra}

As discussed in the introduction, a presentation of the Onsager algebra as a Lie algebra generated by two generators and two relations came to be implicitly in the work of Dolan and Grady. In this chapter, we will consider their algebra defined as follows.

\begin{defn}[The Dolan-Grady algebra] \label{DGalgDef} Let $\mathcal{OA}$ be the Lie algebra over $k$ with generators $A$, $B$ and relations
\begin{align}
    [A,[A,[A,B]]] = 4[A,B], \label{DG1} \\
    [B,[B,[B,A]]] = 4[B,A]. \label{DG2}
\end{align}
Relations \eqref{DG1} and \eqref{DG2} will be referred to as the \emph{Dolan-Grady relations}.
\end{defn}

The goal is to prove the following proposition.

\begin{prop} \label{IsomDG} The Onsager algebra, $\mathcal{O}$, is isomorphic to the Dolan-Grady algebra, $\mathcal{OA}$. 
\end{prop}

In order to have the tools to prove this, a set of elements in $\mathcal{OA}$ satisfying Onsager's relations is needed.

Let $A_0 = A$ and $A_1 = B$. Let $G_1$ be the commutator
\begin{align}
    G_1 = \frac{1}{2}[A_1,A_0] \label{G1Def}
\end{align}
and define $A_m$, $G_m$ ($m \in \mathbb{Z}$) in $\mathcal{OA}$ by
\begin{align}
    A_{m+1} - A_{m-1} &= [G_1, A_m], \label{Am+1Def} \\
    G_m &= \frac{1}{2}[A_m,A_0]. \label{GmDef}
\end{align}
In particular, we note that $G_0 = \frac{1}{2}[A_0,A_0]= 0$.  

\begin{theo} \emph{(\cite[Theorem 1]{RoanS91})} \label{Onsrelthm} The $A_m$ and $G_m$ as defined above satisfy Onsager's relations; i.e.,
\begin{align*}
    [A_l,A_m] &= 2G_{l-m}, \\
    [G_l,A_m] &= A_{m+l} - A_{m-l}, \\
    [G_l,G_m] &= 0 .
\end{align*}
\end{theo}

\begin{rmk} We have used the same notation, i.e.,$A_m$ and $G_l$, as in the definition of the Onsager algebra. This follows the standard practice in the literature. It will be clear from the context in which algebra we are working. In particular, in Section \ref{sec:DolanAndGradySAlgebraAndOnsagerSRelations} and Section \ref{sec:AlternatePresentationOfTheOnsagerAlgebra} we will always be working in the Dolan-Grady algebra. Of course, after Proposition \ref{IsomDG} has been proven, the distinction between $A_m \in \mathcal{O}$ and $A_m \in \mathcal{OA}$ becomes irrelevant. \end{rmk}

The proof of Theorem \ref{Onsrelthm} consists of a series of lemmas and is quite tedious. Hence, Section \ref{sec:DolanAndGradySAlgebraAndOnsagerSRelations} is dedicated to its proof. Once it is proven, the theorem can then be utilized to prove Proposition \ref{IsomDG} in the last section of this chapter.

\section{The Dolan-Grady Algebra and Onsager's Relations}
\label{sec:DolanAndGradySAlgebraAndOnsagerSRelations}

Theorem \ref{Onsrelthm} is proven using a series of lemmas. One begins by showing that (\ref{Onsrel1}) implies
(\ref{Onsrel2}) and (\ref{Onsrel3}). It then only remains to show that (\ref{Onsrel1}) is
satisfied. To do so, we aim to prove that the commutators $[A_l,A_m]$ depend only on the difference
$l-m$ of the indices. Having this, along with (\ref{GmDef}), we can conclude (\ref{Onsrel1}).

\begin{prop} \label{addRelations} Relation (\ref{Onsrel1}) implies the following relations:
\begin{align}
    A_{n+1} - A_{-n-1} &= -\frac{1}{2}[A_0,[A_1,A_{-n}]], & n \geq 0, \label{addRel1} \\
    A_{n+1} - A_{-n+1} &= \frac{1}{2}[A_1,[A_0,A_n]], & n \geq 1. \label{addRel2}
\end{align}
\end{prop}

In order to prove this, one requires the following lemma:
\begin{lem} \emph{(\cite[Lemma 2]{RoanS91})} \label{addLemma} Let $n \geq 1$.
\begin{itemize}
    \item[(i)] If $[A_0,A_n] = [A_{-n},A_0]$, then (\ref{addRel1})$_n \Leftrightarrow$ (\ref{addRel2})$_n$. \\
    \item[(ii)] If $[A_1,A_{n+1}] = [A_{-n+1},A_1]$, then (\ref{addRel1})$_{n-1} \Leftrightarrow$ (\ref{addRel2})$_{n+1}$.
\end{itemize}
\end{lem}

\noindent \textbf{Notation:} Note that the notation (\ref{addRel1})$_m$ above refers to the relation (\ref{addRel1}) for $n = m$ (and similarly for (\ref{addRel2})$_m$).

\begin{proof} Let $n \geq 1$.

(i) Assume $[A_0,A_n] = [A_{-n},A_0]$ and apply $\frac{1}{2}\ad_{A_1}$ to get
\[\frac{1}{2}[A_1,[A_0,A_n]] = \frac{1}{2}[A_1,[A_{-n},A_0]].\]
By the Jacobi identity, we have that
\[ [A_1,[A_{-n},A_0]] = [A_{-n},[A_1,A_0]] - [A_0,[A_1,A_{-n}]].\]
It then follows that
\begin{align*}
	\frac{1}{2}[A_1,[A_0,A_n]]	&= \frac{1}{2}[A_{-n},[A_1,A_0]] - \frac{1}{2}[A_0,[A_1,A_{-n}]] \\
                              &= [A_{-n},G_1] - \frac{1}{2}[A_0,[A_1,A_{-n}]] \tag*{by (\ref{G1Def})} \\
                              &= -(A_{-n+1} - A_{-n-1}) - \frac{1}{2}[A_0,[A_1,A_{-n}]] \tag*{by (\ref{Am+1Def})}.
\end{align*}
So we have that
\begin{align}
    \frac{1}{2}[A_1,[A_0,A_n]] + A_{-n+1} = A_{-n-1} - \frac{1}{2}[A_0,[A_1,A_{-n}]]. \label{LHS=RHS1}
\end{align}
Since
\begin{align*}
    (\ref{addRel1})_n   &\Longleftrightarrow A_{n+1} = \text{RHS of (\ref{LHS=RHS1})} \\ 
                        &\Longleftrightarrow A_{n+1} = \text{LHS of (\ref{LHS=RHS1})} \\
                        &\Longleftrightarrow (\ref{addRel2})_n,
\end{align*}
it allows us to conclude that (\ref{addRel1})$_n \Leftrightarrow$ (\ref{addRel2})$_n$.

(ii) Assume $[A_1,A_{n+1}] = [A_{-n+1},A_1]$ and apply $\frac{1}{2}\ad_{A_0}$ to get
\begin{align*}
    \frac{1}{2}[A_0,[A_{-n+1},A_1]] &= \frac{1}{2}[A_0,[A_1,A_{n+1}]] \\
    \Longleftrightarrow ~ -\frac{1}{2}[A_0,[A_1,A_{-n+1}]] &= \frac{1}{2}[A_0,[A_1,A_{n+1}]].
\end{align*}
By the Jacobi identity, we have that
\[ [A_0,[A_1,A_{n+1}]] = [A_1,[A_0,A_{n+1}]] + [A_{n+1},[A_1,A_0]]. \]
It then follows that
\begin{align*}
-\frac{1}{2}[A_0,[A_1,A_{-n+1}]] &= \frac{1}{2}[A_1,[A_0,A_{n+1}]] + \frac{1}{2}[A_{n+1},[A_1,A_0]] \\
                          	&= \frac{1}{2}[A_1,[A_0,A_{n+1}]] + [A_{n+1},G_1] \tag*{by (\ref{G1Def})} \\
                            &= \frac{1}{2}[A_1,[A_0,A_{n+1}]] - (A_{n+2} - A_n) \tag*{by (\ref{Am+1Def})}.
\end{align*}
So we have that
\begin{align}
    A_n + \frac{1}{2}[A_0,[A_1,A_{-n+1}]] = A_{n+2} - \frac{1}{2}[A_1,[A_0,A_{n+1}]]. \label{LHS=RHS2}
\end{align}
Next we recall that
\begin{align*}
    (\ref{addRel1})_{n-1}: \quad & A_n - A_{-n} = -\frac{1}{2}[A_0,[A_1,A_{-n+1}]], \\
    (\ref{addRel2})_{n+1}: \quad & A_{n+2} - A_{-n} = \frac{1}{2}[A_1,[A_0,A_{n+1}]].
\end{align*}
Hence it follows that
\begin{align*}
	(\ref{addRel1})_{n-1}	&\Longleftrightarrow A_{-n} = \text{LHS of (\ref{LHS=RHS2})} \\
                        &\Longleftrightarrow A_{-n} = \text{RHS of (\ref{LHS=RHS2})} \\
                        &\Longleftrightarrow (\ref{addRel2})_{n+1}.
\end{align*}
This then allows us to conclude that (\ref{addRel1})$_{n-1} \Leftrightarrow$ (\ref{addRel2})$_{n+1}$.
\end{proof}

\begin{proof}[Proof of Proposition \ref{addRelations}] First we note that 
\begin{align*}
	[A_0,A_n] &= 2G_{-n} \quad \text{by (\ref{Onsrel1})}\\
	[A_{-n},A_0] &= 2G_{-n} \quad \text{by (\ref{Onsrel1})}\\
	\implies [A_0,A_n] &= [A_{-n},A_0].
\end{align*}
It then follows from Lemma \ref{addLemma}(i) that (\ref{addRel1})$_n \Leftrightarrow$ (\ref{addRel2})$_n$. Furthermore, it also follows that (\ref{LHS=RHS1}) holds. Similarly, 
\begin{align*}
	[A_1,A_{n+1}] &= 2G_{-n} \quad \text{by (\ref{Onsrel1})}\\
	[A_{-n+1},A_1] &= 2G_{-n} \quad \text{by (\ref{Onsrel1})}\\
	\implies [A_1,A_{n+1}] &= [A_{-n+1},A_1].
\end{align*}
So it follows that (\ref{LHS=RHS2}) holds.

Now we can use all of this to prove our proposition. From above, we know that it is enough to prove (\ref{addRel1})$_n$ and then (\ref{addRel2})$_n$ will follow. To do so we proceed by induction.

\textit{Base Cases:} The following shows (\ref{addRel1})$_0$:
\begin{align*}
    A_1 - A_{-1} &= [G_1,A_0] \tag*{by (\ref{Am+1Def})} \\
                 &= \frac{1}{2}[[A_1,A_0],A_0]  \tag*{by (\ref{G1Def})}\\
                 &= -\frac{1}{2}[A_0,[A_1,A_0]].
\end{align*}
\noindent When it comes to showing (\ref{addRel1})$_1$, we note that it is not only equivalent but also simpler to prove (\ref{addRel2})$_1$.
\begin{align*}
    A_2 - A_{0} &= [G_1,A_1] \tag*{by (\ref{Am+1Def})} \\
                &= \frac{1}{2}[[A_1,A_0],A_1]  \tag*{by (\ref{G1Def})}\\
                &= \frac{1}{2}[A_1,[A_0,A_1]].
\end{align*}
\noindent \textit{Induction Hypothesis:} Assume (\ref{addRel1})$_m$ holds for $0 \leq m \leq n-1$.

\noindent \textit{Inductive Step:} Prove for $m = n$. From (\ref{LHS=RHS1}) and (\ref{LHS=RHS2}), it follows that
\begin{align*}
	&A_{-n-1} - A_{-n+1} - \frac{1}{2}[A_0,[A_1,A_{-n}]] = A_{n+1} - A_{n-1} - \frac{1}{2}[A_0,[A_1,A_{-n+2}]] \\
	&\Leftrightarrow -\frac{1}{2}[A_0,[A_1,A_{-n}]] = A_{n+1} - A_{-n-1} + A_{-n+1} - A_{n-1} - \frac{1}{2}[A_0,[A_1,A_{-(n-2)}]]\\
	&\Leftrightarrow -\frac{1}{2}[A_0,[A_1,A_{-n}]] = A_{n+1} - A_{-n-1} + A_{-n+1} - A_{n-1} + A_{n-1} - A_{-n+1}\\
	&\Leftrightarrow -\frac{1}{2}[A_0,[A_1,A_{-n}]] = A_{n+1} - A_{-n-1} \\
	&\Leftrightarrow	(\ref{addRel1})_n.													
\end{align*}
Observe that the $n-2$ case of the induction hypothesis was used above.
\end{proof}

\begin{lem} \emph{(\cite[Lemma 3]{RoanS91})} \label{Ons1imp23} Relation (\ref{Onsrel1}) implies relations (\ref{Onsrel2}) and (\ref{Onsrel3}).
\end{lem}

\begin{proof} We begin by showing that (\ref{Onsrel1}) implies (\ref{Onsrel2}). Recall that Proposition \ref{addRelations} tells us that (\ref{Onsrel1}) implies (\ref{addRel1}) and (\ref{addRel2}). The following shows that (\ref{Onsrel2}) is satisfied for $m=0$:
\begin{align*}
    [G_l, A_0]  &= \frac{1}{2} [[A_1,A_{1-l}],A_0]  \tag*{by (\ref{Onsrel1})} \quad \\
                &= - \frac{1}{2} [A_0,[A_1,A_{-(l-1)}]] \\
                &= A_l - A_{-l}  \tag*{by (\ref{addRel1})$_{l-1}$.}
\end{align*}
The following shows that (\ref{Onsrel2}) is satisfied for $m=1$:
\begin{align*}
    [G_l, A_1]  &= \frac{1}{2} [[A_l,A_0],A_1]  \tag*{by (\ref{Onsrel1})} \\
                &= \frac{1}{2} [A_1,[A_0,A_l]] \\
                &= A_{1+l} - A_{1-l}  \tag*{by (\ref{addRel2})$_l$.}
\end{align*}
From (\ref{Onsrel1}) we have that $[A_l,A_0] = 2G_l$. We then apply $\ad_{G_m}$ to get:
\begin{align}
	2[G_m,G_l] = [G_m,[A_l,A_0]].
\end{align} 
Furthermore, by the Jacobi identity, we have the following
\begin{align}
	2[G_m,G_l] = [[G_m,A_l],A_0] + [A_l,[G_m,A_0]]. \label{GmGl}
\end{align}
For $m = 1$, we have
\begin{align*}
	2[G_1,G_l] 	&= [[G_1,A_l],A_0] + [A_l,[G_1,A_0]] \\
							&= [A_{l+1} - A_{l-1},A_0] + [A_l,A_1 - A_{-1}]  \tag*{by (\ref{Am+1Def})} \\
							&= 2G_{l+1} - 2G_{l-1} + 2G_{l-1} - 2G_{l+1}  \tag*{by (\ref{Onsrel1})} \\
							&= 0.
\end{align*}
Hence showing that (\ref{Onsrel3}) is satisfied for $l = 1$ or $m = 1$. Next, we note that for any $m \in \mathbb{Z}$, it follows from (\ref{Am+1Def}) that \[ A_{m+2} = A_m + \ad_{G_1}A_{m+1}. \] We claim that  
\begin{align}
	A_{m+l} = \big(g_l(\ad_{G_1})\big)(A_m) + \big(h_l(\ad_{G_1})\big)(A_{m+1}) &\qquad \text{for $l \geq 0$, $m \in \mathbb{Z}$} \label{Am+lDef},
\end{align}
where 
\begin{align*}
g_l(x) &= \begin{cases} 1 & \text{if $l=0$,}\\ 0 & \text{if $l=1$,}\\ xg_{l-1}(x)+g_{l-2}(x) & \text{if $l \geq 2$,} \end{cases} &h_l(x) &= \begin{cases} 0 & \text{if $l=0$,}\\ 1 & \text{if $l=1$,}\\ xh_{l-1}(x)+h_{l-2}(x) & \text{if $l \geq 2$.} \end{cases}
\end{align*}
To prove this claim, we proceed by induction on $l \in \mathbb{N}$. First, we consider the base cases:
\begin{align*}
&l=0,&	A_m			&= \big(g_0(\ad_{G_1})\big)(A_m) + \big(h_0(\ad_{G_1})\big)(A_{m+1}), \\
&l=1,&	A_{m+1}	&= \big(g_1(\ad_{G_1})\big)(A_m) + \big(h_1(\ad_{G_1})\big)(A_{m+1}).
\end{align*}
Assume (\ref{Am+lDef}) holds for $l \geq 2$. We proceed to show it for $l+1$.
\begin{align*}
A_{m+l+1} ={}& A_{m+l-1} + [G_1,A_{m+l}] \tag*{by (\ref{Am+1Def})}\\
					={}& \big(g_{l-1}(\ad_{G_1})\big)(A_m) + \big(h_{l-1}(\ad_{G_1})\big)(A_{m+1}) \\
						&+ [G_1,\big(g_l(\ad_{G_1})\big)(A_m) + \big(h_l(\ad_{G_1})\big)(A_{m+1})]\\
					={}& \big((g_{l-1} + xg_l)(\ad_{G_1})\big)(A_m) + \big((h_{l-1} + xh_l)(\ad_{G_1})\big)(A_{m+1}) \\
					={}& \big(g_{l+1}(\ad_{G_1})\big)(A_m) + \big(h_{l+1}(\ad_{G_1})\big)(A_{m+1}).
\end{align*} 
Hence proving our claim, (\ref{Am+lDef}), which we can now use as follows:
\begin{align*}
[G_l,A_m] &= [G_l, \big(g_m(\ad_{G_1})\big)A_0 + \big(h_m(\ad_{G_1})\big)A_1] \\
					&= \big(g_m(\ad_{G_1})\big)[G_l,A_0] + \big(h_m(\ad_{G_1})\big)[G_l,A_1] \tag*{since $[G_1,G_l]=0$}\\
					&= \big(g_m(\ad_{G_1})\big)(A_l-A_{-l}) + \big(h_m(\ad_{G_1})\big)(A_{l+1}-A_{1-l}) \tag*{by above}\\
					&= \big(\big(g_m(\ad_{G_1})\big)A_l + \big(h_m(\ad_{G_1})\big)A_{l+1}\big) - \big(\big(g_m(\ad_{G_1})\big)A_{-l} + \big(h_m(\ad_{G_1})\big)A_{1-l})\\
					&= A_{l+m} - A_{l-m}.
\end{align*}
Hence, we have (\ref{Onsrel2}).

The next step is to prove that (\ref{Onsrel1}) and (\ref{Onsrel2}) imply (\ref{Onsrel3}). From (\ref{GmGl}) above, we have that
\begin{align*}
 2[G_m,G_l] &= [[G_m,A_l],A_0] + [A_l,[G_m,A_0]] \\
            &= [A_{m+l} - A_{l-m},A_0] + [A_l, A_m - A_{-m}] \tag*{by (\ref{Onsrel2})} \\
            &= 2G_{l+m} - 2G_{l-m} + 2G_{l-m} - 2G_{l+m} \tag*{by (\ref{Onsrel1})} \\
            &= 0.
\end{align*}
Hence, we have (\ref{Onsrel3}). So we have proven that
\begin{itemize}
    \item[(i)] (\ref{Onsrel1}) implies (\ref{Onsrel2}), and
    \item[(ii)] (\ref{Onsrel1}) and (\ref{Onsrel2}) imply (\ref{Onsrel3}).
\end{itemize}
So finally, we conclude that the relation (\ref{Onsrel1}) implies relations (\ref{Onsrel2}) and (\ref{Onsrel3}).
\end{proof}

\begin{lem} \emph{(\cite[Lemma 1]{DaviesB91})} $A_{-1}$, $A_0$, $A_1$, and $A_2$ defined by relations (\ref{Am+1Def}) and (\ref{GmDef}) satisfy
\begin{align}
    [A_{-1}, A_0] = [A_0,A_1] = [A_1,A_2] \label{addRel3}
\end{align}
if and only if $A_0$ and $A_1$ satisfy the Dolan-Grady relations (\ref{DG1}) and (\ref{DG2}) with $A_0 = A$ and $A_1 = B$.
\end{lem}

\begin{proof} We begin by assuming (\ref{addRel3}) and showing that $A_0$ and $A_1$ satisfy the relations (\ref{DG1}) and (\ref{DG2}) with $A_0 = A$ and $A_1 = B$.
\begin{align*}
    [A_0,A_{-1}]    &= [A_0,A_1 - [G_1,A_0]]  \tag*{by (\ref{Am+1Def})} \\
                    &= [A_0,A_1] + [A_0,[A_0,G_1]] \\
                    &= [A_0,A_1] - \frac{1}{2}[A_0,[A_0,[A_0,A_1]]]  \tag*{by (\ref{GmDef}).}
\end{align*}
Since $[A_0,A_{-1}] = -[A_0,A_1]$, it follows that $[A_0,[A_0,[A_0,A_1]]] = 4[A_0,A_1]$. Similarly,
\begin{align*}
    [A_1,A_2]   &= [A_1,A_0] + [A_1,[G_1,A_1]]  \tag*{by (\ref{Am+1Def})} \\
                &= [A_1,A_0] - \frac{1}{2}[A_1,[A_1,[A_1,A_0]]]  \tag*{by (\ref{GmDef}).}
\end{align*}
Since $[A_1,A_2] = [A_0,A_1]$, it follows that $[A_1,[A_1,[A_1,A_0]]] = 4[A_1,A_0]$. Next, we prove the other direction by assuming that $A_0$ and $A_1$ satisfy the relations (\ref{DG1}) and (\ref{DG2}) with $A_0 = A$ and $A_1 = B$.
\begin{align*}
    A_2 &= A_0 + [G_1,A_1]  \tag*{by (\ref{Am+1Def})} \\
        &= A_0 + \frac{1}{2}[[A_1,A_0],A_1]  \tag*{by (\ref{GmDef})} \\
        &= A_0 - \frac{1}{2}[A_1,[A_1,A_0]].
\end{align*}
By applying $\ad_{A_1}$ we get
\begin{align*}
    [A_1,A_2]   &= [A_1,A_0] - \frac{1}{2}[A_1,[A_1,[A_1,A_0]]] \\
                &= [A_1,A_0] - 2[A_1,A_0]  \tag*{by (\ref{DG2})} \\
                &= [A_0,A_1].
\end{align*}
Similarly, we have that
\begin{align*}
    A_{-1}  &= A_1 - [G_1,A_0] \tag*{by (\ref{Am+1Def})} \\
            &= A_1 - \frac{1}{2}[[A_1,A_0],A_0]  \tag*{by (\ref{GmDef})} \\
            &= A_1 - \frac{1}{2}[A_0,[A_0,A_1]].
\end{align*}
By applying $\ad_{A_0}$ we get
\begin{align*}
    [A_0,A_{-1}]	&= [A_0,A_1] - \frac{1}{2}[A_0,[A_0,[A_0,A_1]]] \\
                  &= [A_0,A_1] - 2[A_0,A_1] \\
                  &= [A_1,A_0].
\end{align*}
Equivalently, $[A_{-1},A_0] = [A_0,A_1]$.
\end{proof}

\begin{lem} \emph{(\cite[Lemma 4]{DaviesB91}, \cite[Lemma 4]{RoanS91})} If $A_0$ and $A_1$ satisfy the Dolan-Grady relations (\ref{DG1}) and (\ref{DG2}), then
\begin{align}
    [A_{-2},A_0] = [A_{-1},A_1] = [A_0,A_2] = [A_1,A_3]. \label{addRel4}
\end{align}
\end{lem}

\begin{proof} From (\ref{addRel3}) we have that \[ [A_{-1}, A_0] = [A_0,A_1] = [A_1,A_2]. \] Apply $\ad_{G_1}$ to each of the commutators to get
\begin{align*}
	[G_1,[A_{-1},A_0]] 	&= [[G_1,A_{-1}],A_0] + [A_{-1},[G_1,A_0]]  \tag*{by \eqref{JacobiId}} \\
											&= [A_0-A_{-2},A_0] + [A_{-1},A_1-A_{-1}]  \tag*{by (\ref{Am+1Def})} \\
											&= [A_0,A_{-2}] + [A_{-1},A_1], \\
											\\
	[G_1,[A_0,A_1]]	&= [[G_1,A_0],A_1] + [A_0,[G_1,A_1]]  \tag*{by \eqref{JacobiId}} \\
									&= [A_1-A_{-1},A_1] + [A_0,A_2-A_0]  \tag*{by (\ref{Am+1Def})} \\
									&= [A_1,A_{-1}] + [A_0,A_2], \\
									\\
	[G_1,[A_1,A_2]]	&= [[G_1,A_1],A_2] + [A_1,[G_1,A_2]]  \tag*{by \eqref{JacobiId}} \\
									&= [A_2-A_0,A_2] + [A_1,A_3-A_1]  \tag*{by (\ref{Am+1Def})} \\
									&= [A_2,A_0] + [A_1,A_3]. \\
\end{align*}
Hence it follows that
\begin{align}
	[A_0,A_{-2}] + [A_{-1},A_1] = [A_1,A_{-1}] + [A_0,A_2] = [A_2,A_0] + [A_1,A_3]. \label{3equality}
\end{align}
Next, we apply $\ad_{A_1}\ad_{A_0}$ on $[A_1,A_0]$ to get
\begin{align*}
	[A_1,[A_0,[A_1,A_0]]] &= [[A_1,A_0],[A_1,A_0]] + [A_0,[A_1,[A_1,A_0]]]  \tag*{by \eqref{JacobiId}} \\
	\Leftrightarrow 2[A_1,[A_0,G_1]] &= 2[A_0,[A_1,G_1]]  \tag*{by (\ref{G1Def})} \\
	\Leftrightarrow [A_1,[G_1,A_0]] &= [A_0,[G_1,A_1]] \\
	\Leftrightarrow [A_1,A_1-A_{-1}] &= [A_0,A_2-A_0]  \tag*{by (\ref{Am+1Def})} \\
	\Leftrightarrow [A_{-1},A_1] &= [A_0,A_2] \\
	\Leftrightarrow [A_1,A_{-1}] + [A_0,A_2] &= 0. \\
\end{align*}
From (\ref{3equality}) we then have that
\[ [A_0,A_{-2}] + [A_{-1},A_1] = [A_1,A_{-1}] + [A_0,A_2] = [A_2,A_0] + [A_1,A_3] = 0. \]
Hence we can conclude the desired result, \[[A_{-2},A_0] = [A_{-1},A_1] = [A_0,A_2] = [A_1,A_3].\]
\end{proof}

\begin{cor} \label{Rel1Rel2BC} If $A_0$ and $A_1$ satisfy the Dolan-Grady relations (\ref{DG1}) and (\ref{DG2}), then (\ref{addRel1})$_n$ and (\ref{addRel2})$_n$ are satisfied for $n \in \{0,1,2\}$.
\end{cor}

\begin{proof} We begin by showing that (\ref{addRel1})$_0$ is satisfied as follows:
\begin{align*}
	-\frac{1}{2}[A_0,[A_1,A_0]] &= -\frac{1}{2}[[A_0,A_1],A_0] 	\tag*{by \eqref{JacobiId}} \\
															&= [G_1,A_0] 										\tag*{by (\ref{G1Def})} \\
															&= A_1 - A_{-1}									\tag*{by (\ref{Am+1Def}).}
\end{align*}
Note that we clearly have $[A_0,A_0] = 0 = [A_0,A_0]$. Thus by Lemma \ref{addLemma}(i) we have that (\ref{addRel2})$_0$ holds. Similarly, since from (\ref{addRel3}) we have $[A_1,A_2] = [A_0,A_1]$, it follows from Lemma \ref{addLemma}(ii) that (\ref{addRel2})$_2$ is also satisfied. Furthermore, from (\ref{addRel4}) we have that $[A_0,A_2] = [A_{-2},A_0]$. Hence Lemma \ref{addLemma} tells us that (\ref{addRel1})$_2$ is satisfied. Next we show that (\ref{addRel2})$_1$ is satisfied:
\begin{align*}
	\frac{1}{2}[A_1,[A_0,A_1]]	&= \frac{1}{2}[[A_1,A_0],A_1]	\tag*{by \eqref{JacobiId}} \\
															&= [G_1,A_1]									\tag*{by (\ref{G1Def})} \\
															&= A_2 - A_0									\tag*{by (\ref{Am+1Def}).}
\end{align*}
Since (\ref{addRel3}) tells us that $[A_0,A_1] = [A_{-1},A_0]$, it follows from Lemma \ref{addLemma}(i) that (\ref{addRel1})$_1$ is also satisfied. 
\end{proof}

\begin{lem} \emph{(\cite[Lemma 4]{RoanS91}, \cite[Lemma 5]{DaviesB91})} If $A_0$ and $A_1$ satisfy the Dolan-Grady relations (\ref{DG1}) and (\ref{DG2}), then
\begin{align}
    [A_{-3},A_0] = [A_{-2},A_1] = [A_{-1},A_2] = [A_0,A_3] = [A_1,A_4]. \label{addRel5}
\end{align}
\end{lem}

\begin{proof} We begin by showing that $[A_{-2},A_1] = [A_{-1},A_2] = [A_0,A_3]$. To do so we will need to develop two equations by using (\ref{addRel4}). First of all, from (\ref{addRel4}) we have that $[A_{-1},A_1]=[A_0,A_2]$. To it we apply $\ad_{G_1}$ to get
\begin{align*}
&\quad [G_1,[A_{-1},A_1]] = [G_1,[A_0,A_2]] \\
&\Leftrightarrow [[G_1,A_{-1}],A_1] + [A_{-1},[G_1,A_1]] = [[G_1,A_0],A_2] + [A_0,[G_1,A_2]] \tag*{by \eqref{JacobiId}} \\
&\Leftrightarrow [A_0-A_{-2},A_1] + [A_{-1},A_2-A_0] = [A_1-A_{-1},A_2] + [A_0,A_3-A_1] \tag*{by (\ref{Am+1Def})}\\
&\Leftrightarrow 2[A_0,A_1] - [A_{-2},A_1] + 2[A_{-1},A_2] - [A_{-1},A_0] = [A_1,A_2] + [A_0,A_3] \\
&\Leftrightarrow 2[A_0,A_1] - [A_{-2},A_1] + 2[A_{-1},A_2] - [A_0,A_1] = [A_0,A_1] + [A_0,A_3] \tag*{by (\ref{addRel3}).}
\end{align*}
Which is equivalent to our first desired equation
\begin{align}
	2[A_{-1},A_2] = [A_{-2},A_1] + [A_0,A_3]. \label{equ1}
\end{align}
Next, consider the relation $[A_{-2},A_0]=[A_{-1},A_1]$ (from (\ref{addRel4})) and apply $\ad_{A_0}\ad_{A_1}$:
\begin{align}
	[A_0,[A_1,[A_{-2},A_0]]] = [A_0,[A_1,[A_{-1},A_1]]]. \label{adA0A1}
\end{align}
For computational simplicity let us consider each side of (\ref{adA0A1}) separately and then look at the equality. 
\begin{align*}
\text{LHS of }&\text{(\ref{adA0A1})}	= [A_0,[A_1,[A_{-2},A_0]]] \\
							&= [A_0,[[A_1,A_{-2}],A_0] + [A_{-2},[A_1,A_0]]] \tag*{by \eqref{JacobiId}}\\
							&= [A_0,-[A_0,[A_1,A_{-2}]] + 2[A_{-2},G_1]] \tag*{by (\ref{G1Def})}\\
							&= [A_0,2(A_3-A_{-3}) - 2[G_1,A_{-2}]] \tag*{by (\ref{addRel1})$_2$}\\
							&= [A_0,2A_3 - 2A_{-3} - 2(A_{-1}-A_{-3})] \tag*{by (\ref{Am+1Def})}\\
							&= 2[A_0,A_3-A_{-1}] \\
							&= 2[A_0,A_3] + 2[A_{-1},A_0] \\
							&= 2[A_0,A_3] + 2[A_0,A_1] \tag*{by (\ref{addRel3}),}
\end{align*}
\begin{align*}
\text{RHS of }&\text{(\ref{adA0A1})}	= [A_0,[A_1,[A_{-1},A_1]]] \\
							&= [A_0,[[A_1,A_{-1}],A_1]]  \tag*{by \eqref{JacobiId}}\\
							&= [[A_0,[A_1,A_{-1}]],A_1] + [[A_1,A_{-1}],[A_0,A_1]]  \tag*{by \eqref{JacobiId}}\\
							&= -2[A_2-A_{-2},A_1]+ 2[G_1,[A_1,A_{-1}]]  \tag*{by (\ref{addRel1})$_1$ and (\ref{G1Def})}\\
							&= 2[A_1,A_2] + 2[A_{-2},A_1] + 2[[G_1,A_1],A_{-1}] + 2[A_1,[G_1,A_{-1}]]  \tag*{by \eqref{JacobiId}}\\
							&= 2[A_1,A_2] + 2[A_{-2},A_1] + 2[A_2-A_0,A_{-1}] + 2[A_1,A_0-A_{-2}]  \tag*{by (\ref{Am+1Def})}\\
							&= 2[A_1,A_2] + 2[A_{-2},A_1] + 2[A_2,A_{-1}] + 2[A_{-1},A_0] + 2[A_1,A_0] + 2[A_{-2},A_1] \\
							&= 2[A_0,A_1] + 4[A_{-2},A_1] + 2[A_2,A_{-1}] + 2[A_0,A_1] - 2[A_0,A_1]  \tag*{by (\ref{addRel3})}\\
							&= 2[A_0,A_1] + 4[A_{-2},A_1] + 2[A_2,A_{-1}].
\end{align*}
Next we consider the equality of LHS and RHS of (\ref{adA0A1}) 
\[ 2[A_0,A_3] + 2[A_0,A_1] = 2[A_0,A_1] + 4[A_{-2},A_1] + 2[A_2,A_{-1}], \]
which is equivalent to our second desired equation
\begin{align}
	2[A_{-2},A_1] = [A_0,A_3] + [A_{-1},A_2]. \label{equ2}
\end{align}
Next we combine (\ref{equ1}) and (\ref{equ2}) by utilizing their common term $[A_0,A_3]$.
\begin{align*}
	2[A_{-1},A_2] - [A_{-2},A_1] &= 2[A_{-2},A_1] - [A_{-1},A_2], \\
	\Longleftrightarrow [A_{-1},A_2] &= [A_{-2},A_1].
\end{align*}
So we obtain one of the equalities of (\ref{addRel5}). Now if we replace this result in (\ref{equ1}) we get
\begin{gather*}
	2[A_{-2},A_1] = [A_{-2},A_1] + [A_0,A_3] \\
\Longleftrightarrow [A_{-2},A_1] = [A_0,A_3].
\end{gather*}
Hence, so far we have succeeded in proving that $[A_{-2},A_1] = [A_{-1},A_2] = [A_0,A_3]$.

Next, we aim to prove that $[A_0,A_3]=[A_{-3},A_0]$. To do so first we note that
\begin{align*}
[G_1,[A_{-1},&A_1]] = [[G_1,A_{-1}],A_1] + [A_{-1},[G_1,A_1]] \tag*{by \eqref{JacobiId}}\\
							&= [A_0-A_{-2},A_1] + [A_{-1},A_2-A_0] \tag*{by (\ref{Am+1Def})}\\
							&= [A_0,A_1] - [A_{-2},A_1] + [A_{-1},A_2] - [A_{-1},A_0] \\
							&= [A_0,A_1] - [A_{-2},A_1] + [A_{-2},A_1] - [A_0,A_1] \tag*{by (\ref{addRel3}) \& above.}
\end{align*}
So,
\begin{align}
	[G_1,[A_{-1},A_1]] = 0. \label{G1A-1A1}
\end{align}
Furthermore, we note that
\begin{align*}
[A_0&,[G_1,[A_0,[A_{-1},A_1]]]] \\
={}& [[A_0,G_1],[A_0,[A_{-1},A_1]]] + [G_1,[A_0,[A_0,[A_{-1},A_1]]]] \tag*{by \eqref{JacobiId}}\\
={}& [[[A_0,G_1],A_0],[A_{-1},A_1]] + [A_0,[[A_0,G_1],[A_{-1},A_1]]] \\
 & - [G_1,[A_0,[A_0,[A_1,A_{-1}]]]] \tag*{by \eqref{JacobiId}}\\
={}& -[[A_0,[A_0,G_1]],[A_{-1},A_1]] + [A_0,[[A_0,G_1],[A_{-1},A_1]]]\\
 & + 2[G_1,[A_0,A_2-A_{-2}]] \tag*{by (\ref{addRel1})$_1$}\\
={}& \frac{1}{2}[[A_0,[A_0,[A_0,A_1]]],[A_{-1},A_1]] + [A_0,[G_1,[A_0,[A_1,A_{-1}]]]] \\
 & + [A_0,[A_0,[G_1,[A_{-1},A_1]]]] + 2[G_1,[A_0,A_2] + [A_{-2},A_0]] \tag*{by (\ref{G1Def})}\\
={}& 2[[A_0,A_1],[A_{-1},A_1]] - [A_0,[G_1,[A_0,[A_{-1},A_1]]]] \\
 & + 4[G_1,[A_{-1},A_1]] \tag*{by (\ref{DG1}), (\ref{addRel4}), (\ref{G1A-1A1})}\\
={}& -4[G_1,[A_{-1},A_1]] - [A_0,[G_1,[A_0,[A_{-1},A_1]]]] \tag*{by (\ref{G1Def}),(\ref{G1A-1A1})}\\
={}& - [A_0,[G_1,[A_0,[A_{-1},A_1]]]] \tag*{by (\ref{G1A-1A1}).}
\end{align*}
Hence, we have that $[A_0,[G_1,[A_0,[A_{-1},A_1]]]] = 0$. But we note that
\begin{align*}
0 = [A_0,[G_1,[A_0,[A_{-1}&,A_1]]]]	= -[A_0,[G_1,[A_0,[A_1,A_{-1}]]]] \\
									&= 2[A_0,[G_1,A_2-A_{-2}]] \tag*{by (\ref{addRel1})$_1$}\\
									&= 2[A_0,[G_1,A_2]] - 2[A_0,[G_1,A_{-2}]] \\
									&= 2[A_0,A_3-A_1] - 2[A_0,A_{-1}-A_{-3}] \tag*{by (\ref{Am+1Def})}\\
									&= 2[A_0,A_3] - 2[A_0,A_1] + 2[A_{-1},A_0] - 2[A_{-3},A_0]\\
									&= 2[A_0,A_3] - 2[A_0,A_1] + 2[A_0,A_1] - 2[A_{-3},A_0] \tag*{by (\ref{addRel3})}\\
									&= 2[A_0,A_3] - 2[A_{-3},A_0].
\end{align*}
So we have that
\begin{align*}
	2[A_0,A_3] - 2[A_{-3},A_0] = 0 \\
	\Longleftrightarrow [A_0,A_3] = [A_{-3},A_0].
\end{align*}
To sum up, so far we have that $[A_{-3},A_0]=[A_{-2},A_1]=[A_{-1},A_2]=[A_0,A_3]$. Finally, we seek to prove that $[A_1,A_4]=[A_{-2},A_1]$ in order to complete the proof of our lemma. We proceed by a similar argument used above to get that $[A_0,A_3] = [A_{-3},A_0]$. So first we note that
\begin{align*}
[G_1,[A_0,A_2]]	&= [[G_1,A_0],A_2] + [A_0,[G_1,A_2]] \tag*{by \eqref{JacobiId}}\\
								&= [A_1-A_{-1},A_2] + [A_0,A_3-A_1] \tag*{by (\ref{Am+1Def})}\\
								&= [A_1,A_2] - [A_{-1},A_2] + [A_0,A_3] - [A_0,A_1]\\
								&= [A_0,A_1] - [A_0,A_3] + [A_0,A_3] - [A_0,A_1] \tag*{by (\ref{addRel3}) \& above.}
\end{align*}
So, 
\begin{align}
	[G_1,[A_0,A_2]] = 0. \label{G1A0A2}
\end{align}
Furthermore, we note that 
\begin{align*}
[A_1,&[G_1,[A_1,[A_0,A_2]]]] \\
={}& [[A_1,G_1],[A_1,[A_0,A_2]]]+[G_1,[A_1,[A_1,[A_0,A_2]]]] \tag*{by \eqref{JacobiId}}\\
={}& -[[A_1,[A_1,G_1]],[A_0,A_2]] + [A_1,[[A_1,G_1],[A_0,A_2]]] \\
 & + 2[G_1,[A_1,A_3-A_{-1}]] \tag*{by \eqref{JacobiId}, (\ref{addRel2})$_2$}\\
={}& -\frac{1}{2}[[A_1,[A_1,[A_1,A_0]]],[A_0,A_2]] + [A_1,[A_1,[G_1,[A_0,A_2]]]] \\
 & - [A_1,[G_1,[A_1,[A_0,A_2]]]] + 2[G_1,[A_1,A_3]-[A_1,A_{-1}]] \tag*{by \eqref{JacobiId}, (\ref{G1Def})}\\
={}& -2[[A_1,A_0],[A_0,A_2]] - [A_1,[G_1,[A_1,[A_0,A_2]]]] \\
 & + 2[G_1,[A_0,A_2] + [A_0,A_2]] \tag*{by (\ref{DG2}), (\ref{G1A0A2}), (\ref{addRel4})}\\
={}& -[G_1,[A_0,A_2]] - [A_1,[G_1,[A_1,[A_0,A_2]]]] + 4[G_1,[A_0,A_2]] \tag*{by (\ref{G1Def})}\\
={}& -[A_1,[G_1,[A_1,[A_0,A_2]]]] \tag*{by (\ref{G1A0A2}).}
\end{align*}
Hence, we have that $[A_1,[G_1,[A_1,[A_0,A_2]]]] = 0$. But we note that
\begin{align*}
0 = [A_1,[G_1,[A_1,[A_0,A_2]]]] &= 2[A_1,[G_1,A_3-A_{-1}]] \tag*{by (\ref{addRel2})$_2$}\\
								&= 2[A_1,[G_1,A_3]] - 2[A_1,[G_1,A_{-1}]] \\
								&= 2[A_1,A_4-A_2] - 2[A_1,A_0-A_{-2}] \tag*{by (\ref{Am+1Def})}\\
								&= 2[A_1,A_4] - 2[A_1,A_2] - 2[A_1,A_0] + 2[A_1,A_{-2}]\\
								&= 2[A_1,A_4] + 2[A_1,A_{-2}] \tag*{by (\ref{addRel3}).}
\end{align*}
So we have that
\begin{align*}
	2[A_1,A_4] + 2[A_1,A_{-2}] = 0 \\
	\Longleftrightarrow [A_1,A_4] = [A_{-2},A_1].
\end{align*}
This completes the proof of our lemma.
\end{proof}

\begin{lem} \emph{(\cite[Lemma 5]{RoanS91} \cite[Theorem 1]{DaviesB91})} If $A_0$ and $A_1$ satisfy the Dolan-Grady relations (\ref{DG1}) and (\ref{DG2}), then every adjacent pair $A_m$, $A_{m+1}$ satisfies the relations
\begin{align}
    [A_m,[A_m,[A_m,A_{m+1}]]] = 4[A_m,A_{m+1}], \label{Onsager1adj} \\
    [A_{m+1},[A_{m+1},[A_{m+1},A_m]]] = 4[A_{m+1},A_m], \label{Onsager2adj}
\end{align}
and $[A_{m+1},A_m] = 2G_1$ for all $m$ (in particular, the commutators $[A_{m+1},A_m]$ are independent of $m$).
\end{lem}

\begin{proof} We proceed by induction on $m$. We first consider the base cases $m \in \{-1,0,1\}$. The case $m = 0$ is simply the assumed Dolan-Grady relations (\ref{DG1}) and (\ref{DG2}). Next, consider the case $m = 1$. From (\ref{addRel4}), we have that $[A_0,A_2] = [A_1,A_3]$. Apply $\ad_{G_1}$ to get that
\begin{align*}
    &[G_1,[A_0,A_2]] = [G_1,[A_1,A_3]] \\
    &\Leftrightarrow [[G_1,A_0],A_2] + [A_0,[G_1,A_2]] = [[G_1,A_1],A_3] + [A_1,[G_1,A_3]] \tag*{by \eqref{JacobiId}}\\
    &\Leftrightarrow [A_1 - A_{-1},A_2] + [A_0,A_3 - A_1] = [A_2 - A_0,A_3] + [A_1,A_4 - A_2] \tag*{by (\ref{Am+1Def})}\\
    &\Leftrightarrow 2[A_1,A_2] - [A_{-1},A_2] + 2[A_0,A_3] - [A_0,A_1] = [A_2,A_3] + [A_1,A_4] \\
    &\Leftrightarrow 2[A_0,A_1] + 2[A_0,A_3] - [A_0,A_1] = [A_2,A_3] + 2[A_0,A_3] \tag*{by (\ref{addRel3}), (\ref{addRel5})}\\
    &\Leftrightarrow [A_0,A_1] = [A_2,A_3].
\end{align*}
Hence we have the additional relation
\begin{align}
	[A_0,A_1] = [A_2,A_3] \label{addRel6}
\end{align}
with which we now have the tools needed to show that (\ref{Onsager1adj}) and (\ref{Onsager2adj}) are satisfied for $m = 1$.
\begin{align*}
    [A_1,[A_1,[A_1,A_2]]]   &= [A_1,[A_1,[A_0,A_1]]] \tag*{by (\ref{addRel3})} \\
                            &= -[A_1,[A_1,[A_1,A_0]]] \\
                            &= -4[A_1,A_0]  \tag*{by (\ref{DG2})} \\
                            &= 4[A_0,A_1] \\
                            &= 4[A_1,A_2]  \tag*{by (\ref{addRel3}),}\\
                            \\
    [A_2,[A_2,[A_2,A_1]]]   &= [A_2,[A_2,[A_1,A_0]]] \tag*{by (\ref{addRel3})} \\
                            &= \big[A_2,[[A_2,A_1],A_0] + [A_1,[A_2,A_0]]\big] \tag*{by \eqref{JacobiId}} \\
                            &= [A_2,[[A_1,A_0],A_0]] + [A_2,[A_1,[A_2,A_0]]] \tag*{by (\ref{addRel3})} \\
                            &= -[A_2,[A_0,[A_1,A_0]]] - [A_2,[A_1,[A_0,A_2]]] \\
                            &= 2[A_2,A_1 -A_{-1}] - 2[A_2,A_3-A_{-1}] \tag*{by Corollary \ref{Rel1Rel2BC},}\\
                            & \tag*{(\ref{addRel1})$_0$ and (\ref{addRel2})$_2$} \\
                            &= 2[A_2,A_1] - 2[A_2,A_{-1}] - 2[A_2,A_3] + 2[A_2,A_{-1}] \\
                            &= 2[A_2,A_1] - 2[A_0,A_1] \tag*{by (\ref{addRel6})} \\
                            &= 2[A_2,A_1] - 2[A_1,A_2] \tag*{by (\ref{addRel3})} \\
                            &= 4[A_2,A_1].
\end{align*}
Furthermore, note that
\begin{align*}
	4[A_2,A_1]	&= 4[A_1,A_0] \qquad \text{by (\ref{addRel3})}\\
							&= 8G_1				\qquad \qquad \text{by (\ref{G1Def})} \\
	\implies [A_2,A_1] &= 2G_1.
\end{align*}
Next, we consider the case $m = -1$. From (\ref{addRel4}), we have that $[A_{-2},A_0] = [A_{-1},A_1]$. Apply $\ad_{G_1}$ to get that
\begin{align*}
    &[G_1,[A_{-2},A_0]] = [G_1,[A_{-1},A_1]] \\
    &\Leftrightarrow [[G_1,A_{-2}],A_0] + [A_{-2},[G_1,A_0]] = [[G_1,A_{-1}],A_1] + [A_{-1},[G_1,A_1]] \\
    &\Leftrightarrow [A_{-1}-A_{-3},A_0] + [A_{-2},A_1-A_{-1}] = [A_0-A_{-2},A_1] + [A_{-1},A_2-A_0] \tag*{by (\ref{Am+1Def})} \\
    &\Leftrightarrow 2[A_{-1},A_0] - [A_{-3},A_0] + 2[A_{-2},A_1] - [A_{-2},A_{-1}] = [A_0,A_1] + [A_{-1},A_2] \\
    &\Leftrightarrow 2[A_0,A_1] + [A_0,A_3] - [A_{-2},A_{-1}] = [A_0,A_1] + [A_0,A_3] \tag*{by (\ref{addRel3}) and (\ref{addRel5})} \\
    &\Leftrightarrow [A_0,A_1] = [A_{-2},A_{-1}].
\end{align*}
Hence we have the additional relation
\begin{align}
	[A_0,A_1] = [A_{-2},A_{-1}] \label{addRel7}
\end{align}
with which we now have the tools needed to show that (\ref{Onsager1adj}) and (\ref{Onsager2adj}) are satisfied for $m = -1$.
\begin{align*}
  [A_0,[A_0,[A_0,A_{-1}]]]	&= [A_0,[A_0,[A_1,A_0]]] \tag*{by (\ref{addRel3})} \\
                            &= -[A_0,[A_0,[A_0,A_1]]] \\
                            &= -4[A_0,A_1]  \tag*{by (\ref{DG1})} \\
                            &= 4[A_1,A_0] \\
                            &= 4[A_0,A_{-1}] \tag*{by (\ref{addRel3}),}
\end{align*}
\begin{align*}                            
	[A_{-1},[A_{-1},[A_{-1},&A_0]]]	= [A_{-1},[A_{-1},[A_0,A_1]]] \tag*{by (\ref{addRel3})} \\
  				&= [A_{-1},[[A_{-1},A_0],A_1]] + [A_{-1},[A_0,[A_{-1},A_1]]] \tag*{by \eqref{JacobiId}} \\
          &= [A_{-1},[[A_0,A_1],A_1]] + [A_{-1},[A_0,[A_{-1},A_1]]] \tag*{by (\ref{addRel3})} \\
          &= -[A_{-1},[A_1,[A_0,A_1]]] - [A_{-1},[A_0,[A_1,A_{-1}]]] \\
          &= -2[A_{-1},A_2-A_0] + 2[A_{-1},A_2-A_{-2}] \tag*{by Corollary \ref{Rel1Rel2BC},}\\
          & \tag*{(\ref{addRel1})$_1$ and (\ref{addRel2})$_1$} \\
          &= -2[A_{-1},A_2] + 2[A_{-1},A_0] + 2[A_{-1},A_2] - 2[A_{-1},A_{-2}]\\
          &= 2[A_0,A_1] - 2[A_{-1},A_{-2}] \tag*{by (\ref{addRel3})}\\
          &= 2[A_0,A_1] - 2[A_1,A_0] \tag*{by (\ref{addRel7})} \\
          &= 4[A_0,A_1] \\
          &= 4[A_{-1},A_0] \tag*{by (\ref{addRel3}).}
\end{align*}
Furthermore, note that
\begin{align*}
	4[A_0,A_{-1}]	&= 4[A_1,A_0] \qquad \text{by (\ref{addRel3})}\\
								&= 8G_1				\qquad \qquad \text{by (\ref{G1Def})} \\
	\implies [A_0,A_{-1}] &= 2G_1
\end{align*}
as desired. In the general case, let us consider when $m$ is positive. We assume that (\ref{Onsager1adj}), (\ref{Onsager2adj}) and $[A_{n+1},A_n] = 2G_1$ are satisfied for $1 < n < m$ and show that they are also satisfied for $n = m$. To do so we define $A'_k$ ($k \in \mathbb{Z}$) and $G'_1$ as follows:
\begin{align*}
	A'_0 	&= A_{m-1}, \\
	A'_1	&= A_m, \\
	G'_1	&= \frac{1}{2}[A'_1,A'_0], \\
	A'_{k+1} &= A'_{k-1} + [G'_1,A'_k].
\end{align*}
By the induction hypothesis, the pair $A_{m-1}$, $A_m$ satisfies (\ref{Onsager1adj}), (\ref{Onsager2adj}) and $[A_m,A_{m-1}]$ $= 2G_1$. Hence, so does the pair $A'_0$, $A'_1$. Case $m=1$ then tells that the pair $A'_1$, $A'_2$ also satisfies (\ref{Onsager1adj}), (\ref{Onsager2adj}) and $[A'_2,A'_1] = 2G'_1$. But note that 
\begin{align*}
	G'_1 	&= \frac{1}{2}[A'_1,A'_0] 		& \text{by definition}\\
				&= \frac{1}{2}[A_m, A_{m-1}]	& \text{by definition} \\
				&= G_1												& \text{by the induction hypothesis,}\\
				\\
	A'_1	&= A_m												& \text{by definition,}\\
				\\
	A'_2	&= A'_0 + [G'_1,A'_1]					& \text{by definition} \\
				&= A_{m-1} + [G_1,A_m]				& \text{by definition and above} \\
				&= A_{m-1} + A_{m+1} - A_{m-1}	& \text{by (\ref{Am+1Def})} \\
				&= A_{m+1}.
\end{align*}
Therefore the pair $A_m$, $A_{m+1}$ satisfies (\ref{Onsager1adj}), (\ref{Onsager2adj}) and $[A_{m+1},A_m] = 2G_1$. By induction, this proves our lemma for the case of $m$ positive. The proof for $m$ negative is the same as for when it is positive but our induction hypothesis will be for $-1 > n > m = -k$ where $k$ is any positive integer.
\end{proof}

\begin{lem} \emph{(\cite[Lemma 6]{RoanS91} \cite[Theorem 2]{DaviesB91})} \label{IndiceDep} If $A_0$ and $A_1$ satisfy the Dolan-Grady relations (\ref{DG1}) and (\ref{DG2}), then for $N \in \mathbb{N}_+$ and all values $l$, $m$ satisfying $m = 0,...,N$, $m-l = N$, we have
\begin{align}
    [A_l,A_m] = [A_{l+1},A_{m+1}]
\end{align}
(i.e., the commutators $[A_l,A_m]$ depend only on the difference $l-m$ of the indices).
\end{lem}

\begin{proof} We proceed by induction on $N$. 

\textit{Base Cases:} It is simply a matter of noting that case $N=1$ follows from (\ref{addRel3}), case $N=2$ follows from (\ref{addRel4}) and case $N=3$ follows from (\ref{addRel5}).

\textit{Induction Hypothesis (I.H.):} Assume that for $3 \leq M \leq N-1$ and for all $l$ and $m$ such that $m=0,...,M$ and $m-l=M$, we have that $[A_l,A_m]=[A_{l+1},A_{m+1}]$.

Before proceeding to the inductive step, we will show that here the relations
\begin{align*}
    A_{n+1} - A_{-n-1} &= -\frac{1}{2}[A_0,[A_1,A_{-n}]], & n \geq 0, \qquad \qquad (\ref{addRel1}) \\
    A_{n+1} - A_{-n+1} &= \frac{1}{2}[A_1,[A_0,A_n]], & n \geq 1, \qquad \qquad (\ref{addRel2})
\end{align*}
are satisfied as we will need to use these properties in our proof below. To do so, we again proceed by induction. The base cases ($n = 0,1,2$) are true in general, as proven in Corollary \ref{Rel1Rel2BC}. Next, we assume (\ref{addRel1})$_n$ and (\ref{addRel2})$_n$ are satisfied for $2 \leq n \leq N-1$ and proceed to prove the case $n = N$. We begin by noting that the induction hypothesis for our first induction tells us that for $1 \leq n \leq N-1$, 
\begin{align*}
	[A_{-n},A_0]	&= [A_{1-n},A_1]	&\text{by (I.H.) for $l=-n$, $m=0$} \\
								&= [A_{2-n},A_2]	&\text{by (I.H.) for $l=1-n$, $m=1$} \\
								&\qquad ~ \vdots 	\\
								&= [A_{-1},A_{n-1}]	&\text{by (I.H.) for $l=-2$, $m=n-2$} \\
								&= [A_0,A_n] 				&\text{by (I.H.) for $l=-1$, $m=n-1$.}								
\end{align*}
So by Lemma \ref{addLemma}, we have that (\ref{addRel1})$_n \Leftrightarrow$ (\ref{addRel2})$_n$. Hence it is sufficient to prove (\ref{addRel1})$_n$, the proof of which is as follows: 
\begin{align*}
	-\frac{1}{2}[A_0,[A_1,A_{-n}]]	&= -\frac{1}{2}[[A_0,A_1],A_{-n}] - \frac{1}{2}[A_1,[A_0,A_{-n}]] \tag*{by \eqref{JacobiId}} \\
												&= [G_1,A_{-n}] - \frac{1}{2}[A_1,[A_0,A_{-n}]] \tag*{by (\ref{G1Def})} \\
												&= [G_1,A_{-n}] + \frac{1}{2}[A_1,[A_0,A_n]] \tag*{since $[A_{-n},A_0] = [A_0,A_n]$} \\
												&= [G_1,A_{-n}] + \frac{1}{2}[[A_1,A_0],A_n] + \frac{1}{2}[A_0,[A_1,A_n]] \tag*{by \eqref{JacobiId}} \\
												&= [G_1,A_{-n}] + [G_1,A_n] + \frac{1}{2}[A_0,[A_{-n+2},A_1]] \tag*{by (\ref{G1Def}) and (I.H.)}\\
												&= A_{-n+1} - A_{-n-1} + A_{n+1} - A_{n-1} - \frac{1}{2}[A_0,[A_1,A_{-(n-2)}]] \tag*{by (\ref{Am+1Def})} \\
												&= A_{-n+1} - A_{-n-1} + A_{n+1} - A_{n-1} + A_{n-1} - A_{-n+1} \tag*{by induction} \\
												&= A_{n+1} - A_{-n-1}.
\end{align*}
Hence, by induction, we have that (\ref{addRel1})$_n$ and (\ref{addRel2})$_n$ are satisfied for $0 \leq n \leq N$ and can now return to proving our lemma.

\textit{Inductive Step:} First, consider the case when $M=N-1$. So for all $l$ and $m$ such that $m=0,...,(N-1)$ and $m-l=N-1$, we have that $[A_l,A_m]=[A_{l+1},A_{m+1}]$. Act on both sides of the equation by $\ad_{G_1}$:
\begin{align*}
\text{LHS}	&= [G_1,[A_l,A_m]]\\
		&= [[G_1,A_l],A_m] + [A_l,[G_1,A_m]] \tag*{by \eqref{JacobiId}}\\
		&= [A_{l+1}-A_{l-1},A_m] + [A_l,A_{m+1}-A_{m-1}] \tag*{by (\ref{Am+1Def})}\\
		&= [A_{l+1},A_m] - [A_{l-1},A_m] + [A_l,A_{m+1}] - [A_l,A_{m-1}]\\
		&= [A_l,A_{m+1}] - [A_{l-1},A_m] \tag*{by (I.H.) for $M= N-2$,}
		\\
\text{RHS}	&= [G_1,[A_{l+1},A_{m+1}]]\\
		&= [[G_1,A_{l+1}],A_{m+1}] + [A_{l+1},[G_1,A_{m+1}]] \tag*{by \eqref{JacobiId}}\\
		&= [A_{l+2}-A_l,A_{m+1}] + [A_{l+1},A_{m+2}-A_m] \tag*{by (\ref{Am+1Def})}\\
		&= [A_{l+2},A_{m+1}] - [A_l,A_{m+1}] + [A_{l+1},A_{m+2}] - [A_{l+1},A_m]\\
		&= [A_{l+1},A_{m+2}] - [A_l,A_{m+1}] \tag*{by (I.H.) for $M= N-2$.}
\end{align*}
Hence,
\begin{gather*}
[A_l,A_{m+1}] - [A_{l-1},A_m] = [A_{l+1},A_{m+2}] - [A_l,A_{m+1}]\\
\Longleftrightarrow [A_{l-1},A_m] = 2[A_l,A_{m+1}] - [A_{l+1},A_{m+2}].
\end{gather*}
If we let $l'=l-1$, then we obtain the following $N$ equations
\begin{align*}
	 [A_{l'},A_m] = 2[A_{l'+1},A_{m+1}] - [A_{l'+2},A_{m+2}],
\end{align*}
where $m = 0,\dots,(N-1)$ and $m-l'=N$. This is equivalent to
\begin{align*}
	 [A_{m-N},A_m] = 2[A_{m+1-N},A_{m+1}] - [A_{m+2-N},A_{m+2}],
\end{align*}
for $m = 0,\dots,(N-1)$. For notational simplicity, let $X_m = [A_{m-N},A_m]$. So, since $m = 0,...,(N-1)$, we have $N+2$ variables along with the $N$ equations,
\begin{align}
	X_m - 2X_{m+1} + X_{m+2} = 0. \label{Nequs}
\end{align}
Observe that the coefficient matrix of the linear system is in row-echelon form. For example, for $N=4$ we get
\begin{center}
$\begin{bmatrix}
1&2&1&0&0&0\\
0&1&2&1&0&0\\
0&0&1&2&1&0\\
0&0&0&1&2&1
\end{bmatrix}$.
\end{center}
Hence, in general, we can determine $X_0,\dots,X_{N-1}$ as a linear combination of $X_N$ and $X_{N+1}$. We postpone the proof that there exists one more independent equation which will allow us to take $X_{N+1}$ as independent variable and $X_0,...,X_N$ as dependent variables. Rather, we suppose for a moment that
\begin{align}
	X_j = a_jX_{N+1} &\qquad \text{for $0 \leq j \leq N$, $a_j \in k$.}
\end{align}
It follows from (\ref{Nequs}) that 
\begin{gather*}
	X_m - 2X_{m+1} + X_{m+2} = 0 \\
	\Leftrightarrow a_mX_{N+1} - 2a_{m+1}X_{N+1} + a_{m+2}X_{N+1} = 0\\
	\Leftrightarrow (a_m - 2a_{m+1} + a_{m+2})X_{N+1} = 0.
\end{gather*}
Thus, for all $m=0,...,N-1$, we have that 
\begin{align}
	a_m - 2a_{m+1} + a_{m+2} = 0. \label{coefaN}
\end{align}
Using these relations, one can prove the claim that \[a_{N-q} = (q+1)a_N - q\text{, for }1 \leq q \leq N.\] To do so, we proceed by induction on $q$. 

\noindent\emph{Base Cases:} For $q=1$ and $q=2$ we have that
\begin{align*}
a_{N-1}	&= 2a_N - a_{N+1}	& a_{N-2}	&= 2a_{N-1} - a_N		\tag*{by (\ref{coefaN})}\\
				&= 2a_N - 1,			&					&= 2(2a_N - 1) - a_N \\
				&									&					&= 3a_N - 2. 
\end{align*}

\noindent\emph{Induction Hypothesis (I.H.):} Assume $a_{N-q} = (q+1)a_N - q$ for $1 \leq q \leq N-1$. 

\noindent\emph{Inductive Step:} By our induction hypothesis we have that 
\[a_{N-(N-2)} = (N-1)a_N - (N-2) \quad \text{and} \quad a_{N-(N-1)} = Na_N - (N-1).\]
Using this we can show the $q=N$ case as follows:
\begin{align*}
	a_{N-N}	&= a_0 = 2a_1 - a_2 \tag*{by (\ref{coefaN})} \\
					&= 2(Na_N - (N-1)) - ((N-1)a_N - (N-2)) \tag*{by I.H.}\\
					&= (N+1)a_N - N.
\end{align*}
So this proves our claim and gives us the relations between our variables as follows:
\begin{align*}
	X_{N+1}	&= X_{N+1}, \\
	X_N 		&= a_NX_{N+1}, \text{ and}\\
	X_{N-q}	&= ((q+1)a_N - q)X_{N+1}, \quad \text{ for } 1 \leq q \leq N.
\end{align*}
Then, once we have our additional independent equation, we will be able to figure out the value of $a_N$. 

To find one more independent equation, we must consider the even and odd $N$ cases separately. Let us begin with the even case, $N=2m$, and consider the ($N-1$) case which tells us that $[A_{-2m+2},A_1] = [A_{-m+1},A_m]$. We act on both sides of this relation by $\ad_{A_1}\ad_{A_0}$ to get
\begin{align*}
\text{LHS}	={}& [A_1,[A_0,[A_{-2m+2},A_1]]]\\
						={}& -[A_1,[A_0,[A_1,A_{-2m+2}]]] \\
						={}& 2[A_1,A_{2m-1}-A_{-2m+1}] \tag*{by (\ref{addRel1})$_{2m-2}$}\\
						={}& 2[A_1,A_{2m-1}] + 2[A_{-2m+1},A_1],\\
\quad \\
\text{RHS} ={}& [A_1,[A_0,[A_{-m+1},A_m]]] \\
		={}& [A_1,[[A_0,A_{-m+1}],A_m]] + [A_1,[A_{-m+1},[A_0,A_m]]] \tag*{by \eqref{JacobiId}}\\
		={}& [[A_1,[A_0,A_{-m+1}]],A_m] + [[A_0,A_{-m+1}],[A_1,A_m]] + [[A_1,A_{-m+1}],[A_0,A_m]]\\
		 {}& + [A_{-m+1},[A_1,[A_0,A_m]]] \tag*{by \eqref{JacobiId}}\\
		={}& [[[A_1,A_0],A_{-m+1}],A_m] + [[A_0,[A_1,A_{-m+1}]],A_m] + [[A_m,A_1],[A_0,A_{-m+1}]]\\
		 {}& + [[A_m,A_0],[A_1,A_{-m+1}]] + 2[A_{-m+1},A_{m+1}-A_{-m+1}] \tag*{by \eqref{JacobiId} \& (\ref{addRel2})$_m$}\\
		={}& 2[[G_1,A_{-m+1}],A_m] - 2[A_m-A_{-m},A_m] + 0 \\
		 {}& + 0 + 2[A_{-m+1},A_{m+1}-A_{-m+1}] \tag*{by I.H., (\ref{G1Def}) and (\ref{addRel1})$_{m-1}$}\\
		={}& 2[A_{-m+2}-A_{-m},A_m] + 2[A_{-m},A_m] + 2[A_{-m+1},A_{m+1}] \tag*{by (\ref{Am+1Def})}\\
		={}& 2[A_{-m+2},A_m] + 2[A_{-m+1},A_{m+1}],\\
\quad \\
\text{LHS} ={}& \text{RHS} \\
\Leftrightarrow{}& 2[A_1,A_{2m-1}] + 2[A_{-2m+1},A_1] = 2[A_{-m+2},A_m] + 2[A_{-m+1},A_{m+1}] \\
\Leftrightarrow{}& [A_1,A_{2m-1}] + [A_{-2m+1},A_1] = [A_{-m+2},A_m] + [A_{-m+1},A_{m+1}] \\
\Leftrightarrow{}& [A_{-2m+1},A_1] = [A_{-m+1},A_{m+1}] \tag*{by I.H. for $N-2$} \\
\Leftrightarrow{}& [A_{-N+1},A_1] = [A_{-\frac{N}{2}+1},A_{\frac{N}{2}+1}] \tag*{since $N=2m$}\\
\Leftrightarrow{}& [A_{1-N},A_1] = [A_{\frac{N}{2}+1-N},A_{\frac{N}{2}+1}] \\
\Leftrightarrow{}& X_1 = X_{\frac{N}{2}+1} \\
\Leftrightarrow{}& X_1 - X_{\frac{N}{2}+1} = 0.
\end{align*}
So we have our additional independent equation which allows us to take $X_{N+1}$ as independent variable and can now determine the value of $a_N$ in the even case. 
\begin{gather*}
X_1 - X_{\frac{N}{2}+1} = 0 \\
\Leftrightarrow	(Na_N-(N-1))X_{N+1} - \left(\frac{N}{2}a_N - \frac{N}{2} + 1\right)X_{N+1} = 0 \\
\Leftrightarrow \left(\left(N-\frac{N}{2}\right)a_N - N + \frac{N}{2} +1 - 1\right)X_{N+1} = 0 \\
\Leftrightarrow \frac{N}{2}a_N - \frac{N}{2} = 0 \\
\Leftrightarrow a_N = 1.
\end{gather*}
Hence, we have that $a_{N-q} = (q+1)1 - q = 1$ for all $1 \leq q \leq N$, leading us to the conclusion that 
\begin{gather*}
X_0 = X_1 = \ldots = X_{N-1} = X_N = X_{N+1}\\
\Longleftrightarrow [A_{-N},A_0] = [A_{1-N},A_1] = \ldots = [A_0,A_N] = [A_1,A_{N+1}]. 
\end{gather*}
So we have the result for even $N$. 

Next, we consider the odd case, $N=2m+1$. This case requires a little more work in order to find our additional equation. First, by induction hypothesis, we have that $[A_{-2m},A_0]=[A_{-m},A_m]$. If we act on both sides by $\ad_{A_0}\ad_{A_1}$, we obtain
\begin{align*}
\text{LHS} 	={}& [A_0,[A_1,[A_{-2m},A_0]]]\\
						={}& [A_0,[[A_1,A_{-2m}],A_0]] + [A_0,[A_{-2m},[A_1,A_0]]] \tag*{by \eqref{JacobiId}}\\
						={}& -[A_0,[A_0,[A_1,A_{-2m}]]] - 2[A_0,[G_1,A_{-2m}]] \tag*{by (\ref{G1Def})}\\
						={}& 2[A_0,A_{2m+1}-A_{-2m-1}] - 2[A_0,A_{-2m+1}-A_{-2m-1}] \tag*{by (\ref{addRel1})$_{2m}$ and (\ref{Am+1Def})}\\
						={}& 2[A_0,A_{2m+1}] + 2[A_{-2m+1},A_0],\\
						\\
\text{RHS}	={}& [A_0,[A_1,[A_{-m},A_m]]]\\
						={}& [A_0,[[A_1,A_{-m}],A_m]] + [A_0,[A_{-m},[A_1,A_m]]] \tag*{by \eqref{JacobiId}}\\
						={}& [[A_0,[A_1,A_{-m}]],A_m] + [[A_1,A_{-m}],[A_0,A_m]] + [[A_0,A_{-m}],[A_1,A_m]] \\
						 {}& + [A_{-m},[A_0,[A_1,A_m]]] \tag*{by \eqref{JacobiId}}\\
						={}& -2[A_{m+1}-A_{-m-1},A_m] + [[A_1,A_{-m}],[A_0,A_m]] + [[A_0,A_{-m}],[A_1,A_m]] \\
						 {}& + [A_{-m},[[A_0,A_1],A_m]] + [A_{-m},[A_1,[A_0,A_m]]] \tag*{by (\ref{addRel1})$_m$ and \eqref{JacobiId}}\\
						={}& 2[A_m,A_{m+1}] + 2[A_{-m-1},A_m] + [[A_1,A_{-m}],[A_0,A_m]] + [[A_0,A_{-m}],[A_1,A_m]] \\
						 {}& - 2[A_{-m},[G_1,A_m]] + 2[A_{-m},A_{m+1}-A_{-m+1}] \tag*{by (\ref{G1Def}) and (\ref{addRel2})$_m$}\\
						={}& 2[A_m,A_{m+1}] + 2[A_{-m-1},A_m] + [[A_1,A_{-m}],[A_0,A_m]] + [[A_0,A_{-m}],[A_1,A_m]] \\
						 {}& - 2[A_{-m},A_{m+1}-A_{m-1}] + 2[A_{-m},A_{m+1}-A_{-m+1}] \tag*{by (\ref{Am+1Def})}\\
						={}& 2[A_m,A_{m+1}] + 2[A_{-m-1},A_m] + [[A_1,A_{-m}],[A_0,A_m]] + [[A_0,A_{-m}],[A_1,A_m]] \\
						 {}& + 2[A_{-m},A_{m-1}] - 2[A_{-m},A_{-m+1}]\\
						={}& 2[A_{-m-1},A_m] + 2[A_{-m},A_{m-1}] + [[A_1,A_{-m}],[A_0,A_m]] + [[A_0,A_{-m}],[A_1,A_m]].
\end{align*}
So from LHS $=$ RHS, we get that
\begin{align*}
2[A_0,A_{2m+1}] + 2[A_{-2m+1},A_0] ={}& 2[A_{-m-1},A_m] + 2[A_{-m},A_{m-1}] \\
																			& + [[A_1,A_{-m}],[A_0,A_m]] + [[A_0,A_{-m}],[A_1,A_m]].
\end{align*}
Then, we note that by the induction hypothesis we get $[A_{-2m+1},A_0] = [A_{-m},A_{m-1}]$. Hence we get the first equation we will be working with:
\begin{align}
2[A_0,A_{2m+1}] = 2[A_{-m-1},A_m] + [[A_1,A_{-m}],[A_0,A_m]] + [[A_0,A_{-m}],[A_1,A_m]]. \label{equn1}
\end{align}
Next, again by induction hypothesis, we have that $[A_{-2m+1},A_1] = [A_{-m},A_m]$. If we act on both sides by $\ad_{A_1}\ad_{A_0}$, then we have that
\begin{align*}
\text{LHS}	&= [A_1,[A_0,[A_{-2m+1},A_1]]] \\
						&= -[A_1,[A_0,[A_1,A_{-2m+1}]]] \\
						&= 2[A_1,A_{2m}-A_{-2m}] \tag*{by (\ref{addRel1})$_{2m-1}$}\\
						&= 2[A_1,A_{2m}] - 2[A_1,A_{-2m}],\\
						\\
\text{RHS}	&= [A_1,[A_0,[A_{-m},A_m]]] \\
						&= [A_1,[[A_0,A_{-m}],A_m]] + [A_1,[A_{-m},[A_0,A_m]]] \tag*{by \eqref{JacobiId}}\\
						&= [[A_1,[A_0,A_{-m}]],A_m] + [[A_0,A_{-m}],[A_1,A_m]] + [[A_1,A_{-m}],[A_0,A_m]] \\
						& ~ + [A_{-m},[A_1,[A_0,A_m]]] \tag*{by \eqref{JacobiId}}\\
						&= [[[A_1,A_0],A_{-m}],A_m] + [[A_0,[A_1,A_{-m}]],A_m] + [[A_0,A_{-m}],[A_1,A_m]] \\
						& ~ + [[A_1,A_{-m}],[A_0,A_m]] + 2[A_{-m},A_{m+1}-A_{-m+1}] \tag*{by \eqref{JacobiId}, (\ref{addRel2})$_m$}\\
						&= 2[[G_1,A_{-m}],A_m] - 2[A_{m+1}-A_{-m-1},A_m] + [[A_0,A_{-m}],[A_1,A_m]] \\
						& ~ + [[A_1,A_{-m}],[A_0,A_m]] + 2[A_{-m},A_{m+1}] - 2[A_{-m},A_{-m+1}] \tag*{by (\ref{G1Def}), (\ref{addRel1})$_m$}\\
						&= 2[A_{-m+1},A_m] + 2[A_m,A_{-m-1}] - 2[A_{m+1},A_m] + 2[A_{-m-1},A_m]  \\
						& + [[A_0,A_{-m}],[A_1,A_m]] + [[A_1,A_{-m}],[A_0,A_m]] + 2[A_{-m},A_{m+1}] - 2[A_{-m},A_{-m+1}] \\
						&= 2[A_{-m+1},A_m] + 2[A_{-m},A_{m+1}] + [[A_0,A_{-m}],[A_1,A_m]] + [[A_1,A_{-m}],[A_0,A_m]].  
\end{align*}
So from LHS $=$ RHS, we have that 
\begin{align*}
2[A_1,A_{2m}] - 2[A_1,A_{-2m}] ={}& 2[A_{-m+1},A_m] + 2[A_{-m},A_{m+1}] \\
																	& + [[A_1,A_{-m}],[A_0,A_m]] + [[A_0,A_{-m}],[A_1,A_m]]. \end{align*}
Then, we note that by induction hypothesis we have that $[A_1,A_{2m}] = [A_{-m+1},A_m]$. Hence we get the second equation we will be working with:
\begin{align}
2[A_{-2m},A_1] = 2[A_{-m},A_{m+1}] + [[A_1,A_{-m}],[A_0,A_m]] + [[A_0,A_{-m}],[A_1,A_m]]. \label{equn2}
\end{align}
We then combine (\ref{equn1}) and (\ref{equn2}) to get our additional independent equation
\begin{align*}
&[A_0,A_{2m+1}] - [A_{-2m},A_1] = [A_{-m-1},A_m] - [A_{-m},A_{m+1}] \\
&\Leftrightarrow [A_0,A_N] - [A_{1-N},A_1] = \left[A_{\frac{-1-N}{2}},A_{\frac{N-1}{2}}\right] - \left[A_{\frac{1-N}{2}},A_{\frac{N+1}{2}}\right] \tag*{since $N=2m+1$}\\
&\Leftrightarrow [A_{N-N},A_N] - [A_{1-N},A_1] - \left[A_{\frac{N-1}{2}-N},A_{\frac{N-1}{2}}\right] + \left[A_{\frac{N+1}{2}-N},A_{\frac{N+1}{2}}\right] = 0 \\
&\Leftrightarrow X_N - X_1 - X_{\frac{N-1}{2}} + X_{\frac{N+1}{2}} = 0.
\end{align*}
We can now determine the value of $a_N$ in the odd case. 
\begin{align*}
& X_N - X_1 - X_{\frac{N-1}{2}} + X_{\frac{N+1}{2}} = 0 \\
&\Leftrightarrow	X_N - X_{N-(N-1)} - X_{N-\frac{N+1}{2}} + X_{N-\frac{N-1}{2}} = 0 \\
&\Leftrightarrow 	a_NX_{N+1} - (Na_N-(N-1))X_{N+1} - \left(\left(\frac{N+3}{2}\right)a_N-\left(\frac{N+1}{2}\right)\right)X_{N+1} \\
& \qquad \qquad \qquad \qquad \qquad \qquad \qquad \qquad + \left(\left(\frac{N+1}{2}\right)a_N-\left(\frac{N-1}{2}\right)\right)X_{N+1} = 0 \\
&\Leftrightarrow \left(1 - N - \left(\frac{N+3}{2}\right) + \left(\frac{N+1}{2}\right)\right)a_NX_{N+1} \\
& \qquad \qquad \qquad \qquad \qquad \qquad \quad + \left((N-1) + \left(\frac{N+1}{2}\right) - \left(\frac{N-1}{2}\right)\right)X_{N+1} = 0 \\
&\Leftrightarrow \left(\left(\frac{2-2N-(N+3)+(N+1)}{2}\right)a_N + \frac{2N-2+N+1-N+1}{2}\right)X_{N+1} = 0 \\
&\Leftrightarrow \frac{N}{2}a_N - \frac{N}{2} = 0 \qquad \text{since $X_{N+1} \neq 0$}\\
&\Leftrightarrow a_N = 1.
\end{align*}
Hence, we have that $a_{N-q} = (q+1)1 - q = 1$ for all $1 \leq q \leq N$, leading us to the conclusion that 
\begin{gather*}
X_0 = X_1 = \ldots = X_{N-1} = X_N = X_{N+1}\\
\Longleftrightarrow [A_{-N},A_0] = [A_{1-N},A_1] = \ldots = [A_0,A_N] = [A_1,A_N+1]. 
\end{gather*}
So we have the result for odd $N$, and this completes the proof of our lemma. 
\end{proof}

\begin{cor} \label{corG-m} If $A_0$ and $A_1$ satisfy the Dolan-Grady relations (\ref{DG1}) and (\ref{DG2}), then (\ref{Onsrel1}), i.e., $[A_l,A_m] = 2G_{l-m}$, is satisfied. Furthermore, it follows that	$G_{-m} = -G_m$.
\end{cor}

\begin{proof} It follows from Lemma \ref{IndiceDep} that the commutators $[A_l,A_m]$ depend only on the difference $l-m$ of the indices. Furthermore, recall from (\ref{GmDef}) that we defined $[A_m,A_0] = 2G_m$. We can then use this to generalize to the case $[A_l,A_m]$ as follows:
\begin{align*}
[A_l,A_m] &= [A_{l-1},A_{m-1}] \tag*{by Lemma \ref{IndiceDep}}\\
					&= \ldots \\
					&= [A_{l-m+1},A_1] \tag*{by Lemma \ref{IndiceDep}}\\
					&= [A_{l-m},A_0] \tag*{by Lemma \ref{IndiceDep}}\\
					&= 2G_{l-m} \tag*{by (\ref{GmDef}).}\\
\end{align*}
Hence (\ref{Onsrel1}) is satisfied. Furthermore, it follows that
\begin{align*}
	G_{-m} 	&= \frac{1}{2}[A_{-m},A_0] \tag*{by (\ref{GmDef})}\\
					&= -\frac{1}{2}[A_0,A_{-m}] \\
					&= -\frac{1}{2}[A_m,A_0] \tag*{by Lemma \ref{IndiceDep}}\\
					&= -G_m \tag*{by (\ref{GmDef}).}
\end{align*}
\end{proof}

So, we have shown that (\ref{Onsrel1}) is satisfied. Furthermore, it follows from Lemma \ref{Ons1imp23} that (\ref{Onsrel2}) and {\ref{Onsrel3}) are also satisfied. This proves Theorem \ref{Onsrelthm}. This theorem is then utilized to show our Proposition \ref{IsomDG} in the next section.

\section{Alternate Presentation of the Onsager Algebra}
\label{sec:AlternatePresentationOfTheOnsagerAlgebra}

In this section, we utilize Theorem \ref{Onsrelthm} in order to show Proposition \ref{IsomDG} which states that Onsager's algebra, $\mathcal{O}$, is isomorphic to the Dolan-Grady algebra, $\mathcal{OA}$. 

\begin{proof}[Proof of Proposition \ref{IsomDG}] Consider the map $\zeta : \mathcal{OA} \rightarrow \mathcal{O}$ given by
\begin{align}
    A \mapsto A_0, \qquad B \mapsto A_1. \label{Oisom}
\end{align}
It follows from
\begin{align*}
    [A_0,[A_0,[A_0,A_1]]] &= -[A_0,[A_0,2G_1]] \\
                          &= 2[A_0,[G_1,A_0]] \\
                          &= 2[A_0,A_1 - A_{-1}] \\
                          &= 2([A_0,A_1] - [A_0,A_{-1}]) \\
                          &= -8G_1 \\
                          &= 4[A_0,A_1],
\end{align*}
\begin{align*}
    [A_1,[A_1,[A_1,A_0]]] &= [A_1,[A_1,2G_1]] \\
                          &= -2[A_1,[G_1,A_1]] \\
                          &= -2[A_1,A_2 - A_0] \\
                          &= -2([A_1,A_2] - [A_1,A_0]) \\
                          &= 8G_1 \\
                          &= 4[A_1,A_0],
\end{align*}
that $A_0$, $A_1$ satisfy relations (\ref{DG1}) and (\ref{DG2}), hence proving the existence of a homomorphism as defined in (\ref{Oisom}). This homomorphism is unique since $A$ and $B$ generate $\mathcal{OA}$.

Next, consider the subalgebra $N$ generated by $A_0$ and $A_1$. We claim that from the relations (\ref{Onsrel1}) and (\ref{Onsrel2}) it follows that $N$ is the algebra $\mathcal{O}$. This can be proven by induction on $l \in \mathbb{N}$ as follows.

\emph{Base Case:} Consider the case when $l=1$. From $[A_1,A_0] = 2G_1$, it follows that $G_1 \in N$. Furthermore, since $[G_1,A_0] = A_1 - A_{-1}$, it follows that $A_{-1} = A_1 - (A_1 - A_{-1}) \in N$.

\emph{Induction Hypothesis:} Assume that for all $l < n$ ($n \in \mathbb{N}$), $A_l$, $A_{-l}$, $G_l$ and $G_{-l}$ are included in $N$.

\emph{Inductive Step:} Consider when $l = n$. It follows from
\begin{align*}
    [A_{n-1},A_{-1}] = 2G_n = -2G_{-n} \qquad \text{and} \qquad [G_1,A_{n-1}] = A_n - A_{n-2}
\end{align*}
that $G_n$, $G_{-n}$ and $A_n = (A_n - A_{n-2}) + A_{n-2}$ are included in $N$. Furthermore, from $[G_n,A_0] = A_n - A_{-n}$, it follows that $A_{-n} = A_n - (A_n - A_{-n}) \in N$.

So this confirms that $N$ is in fact the Lie algebra $\mathcal{O}$ since it contains its basis, $\{A_m, G_l \mid m \in \mathbb{Z}, l \in \mathbb{N}_+\}$. Hence, $\zeta$ is surjective. It remains to show that $\zeta$ is injective.

Let $\widetilde{A_0} = A$ and $\widetilde{A_1} = B$. Let $\{\widetilde{A_m}, \widetilde{G_l}| m,l \in \mathbb{Z}\}$ be defined as follows 
\begin{align}
    \widetilde{G_1} &= \frac{1}{2}[\widetilde{A_1},\widetilde{A_0}] \label{G1tilde}\\
    \widetilde{A_{m+1}} - \widetilde{A_{m-1}} &= [\widetilde{G_1}, \widetilde{A_m}], \label{Am+1tilde}\\
    \widetilde{G_m} &= \frac{1}{2}[\widetilde{A_m},\widetilde{A_0}]. \label{Gmtilde} 
\end{align}
Theorem \ref{Onsrelthm} then tells us that $\widetilde{A_m}$ and $\widetilde{G_m}$ satisfy the relations (\ref{Onsrel1}), (\ref{Onsrel2}) and (\ref{Onsrel3}).

Clearly, we have that $\mathcal{OA} \subseteq \Span\{\widetilde{A_m}, \widetilde{G_l}\}$. Furthermore, one notices that $\widetilde{A_m}$ and $\widetilde{G_l}$ can be written in terms of $\widetilde{A_0} = A$ and $\widetilde{A_1} = B$ by recursively applying the definitions of $\widetilde{A_m}$ and $\widetilde{G_l}$ above (\ref{Am+1tilde}) and (\ref{Gmtilde}), so $\Span\{\widetilde{A_m}, \widetilde{G_l}\} \subseteq \mathcal{OA}$. Hence $\Span\{\widetilde{A_m}, \widetilde{G_l}\} = \mathcal{OA}$.

Next, we claim that $\zeta(\widetilde{A_m}) = A_m$ and $\zeta(\widetilde{G_l}) = G_l$. To prove this we proceed by induction on $l \in \mathbb{N}$.

\textit{Base Cases:}
\begin{align*}
l = 0 && \zeta(\widetilde{A_0}) &= \zeta(A) = A_0, \\
l = 1 && \zeta(\widetilde{A_1}) &= \zeta(B) = A_1, \\
      && \zeta(\widetilde{G_1}) &= \frac{1}{2}[\zeta(\widetilde{A_1}),\zeta(\widetilde{A_0})] = \frac{1}{2}[A_1,A_0] = G_1, \\
      && \zeta(\widetilde{G_{-1}}) &= -\zeta(\widetilde{G_1}) = -G_l = G_{-l}, \\
      && \zeta(\widetilde{A_{-1}}) &= \zeta(\widetilde{A_1} - [\widetilde{G_1},\widetilde{A_0}]) = \zeta(\widetilde{A_1}) - [\zeta(\widetilde{G_1}),\zeta(\widetilde{A_0})] \\
      &&&= \zeta(B) - [G_1,A_0] = A_1 - A_1 + A_{-1} = A_{-1}.
\end{align*}

\textit{Induction Hypothesis:} Assume $\zeta(\widetilde{A_n}) = A_n$, $\zeta(\widetilde{A_{-n}}) = A_{-n}$, $\zeta(\widetilde{G_n}) = G_n$ and $\zeta(\widetilde{G_{-n}}) = G_{-n}$ for $0 < n < l$.

\textit{Inductive Step:} Prove for $n = l$.
\begin{align*}
	\zeta(\widetilde{A_l})	&= \zeta(\widetilde{A_{l-2}} + [\widetilde{G_1},\widetilde{A_{l-1}}]) \\
                          &= \zeta(\widetilde{A_{l-2}}) + [\zeta(\widetilde{G_1}),\zeta(\widetilde{A_{l-1}})] \\
                          &= A_{l-2} + [G_1,A_{l-1}] \qquad \text{by induction hypothesis} \\
                          &= A_{l-2} + A_l - A_{l-2} = A_l, \\
	\zeta(\widetilde{G_l})	&= \frac{1}{2}[\zeta(\widetilde{A_l}, \zeta(\widetilde{A_0})] \\
                          &= \frac{1}{2}[A_l,A_0] \qquad \text{by above and by induction hypothesis} \\
                          &= G_l, \\    
	\zeta(\widetilde{A_{-l}})	&= \zeta(\widetilde{A_l} - [\widetilde{G_l},\widetilde{A_0}]) \\
                           	&= \zeta(\widetilde{A_l}) - [\zeta(\widetilde{G_l}),\zeta(\widetilde{A_0})] \\
                            &= A_l - [G_l,A_0] \qquad \text{by above and induction hypothesis} \\
                            &= A_l - A_l + A_{-l} = A_{-l}, \\
	\zeta(\widetilde{G_{-l}})	&= -\zeta(\widetilde{G_l}) = -G_l = G_{-l}.
\end{align*}
This proves our claim. Since $\{A_m, G_l \mid m \in \mathbb{Z}, l \in \mathbb{N}_+\}$ is a basis of $\mathcal{O}$, it follows that $\zeta$ is injective. Hence $\zeta$ is an isomorphism.
\end{proof}

We conclude this chapter by noting that following the result of Proposition \ref{IsomDG}, Definition \ref{OnsDefOrig} and Definition \ref{DGalgDef} are used interchangeably in the literature to define the Onsager algebra.
\cleardoublepage

\chapter{The Onsager Algebra as an Equivariant Map Algebra}
\label{sec:TheOnsagerAlgebraAsAnEquivariantMapAlgebra}

\section{Equivariant Map Algebras}
\label{sec:EquivariantMapAlgebras}

Let $X$ be an algebraic variety and let $\mathfrak{g}$ be an finite-dimensional Lie algebra, both defined over an algebraically closed field of characteristic zero $k$ and equipped with the action of a finite group $\Gamma$ by automorphisms. 

\begin{defn} A \emph{regular map} between affine varieties is a mapping which is given locally by polynomials. 
\end{defn}

\begin{defn}[Map Algebra] The \emph{map algebra} associated to $X$ and $\mathfrak{g}$, denoted $M(X,\mathfrak{g})$, is the Lie algebra of regular maps from $X$ to $\mathfrak{g}$ with bracket defined as \[ [\beta, \gamma]_{M(X,\mathfrak{g})}(x) := [\beta(x),\gamma(x)]_{\mathfrak{g}}, \] for $x \in X$ and $\beta, \gamma \in M(X,\mathfrak{g})$. 
\end{defn}

\begin{defn}[Equivariant Map] A map $f: Y \rightarrow Z$ between the varieties $Y$ and $Z$ is said to be \emph{equivariant} under the action of a group $\Gamma$ if \[ f(g \cdot x) = g \cdot f(x), \] for any $x \in Y$ and $g \in \Gamma$. 
\end{defn}

\begin{defn}[Equivariant Map Algebra] The \emph{equivariant map algebra} associated to $X$, $\mathfrak{g}$ and $\Gamma$, denoted $\mathfrak{M} = M(X,\mathfrak{g})^\Gamma$, is the algebra of maps equivariant with respect to the action of $\Gamma$, that is \[ M(X,\mathfrak{g})^\Gamma = \{\alpha \in M(X,\mathfrak{g}) ~ | ~ \alpha(g \cdot x) = g \cdot \alpha (x) ~ \forall x \in X, g \in \Gamma\}. \]
\end{defn}

Denoting by $A_X$ the coordinate ring of $X$, Lemma 3.4 from \cite{NeherESavageASenesiP09} shows that an equivariant map algebra can also be realized as the fixed point Lie algebra $\mathfrak{M} = (\mathfrak{g} \otimes A_X)^\Gamma$. 

In this chapter, we will consider the case where $X = \Spec k[t,t^{-1}] = k \backslash \{0\}$ (so, $A_X = k[t,t^{-1}]$), $\mathfrak{g} = \mathfrak{sl}_2$ and $\Gamma = \{1,\alpha\}$ where $\alpha$ acts as an involution on $X$ and $\mathfrak{g}$. In Section \ref{sec:RealizationOfTheOnsagerAlgebra}, we will consider the case where $\alpha$ acts as the Chevalley involution and show that this equivariant map algebra, $(k[t,t^{-1}] \otimes \mathfrak{sl}_2)^\Gamma$, is isomorphic to the Onsager algebra, $\mathcal{O}$. In Section \ref{sec:ClosedIdealsOfTheOnsagerAlgebra}, we will use this correspondence to describe the closed ideals of $\mathcal{O}$. And finally, in Section \ref{sec:AlternateInvolution}, we will consider an alternate involution that appears in the literature and show that the equivariant map algebra obtained is isomorphic to the one defined with the Chevalley involution (and hence is also isomorphic to the Onsager algebra).

\section{Realization of the Onsager Algebra}
\label{sec:RealizationOfTheOnsagerAlgebra}

\noindent Recall the standard basis of $\mathfrak{sl}_2$,
\begin{align*}
    \left\{ e = \left( \begin{array}{cc} 0 & 1 \\ 0 & 0 \end{array} \right), \ f = \left( \begin{array}{cc} 0 & 0 \\ 1 & 0 \end{array} \right), \ h = \left( \begin{array}{cc} 1 & 0 \\ 0 & -1 \end{array} \right) \right\},
\end{align*}
satisfying the relations $[e,f] = h$, $[h,e] = 2e$, and $[h,f] = -2f$. Let $\mathcal{L} = \mathcal{L}(\mathfrak{sl}_2) = k[t,t^{-1}] \otimes \mathfrak{sl}_2$ denote the $\mathfrak{sl}_2$-loop algebra with bracket given by
\begin{align}
    [p \otimes x, q \otimes y] = pq \otimes [x,y] 
\end{align}
for $p,q \in k[t,t^{-1}]$ and $x, y \in \mathfrak{sl}_2$.

\begin{defn}[Chevalley Involution] The \emph{Chevalley involution} on $\mathfrak{sl}_2$ is defined by $u \mapsto \overline{u} = \left( \begin{smallmatrix} 0 & 1 \\ 1 & 0 \end{smallmatrix} \right) u \left( \begin{smallmatrix} 0 & 1 \\ 1 & 0 \end{smallmatrix} \right)$ for $u \in \mathfrak{sl}_2$, hence by
\begin{align}
    \overline{e} := f, \quad \overline{f} := e, \quad \overline{h} := -h .
\end{align}
The Chevalley involution is an involution, defined as an automorphism of order 2. (In this thesis, we will not use the more general concept of a Chevalley involution associated to a Chevalley system, \cite[Section 2.4]{BourbakiN75}). 

The Chevalley involution on $\mathfrak{sl}_2$ induces an involution on the loop algebra $\mathcal{L}$, also called the \emph{Chevalley involution} and denoted $\omega$:
\begin{align}
    \omega(p(t) \otimes x) = p(t^{-1}) \otimes \overline{x} , \quad \text{for } p(t) \in k[t,t^{-1}], x \in \mathfrak{sl}_2.
\end{align}
\end{defn}

Let $\mathcal{L}^{\omega}$ denote the Lie subalgebra of $\mathcal{L}$  fixed by the Chevalley involution, thus $\mathcal{L}^{\omega} = (k[t,t^{-1}] \otimes \mathfrak{sl}_2)^\Gamma$ with $\Gamma = \{1,\omega\}$.

\begin{prop} \label{LwBasis} The following elements form a basis of $\mathcal{L}^{\omega}$:
\begin{align}
    c_l &:= (t^l - t^{-l}) \otimes h = -c_{-l}, &l \in \mathbb{N}_+, \label{cldef} \\
    b_m &:= t^m \otimes e + t^{-m} \otimes f, &m \in \mathbb{Z}. \label{bmdef}
\end{align}
\end{prop}

\begin{proof} First note that $\{e-f, h, e+f\}$ forms a basis of $\mathfrak{sl}_2$ and $\{t^i + t^{-i} ~|~ i \in \mathbb{N}\} \cup \{t^j - t^{-j} ~|~ j \in \mathbb{N}_+\}$ forms a basis of $k[t,t^{-1}]$. It follows that 
\begin{align*}
 \left\{ \left.\begin{array}{c} (t^i + t^{-i}) \otimes (e-f), (t^i + t^{-i}) \otimes h, (t^i + t^{-i}) \otimes (e+f), \\ (t^j - t^{-j}) \otimes (e-f), (t^j - t^{-j}) \otimes h, (t^j - t^{-j}) \otimes (e+f) \end{array} \right| i \in \mathbb{N}, j \in \mathbb{N}_+ \right\} 
\end{align*}
forms a basis of $k[t,t^{-1}] \otimes \mathfrak{sl}_2 = \mathcal{L}$.

It is simple to notice that, for $i \in \mathbb{N}$ and $j \in \mathbb{N}_+$, under the action of the Chevalley involution,
\begin{align*}
	a_i &= (t^i + t^{-i}) \otimes (e+f), \\
	c_j &= (t^j - t^{-j}) \otimes h, \\
	d_j &= (t^j - t^{-j}) \otimes (e-f),
\end{align*}
are eigenvectors associated to the eigenvalue $1$ and $(t^j - t^{-j}) \otimes (e+f)$, $(t^i + t^{-i}) \otimes h$ and $(t^i + t^{-i}) \otimes (e-f)$ are eigenvectors associated to eigenvalue $-1$. Hence \[\{a_i ~ | ~ i \in \mathbb{N}\} \cup \{c_j, d_j ~ | ~ j \in \mathbb{N}_+\}\] is a basis of $\mathcal{L}^{\omega}$. Since $a_i = b_i + b_{-i}$ and $d_j = b_j - b_{-j}$, it follows that \[\{c_l ~|~ l \in \mathbb{N}_+\} \cup \{b_m ~|~ m \in \mathbb{Z}\}\] is a spanning set of $\mathcal{L}^{\omega}$. As it is clearly linearly independent, it is in fact a basis of $\mathcal{L}^{\omega}$.
\end{proof}

To simplify notation, from now on we will omit the tensors. Hence our basis will simply be denoted by
\begin{align*}
    c_l &= (t^l - t^{-l}) h = -c_{-l}, &l \in \mathbb{N}_+, \\
    b_m &= t^m e + t^{-m} f, &m \in \mathbb{Z}.
\end{align*}

\begin{theo} \label{ThmBasisRels} The basis of $\mathcal{L}^{\omega}$ from Proposition \ref{LwBasis} satisfies the following relations:
\begin{enumerate}
    \item[(i)] $[b_l,b_m] = c_{l-m}$,
    \item[(ii)] $[c_l,b_m] = 2(b_{m+l} - b_{m-l})$,
    \item[(iii)] $[c_l, c_m] = 0$,
\end{enumerate}
for $l \in \mathbb{N}$ and $m \in \mathbb{Z}$.
\end{theo}

\begin{proof}
\begin{enumerate}
\item[(i)]
\begin{align*}
    \qquad [b_l, b_m] &= [t^le + t^{-l}f, t^me + t^{-m}f] \\
                          &= [t^le, t^{-m}f] + [t^{-l}f, t^me] \\
                          &= t^{l-m}h - t^{-l+m}h \\
                          &= t^{l-m}h - t^{-(l-m)}h \\
                          &= c_{l-m}.
\end{align*}
\item[(ii)]
\begin{align*}    
    \qquad [c_l,b_m] &= [t^lh - t^{-l}h, t^me + t^{-m}f] \\
                     &= [t^lh,t^me] + [t^lh,t^{-m}f] + [-t^{-l}h,t^me] + [-t^{-l}h,t^{-m}f] \\
                     &= 2t^{l+m}e - 2t^{l-m}f - 2t^{m-l}e + 2t^{-l-m}f \\
                     &= 2((t^{l+m}e + t^{-(l+m)}f) - (t^{m-l}e + t^{-(m-l)}f)) \\
                     &= 2(b_{l+m} - b_{m-l}).
\end{align*}
\item[(iii)]
\begin{align*}
    \qquad [c_l, c_m] &= [t^lh - t^{-l}h, t^mh - t^{-m}h] = 0.
\end{align*}
\end{enumerate}
\end{proof}

\begin{lem} \label{IsomChev} The Onsager algebra is isomorphic to the Lie subalgebra of $\mathcal{L}$  fixed by the Chevalley involution; i.e., $\mathcal{O} \cong \mathcal{L}^{\omega}$.
\end{lem}

\begin{proof} Consider the $k$-linear map $\gamma : \mathcal{O} \rightarrow \mathcal{L}^{\omega}$ defined by 
\begin{align}
    A_m \longmapsto b_m \qquad \qquad G_l \longmapsto \frac{1}{2}c_l \label{involutionhom}
\end{align}
for $m \in \mathbb{Z}$ and $l \in \mathbb{N}_+$. It follows from Theorem \ref{ThmBasisRels} that 
\begin{align*}
    [\gamma(A_l),\gamma(A_m)] &= [b_l,b_m] = c_{l-m} = 2\left(\frac{1}{2}c_{l-m}\right)\\
    													&= 2\gamma(G_{l-m}) = \gamma(2G_{l-m})\\
    													&= \gamma([A_l,A_m]), \\
    [\gamma(G_l),\gamma(A_m)] &= \left[\frac{1}{2}c_l,b_m\right] = b_{m+l} - b_{m-l}\\
    													&= \gamma(A_{m+l}) - \gamma(A_{m-l}) = \gamma(A_{m+l} - A_{m-l})\\
    													&= \gamma([G_l,A_m]), \\
    [\gamma(G_l),\gamma(G_m)] &= \left[\frac{1}{2}c_l,\frac{1}{2}c_m\right] \\
    													&= 0 = \gamma(0)\\
    													&= \gamma([G_l,G_m]).
\end{align*}
Hence, $\gamma$ as defined in (\ref{involutionhom}) is a homomorphism. Since $\{A_m, G_l | m \in \mathbb{Z}, l \in \mathbb{N}_+\}$ forms a basis of $\mathcal{O}$ and $\{b_m, \frac{1}{2}c_l | m \in \mathbb{Z}, l \in \mathbb{N}_+\}$ forms a basis of $\mathcal{L}^{\omega}$, (\ref{involutionhom}) defines a bijection between $\mathcal{O}$ and $\mathcal{L}^{\omega}$. This brings us to the conclusion that $\gamma$ defines an isomorphism; i.e., $\mathcal{O} \cong \mathcal{L}^{\omega}$.
\end{proof}

\begin{cor} \label{b0b1gen} $b_0 = e + f$ and $b_1 = te + t^{-1}f$ generate $\mathcal{L}^{\omega}$ as a Lie algebra. 
\end{cor}

\begin{proof} Consider the isomorphisms $\gamma$ from Lemma \ref{IsomChev} and $\zeta$ from Proposition \ref{IsomDG}:
\begin{align*}
	\mathcal{OA} 	&\longrightarrow \mathcal{O} 	~ \longrightarrow \mathcal{L}^{\omega}, \\
			A					&\longmapsto			A_0					~ \longmapsto			b_0, \\
			B					&\longmapsto			A_1					~ \longmapsto			b_1. \\
\end{align*}
Since $A, B$ generate $\mathcal{OA}$, then $(\gamma \circ \zeta)(A) = b_0$ and $(\gamma \circ \zeta)(B) = b_1$ generate $\mathcal{L}^{\omega}$.
\end{proof}

\section{Closed Ideals of the Onsager Algebra}
\label{sec:ClosedIdealsOfTheOnsagerAlgebra}

By exploiting the realization of the Onsager algebra as a fixed subalgebra of the $\mathfrak{sl}_2$-loop algebra, one can classify its closed ideals. 

We begin by noting that an element $X$ of the $\mathfrak{sl}_2$-loop algebra $\mathcal{L}$ is an element of $\mathcal{L}^{\omega}$ (and hence of the Onsager algebra) if it satisfies the criterion developed below.

If $X \in \mathcal{L}$ then, for some $p(t),q(t),r(t) \in k[t,t^{-1}]$, \[X = p(t)e + q(t)f + r(t)h.\]
We apply the Chevalley involution to get \[ \omega(X) = p(t^{-1})f + q(t^{-1})e - r(t^{-1})h.\]
So, 
\begin{align*}
	X \in \mathcal{L}^{\omega} \Leftrightarrow \begin{cases} q(t) = p(t^{-1}),\\ r(t) = -r(t^{-1}). \end{cases}
\end{align*}
Hence giving us the criterion:
\begin{align}
	X \in \mathcal{L}^{\omega} ~ \Leftrightarrow ~ X = p(t)e + p(t^{-1})f + r(t)h ~ ~ \text{with } r(t) + r(t^{-1}) = 0. \label{FixedElemCrit}
\end{align}

\begin{lem} A polynomial $r(t) \in k[t,t^{-1}]$ such that $r(t) + r(t^{-1}) = 0$ can be uniquely written in the form $r(t) = r_+(t) - r_+(t^{-1})$ with $r_+(t) \in k[t]$, $r(0) = 0$, i.e., $r_+(t) = t k[t]$.
\end{lem}

\begin{proof} Let $r(t) \in k[t,t^{-1}]$. By allowing zero coefficients, \[r(t) = \sum^n_{i=-n} a_it^i\] for some $n \in \mathbb{Z}^+$. Then
\begin{align*}
	r(t) = -r(t^{-1})	&\Leftrightarrow \sum^n_{i=-n} a_it^i = \sum^n_{i=-n} -a_it^{-i} \\
										&\Leftrightarrow \sum^n_{i=-n} (a_i + a_{-i})t^i = 0\\
										&\Leftrightarrow a_i = -a_{-i} ~ \text{for } i \in [-n,n].
\end{align*}
In particular, note that $a_0 = 0$. Now, if we define \[r_+(t) = \sum^n_{i=1} a_it^i,\] then
\begin{align*}
	r(t) 	&= r_+(t) + \sum^{-1}_{i=-n} a_it^i = r_+(t) - \sum^n_{i=1} a_it^{-i} = r_+(t) - r_+(t^{-1}).
\end{align*}
\end{proof}

\begin{defn}[Reciprocal Polynomial] Let $P(t)$ be a non-trivial monic polynomial in $k[t]$. $P(t)$ is called a \emph{reciprocal polynomial} if $P(t) = \pm t^dP(t^{-1})$, where $d$ is the degree of $P(t)$.
\end{defn}

The reason for this terminology is the following. If we write $P(t) = a_0 + a_1t + \dots + a_{d-1}t^{d-1} + a_dt^d$ with $a_d=1$, then $P(t)$ is reciprocal if and only if $a_i = a_{d-i}$ for $0 \leq i \leq d$. For example, $1$ is the only reciprocal polynomial of degree $0$ and $t+1$ and $t-1$ are the only reciprocal polynomial of degree $1$. The reciprocal polynomials of degree 2 are the polynomials $t^2 -1$ and $t^2 + a_1t + 1$. 

\begin{rmk} \label{RecPolProp1} If $P(t)$ is a reciprocal polynomial of degree $d$, then $P(t)k[t,t^{-1}] = P(t^{-1})k[t,t^{-1}]$ since $t^d$ is a unit in $k[t,t^{-1}]$.
\end{rmk}

\begin{defn}[Divisible] (\cite{DateERoanS00}) Let $P(t)$ be a non-trivial polynomial in $k[t]$. An element $X$ of $\mathcal{L}^{\omega}$ is said to be \emph{divisible by} $P(t)$, denoted $P(t) \mid X$, if $X = P(t)\alpha$ with $\alpha \in \mathcal{L}$. The notation $P(t) \mid X_1, X_2, \dots, X_n$ will denote the fact that $X_1, X_2, \dots, X_n$ are divisible by $P(t)$.
\begin{align*}
\text{Define } I_{P(t)} &:= \{X \in \mathcal{L}^{\omega} ~ | ~ P(t) \mid X\} \\
												&= \{P(t)\alpha \in \mathcal{L}^{\omega} ~ | ~ \alpha \in \mathcal{L} \} 	\\
												&= \{p(t)e + p(t^{-1})f + r(t)h \in \mathcal{L}^{\omega} ~ | ~ p(t),p(t^{-1}),r(t) \in P(t)k[t,t^{-1}]\}.
\end{align*}
\end{defn}

\begin{lem} \label{IPtIdeal} $I_{P(t)}$ is an ideal of $\mathcal{L}^{\omega}$.
\end{lem}

\begin{proof} It is clear that $I_{P(t)}$ is a subspace of $\mathcal{L}^{\omega}$. Next, consider an arbitrary $X = P(t)\alpha \in I_{P(t)}$. Since $\mathcal{L}^{\omega}$ is a subalgebra of $\mathcal{L}$, then $[X,Y] \in \mathcal{L}^{\omega}$ for any $Y \in \mathcal{L}^{\omega}$. Furthermore, for any $Y \in \mathcal{L}^{\omega}$, \[[X,Y]	= [P(t)\alpha, Y] = P(t)[\alpha,Y] \] is divisible by $P(t)$. Hence, $[X,Y] \in I_{P(t)}$ for any $Y \in \mathcal{L}^{\omega}$ and so $I_{P(t)}$ is indeed an ideal of $\mathcal{L}^{\omega}$.
\end{proof}

\begin{rmk} Lemma \ref{IPtIdeal} is also clear from the description of $I_{P(t)} = \mathcal{L}^{\omega} \cap P(t)\mathcal{L}$ and the fact that $P(t)\mathcal{L}$ is an ideal of $\mathcal{L}$. 
\end{rmk}


\begin{rmk} If $P(t)$ is a reciprocal polynomial then, by Remark \ref{RecPolProp1},
\begin{align*}
	I_{P(t)} = \{p(t)e + p(t^{-1})f + r(t)h \in \mathcal{L}^{\omega} ~ | ~ p(t),r(t) \in P(t)k[t,t^{-1}]\}. 
\end{align*}
\end{rmk}

\begin{lem} \label{InterIdealsLCM} If $P(t)$ and $Q(t)$ are two reciprocal polynomials, then $I_{P(t)} \cap I_{Q(t)} = I_{\lcm(P(t),Q(t))}$, where $\lcm$ denotes the least common multiple in $k[t]$. 
\end{lem}

\begin{proof} Let $R(t) = \lcm(P(t),Q(t))$. Let $X \in I_{R(t)}$. Then
	\begin{align*}
		& ~ R(t) \mid X \\
		\Rightarrow & ~ Q(t) \mid X \text{ and } P(t) \mid X \\
		\Rightarrow & ~ X \in I_{P(t)} \cap I_{Q(t)}.
	\end{align*}
Thus, $I_{R(t)} \subseteq I_{P(t)} \cap I_{Q(t)}$. Let $Y = p(t)e + p(^{-1})f + r(t)h \in I_{P(t)} \cap I_{Q(t)}$. Then
	\begin{align*}
							 	& ~ P(t) \mid p(t),r(t) \text{ and } Q(t) \mid p(t),r(t) \\
		\Rightarrow & ~ p(t),r(t) \in P(t)k[t,t^{-1}] \cap Q(t)k[t,t^{-1}] = R(t)k[t,t^{-1}] \\
		\Rightarrow & ~ Y \in I_{R(t)}.
	\end{align*}
Thus, $I_{P(t)} \cap I_{Q(t)} \subseteq	I_{R(t)}$.
\end{proof}	

In particular, if $Q(t) \mid P(t)$ then it follows from Lemma \ref{InterIdealsLCM} that $I_{P(t)} \subseteq I_{Q(t)}$. Hence, in this case, there is a canonical projection 
\begin{align*}
	\mathcal{L}^{\omega}/I_{P(t)} \longrightarrow \mathcal{L}^{\omega}/I_{Q(t)}.
\end{align*}

\begin{lem} \emph{(\cite[Lemma 1]{DateERoanS00})} \label{DRLem1} Let $P_j(t)$, $1 \leq j \leq J$, be pairwise relatively prime reciprocal polynomials and $P(t) := \prod_{j=1}^J P_j(t)$. Then the canonical projections give rise to an isomorphism of Lie algebras:
\begin{align*}
	\mathcal{L}^{\omega}/I_{P(t)} \longrightarrow \prod_{j=1}^J \mathcal{L}^{\omega}/I_{P_j(t)} ,
\end{align*}
where the Lie algebra on the right hand side is the direct product algebra.
\end{lem}

\begin{proof} We begin by showing that \[\bigcap_{j=1}^J (P_j(t)) = (\prod_{j=1}^J P_j(t)),\] where $(Q(t))$ denotes the ideal of $k[t]$ generated by a polynomial $Q(t) \in k[t]$. We proceed by induction on $J$. The $J=1$ case is trivial. In the case of $J=2$, since $P_1(t)$ and $P_2(t)$ are relatively prime, we have that
\begin{align*}
	(P_1(t)) \cap (P_2(t)) 	&= (\lcm(P_1(t),P_2(t)))\\
													&= (P_1(t)P_2(t)).
\end{align*}
Next, we assume that \[\bigcap_{j=1}^K (P_j(t)) = (\prod_{j=1}^K P_j(t))\] for $K < J$ and proceed to showing the case $K=J$. Let $A = P_1(t)$ and $B = \prod_{j=2}^J P_j(t)$. By induction hypothesis, we have that $(B) = \bigcap_{j=2}^J (P_j(t))$. It is known that since $A = P_1(t)$ and $P_j(t)$ for $2 \leq j \leq J$ are relatively prime, then $A$ and $B$ are also relatively prime. Hence we have that \[\bigcap_{j=1}^J(P_j(t)) = (A)\cap(B) = (\lcm(A,B)) = (AB) = (\prod_{j=1}^J P_j(t)),\] which proves our claim. With this, we note that injectivity simply follows from Lemma \ref{InterIdealsLCM} and the fact that $P(t) := \prod_{j=1}^J P_j(t)$ where $P_j(t)$, $1 \leq j \leq J$, are pairwise relatively prime reciprocal polynomials. So, it remains to show surjectivity.

For $X_j = p_j(t)e + p_j(t^{-1})f + q_j(t)h \in \mathcal{L}^{\omega}$, for $1 \leq j \leq J$, let $N$ be a positive integer such that $t^Np_j(t)$ and $t^Nq_j(t)$ are all polynomials in $t$. 

By the Chinese remainder theorem, there exist polynomials $\tilde{p}(t)$, $\tilde{q}(t) \in k[t]$ such that the following relations hold in $k[t]$:
\begin{align*}
	\tilde{p}(t) \equiv t^Np_j(t), \quad \tilde{q}(t) \equiv t^Nq_j(t) \qquad (\Mod P_j(t)) \quad \forall j.
\end{align*}
Define the element $X$ of $\mathcal{L}^{\omega}$ by
\[ X := p(t)e + p(t^{-1})f + q(t)h, \]
where 
\[ p(t) = \frac{\tilde{p}(t)}{t^N}, ~ q(t) = \frac{1}{2}\left( \frac{\tilde{q}(t)}{t^N} - t^N\tilde{q}(t^{-1}) \right) \in k[t,t^{-1}]. \]

\noindent Note that $p(t) - p_j(t)$ and $q(t) - q_j(t)$ are divisible by $P_j(t)$ for all $j$:

\begin{align*}
	p(t) - p_j(t) &= \frac{\tilde{p}(t)}{t^N} - p_j(t) \\
								&\equiv \frac{t^Np_j(t)}{t^N} - p_j(t) ~ (\Mod P_j(t)) \\
								&= 0, \\
								\\
	q(t) - q_j(t) &= \frac{1}{2}\left( \frac{\tilde{q}(t)}{t^N} - t^N\tilde{q}(t^{-1}) \right) - q_j(t) \\
								&\equiv \frac{1}{2}\left( \frac{t^Nq_j(t)}{t^N} - t^Nt^{-N}q_j(t^{-1}) \right) - q_j(t) ~ (\Mod P_j(t)) \\
								&= \frac{1}{2}\left( q_j(t) - q_j(t^{-1}) \right) - q_j(t) \\
								&= \frac{1}{2}\left( q_j(t) + q_j(t) \right) - q_j(t) \\
								&= 0.
\end{align*}
Hence, $p(t) - p_j(t)$ and $q(t) - q_j(t)$ are in $I_{P_j(t)}$ and $X \equiv X_j ~ (\Mod I_{P_j(t)})$ for all $j$.

\end{proof}

\begin{lem} \emph{(\cite[Lemma 2]{DateERoanS00})} \label{DRLem2} Let $P(t)$ be a reciprocal polynomial and write 
\begin{align*}
	P(t) = (t-1)^L(t+1)^KP^*(t) \qquad L, K \geq 0 \qquad P^*(\pm1) \neq 0.
\end{align*}
Denote $\tilde{P}(t) := (t-1)^{2[L/2]}(t+1)^{2[K/2]}P^*(t)$, where $[r]$ stands for the integral part of a rational number $r$. Then 
\begin{align*}
	Z(I_{P(t)}) = \{p(t)e + p(t^{-1})f + q(t)h \in \mathcal{L}^{\omega} ~ | ~ \tilde{P}(t) \mid p(t), P(t) \mid q(t)\}.
\end{align*}
As a consequence, $I_{P(t)}$ is closed if and only if the zero-multiplicities of $P(t)$ at $t = \pm1$ are even. 
\end{lem}

\begin{proof} Let $X = p(t)e + p(t^{-1})f + q(t)h \in \mathcal{L}^{\omega}$. Then
\begin{align*}
	X \in Z(I_{P(t)}) &\Leftrightarrow [X,\mathcal{L}^{\omega}] \subseteq I_{P(t)} \\
										&\Leftrightarrow P(t) \mid  [X,\mathcal{L}^{\omega}]\\
										&\Leftrightarrow P(t) \mid [X,e+f], [X,te+t^{-1}f] \tag*{by Corollary \ref{b0b1gen}.}
\end{align*}
Note that 
\begin{align*}
	[X,e+f] &= [p(t)e + p(t^{-1})f + q(t)h, e + f] \\
					&= 2q(t)e - 2q(t)f + (p(t)-p(t^{-1}))h \\
					&= 2q(t)e + 2q(t^{-1})f + (p(t)-p(t^{-1}))h \tag*{since $q(t) + q(t^{-1}) = 0$,}\\
					\\
	[X,te+t^{-1}f] 	&= [p(t)e + p(t^{-1})f + q(t)h, te + t^{-1}f] \\
									&= t^{-1}p(t)h - tp(t^{-1})h + 2tq(t)e - 2t^{-1}q(t)f \\
									&= 2tq(t)e + 2t^{-1}q(t^{-1})f + (t^{-1}p(t) - tp(t^{-1}))h \tag*{since $q(t) + q(t^{-1}) = 0$.}
\end{align*}
By furthermore noting that \[ t^{-1}p(t) - tp(t^{-1}) = -(t-t^{-1})p(t) + t(p(t) - p(t^{-1})), \] we have that
\begin{align*}
	P(t) &\mid [X,e+f], [X,te+t^{-1}f] \\
	\Leftrightarrow P(t) &\mid q(t), (t-t^{-1})p(t), p(t)-p(t^{-1}).
\end{align*}
Thus the claim of the lemma is equivalent to showing that \[ P(t) \mid (t-t^{-1})p(t), p(t)-p(t^{-1}) \Leftrightarrow \tilde{P}(t) \mid  p(t).\]
Let $p(t)$ be an element of $k[t,t^{-1}]$ such that $P(t) \mid (t-t^{-1})p(t), p(t)-p(t^{-1})$. We note that from $P(t) \mid (t-t^{-1})p(t)$  it follows that there exists an $\alpha \in k[t,t^{-1}]$ such that $(t-t^{-1})p(t) = \alpha P(t)$. Then, since $(t-t^{-1}) = t^{-1}(t-1)(t+1)$, we have that 
\begin{align*}
	& \qquad (t-t^{-1})p(t) = \alpha (t-1)^L(t+1)^KP^*(t)  \\
	&\Leftrightarrow t^{-1}(t-1)(t+1)p(t) = \alpha (t-1)^L(t+1)^KP^*(t) \\
	&\Leftrightarrow p(t) = \alpha t (t-1)^{L-1}(t+1)^{K-1}P^*(t).
\end{align*}
Hence, we have that
\begin{align*}
	P(t) \mid (t-t^{-1})p(t) \Leftrightarrow P^*(t) \mid p(t), ~ (t-1)^{L-1} \mid p(t), ~ (t+1)^{K-1} \mid p(t). 
\end{align*}
Next we note that in the trivial case when $L=K=0$, we simply have $P(t) = P^*(t) = \tilde{P}(t)$ and $\tilde{P}(t) \mid p(t)$ follows directly.

Next, consider the case when $L>0$. From $(t-1)^{L-1} \mid p(t)$, it follows that there exists $h(t) \in k[t,t^{-1}]$ such that $p(t) = h(t)(t-1)^{L-1}$. Furthermore, since $(t^{-1}-1)^{L-1} = (-t)^{1-L}(t-1)^{L-1}$, we have that
\begin{align*}
	p(t) - p(t^{-1}) 	&= h(t)(t-1)^{L-1} - h(t^{-1})(t^{-1}-1)^{L-1} \\
										&= (t-1)^{L-1}(h(t) - h(t^{-1})(-t)^{1-L}).
\end{align*}
Since $(t-1)^L \mid (p(t)-p(t^{-1}))$, it follows that $(t-1) \mid (h(t) - h(t^{-1})(-t)^{1-L})$, i.e., $h(1)(1-(-1)^{1-L}) = 0$. Note that if $L$ is even, then we have that $2h(1) = 0$, i.e., $h(1)=0$. It follows that $(t-1) \mid h(t)$ and so that $(t-1)^L \mid p(t)$. On the other hand, if $L$ is odd, we do not get any additional information. To summarize, we have that
\begin{align*}
	(t-1)^L& \mid p(t) \qquad \text{for $L$ even,}\\
	(t-1)^{L-1}& \mid p(t) \qquad \text{for $L$ odd}, 
\end{align*}
which is equivalent to $(t-1)^{2[L/2]} \mid p(t)$. If $K=0$, then we can conclude that $\tilde{P}(t) \mid p(t)$. So, consider the case when $K>0$. From $(t+1)^{K-1} \mid p(t)$, it follows that there exists $g(t) \in k[t,t^{-1}]$ such that $p(t) = g(t)(t+1)^{K-1}$. We then have that
\begin{align*}
	p(t) - p(t^{-1}) 	&= g(t)(t+1)^{K-1} - g(t^{-1})(t^{-1}+1)^{K-1} \\
										&= (t+1)^{K-1}(g(t) - g(t^{-1})t^{1-K})
\end{align*}
Since $(t+1)^K \mid p(t)-p(t^{-1})$, it follows that $(t+1) \mid (g(t) - g(t^{-1})t^{1-K})$, i.e., $g(-1)(1-(-1)^{1-K}) = 0$. Note that if $L$ is even, then we have that $2g(-1) = 0$, i.e., $g(-1)=0$. It follows that $(t+1) \mid g(t)$ and so that $(t+1)^K \mid p(t)$. On the other hand, if $L$ is odd, we do not get any additional information. To summarize, we have that
\begin{align*}
	(t+1)^K& \mid p(t) \qquad \text{for $K$ even,}\\
	(t+1)^{K-1}& \mid p(t) \qquad \text{for $K$ odd}, 
\end{align*}
which is equivalent to $(t+1)^{2[K/2]} \mid p(t)$. We can then conclude that, in all possible cases, one has that $\tilde{P}(t) \mid p(t)$. It remains to show that \[\tilde{P}(t) \mid  p(t) \Rightarrow P(t) \mid (t-t^{-1})p(t), p(t)-p(t^{-1}).\] 

If $\tilde{P}(t) \mid p(t)$, then there exists an $\beta(t) \in k[t,t^{-1}]$ such that 
\begin{align*}
	p(t) 	&= \beta(t) \tilde{P}(t) \\
				&= \beta(t)(t-1)^{2[L/2]}(t+1)^{2[K/2]}P^*(t)\\
				&= \begin{cases} \beta(t)P(t) & L,K \text{ even,}\\ \beta(t)(t+1)^{-1}P(t) & \text{$L$ even, $K$ odd,}\\ \beta(t)(t-1)^{-1}P(t) & \text{$K$ even, $L$ odd,}\\ \beta(t)(t-1)^{-1}(t+1)^{-1}P(t) &\text{$L,K$ odd.} \end{cases}
\end{align*}
Since $t - t^{-1} = t^{-1}(t-1)(t+1)$ it follows that 
\begin{align*}
	(t-t^{-1})p(t)	= \begin{cases} \beta(t)t^{-1}(t-1)(t+1)P(t) & L,K \text{ even,}\\ \beta(t)t^{-1}(t-1)P(t) & \text{$L$ even, $K$ odd,}\\ \beta(t)t^{-1}(t+1)P(t) & \text{$K$ even, $L$ odd,}\\ \beta(t)t^{-1}P(t) & \text{$L,K$ odd,} \end{cases}
\end{align*}
which allows us to conclude that $P(t) \mid (t-t^{-1})p(t)$. Furthermore, with the same case breakdown as above, we have that
\begin{align*}
	p(t) - p(t^{-1})	&= \begin{cases} \beta(t)P(t) - \beta(t^{-1})P(t^{-1}),\\ \beta(t)(t+1)^{-1}P(t) - \beta(t^{-1})(t^{-1}+1)^{-1}P(t^{-1}),\\ \beta(t)(t-1)^{-1}P(t) - \beta(t^{-1})(t^{-1}-1)^{-1}P(t^{-1}),\\ \beta(t)(t-1)^{-1}(t+1)^{-1}P(t) - \beta(t^{-1})(t^{-1}-1)^{-1}(t^{-1}+1)^{-1}P(t^{-1}). \end{cases}
\end{align*}
Since $P(t)$ is a reciprocal polynomial, we have that $P(t) = \pm t^dP(t^{-1})$ where $d$ is the degree of $P(t)$. Without loss of generality, we can assume that $P(t^{-1}) = t^{-d}P(t)$. It follows that the cases become
\begin{align*}
	p(t) - p(t^{-1})	&= \begin{cases} \left(\beta(t) - \beta(t^{-1})t^{-d}\right)P(t),\\ \left(\beta(t)(t+1)^{-1} - \beta(t^{-1})(t^{-1}+1)^{-1}t^{-d}\right)P(t),\\ \left(\beta(t)(t-1)^{-1} - \beta(t^{-1})(t^{-1}-1)^{-1}t^{-d}\right)P(t),\\ \left(\beta(t)(t-1)^{-1}(t+1)^{-1} - \beta(t^{-1})(t^{-1}-1)^{-1}(t^{-1}+1)^{-1}t^{-d}\right)P(t), \end{cases}
\end{align*}
and hence we have that $P(t) \mid (p(t)-p(t^{-1}))$. Finally, we can conclude that 
\[X \in Z(I_{P(t)}) ~ \Leftrightarrow ~ X \in \{p(t)e + p(t^{-1})f + q(t)h \in \mathcal{L}^{\omega} ~ | ~ \tilde{P}(t) \mid p(t), P(t) \mid q(t)\}. \]
\end{proof}

\begin{lem} \emph{(\cite[Lemma 3]{DateERoanS00})} \label{DRLem3} Let $I$ be an ideal in $\mathcal{L}^{\omega}$ and $r(t)$ be an element of $k[t,t^{-1}]$. 
\begin{enumerate}
	\item[(i)] If $r(t)e + r(t^{-1})f$ is an element in $I$, then $(p(t)-p(t^{-1}))h \in I$ for any $p(t) \in r(t)k[t,t^{-1}]$.
	\item[(ii)] For a closed ideal $I$ and an integer $l$, one has
	\begin{align*}
		(t^jr(t) - t^{-j}r(t^{-1}))h \in I ~ (j = 0,-1) ~\Longrightarrow &~ p(t^{\pm1})e + p(t^{\pm1})f \in I \\
																																     &\text{for any } p(t) \in r(t)k[t,t^{-1}] \\
	\end{align*}
	\vspace{-2cm}
	\begin{align*}
		r(t)e + r(t^{-1})f \in I &\Longleftrightarrow r(t^{-1})e + r(t)f \in I \\
														&\Longleftrightarrow t^lr(t)e + t^{-l}r(t^{-1})f \in I \\
														&\Longleftrightarrow (t^jr(t) - t^{-j}r(t^{-1}))h \in I ~ (j = 0,-1).
	\end{align*}
\end{enumerate}
\end{lem}

\begin{proof} Let $I$ be an ideal in $\mathcal{L}^{\omega}$ and $r(t)$ be an element of $k[t,t^{-1}]$. Note that $r(t)$ can be written as $r(t) = \sum_n a_n t^n$, for $a_n \in k$ such that $a_n = 0$ when $|n| \gg 0$. 

(i) Assume $r(t)e + r(t^{-1})f \in I$ and let $p(t) \in r(t)k[t,t^{-1}]$. So, $p(t)$ can be written as $r(t) = \sum_l d_l t^lr(t)$, for $d_l \in k$ such that $d_l = 0$ when $|l| \gg 0$. Denote $p_l(t) := t^lr(t)$ and so $p(t) = \sum_l d_lp_l(t)$. We note that if $(p_l(t) - p_l(t^{-1}))h \in I$, for any $l \in \mathbb{Z}$, then 
\begin{align*}
	 & \qquad 			\sum_l d_l (p_l(t)-p_l(t^{-1}))h \in I, \\
	 &\Leftrightarrow \sum_l d_l p_l(t) h - \sum_l d_l p_l(t^{-1})h \in I, \\
	 &\Rightarrow	(p(t) - p(t^{-1}))h \in I.
\end{align*}
Hence, it is sufficient to show the claim (i) for $p(t) = t^mr(t)$ for arbitrary $m \in \mathbb{Z}$.

With this notation and (\ref{bmdef}), we have that
\[ r(t)e + r(t^{-1})f \in I ~	\Leftrightarrow ~ \sum_n a_n (t^ne + t^{-n}f) \in I ~ \Leftrightarrow ~ \sum_n a_n b_n \in I. \]
This implies that, for all $m \in \mathbb{Z}$, 
\begin{align*}
	&~ \left[\sum_n a_n b_n, b_{-m}\right] \in I \\
	\Rightarrow &~ \sum_n a_n [b_n, b_{-m}] \in I  \\
	\Rightarrow &~ \sum_n a_n c_{m+n} \in I \tag*{by Theorem \ref{ThmBasisRels}}\\
	\Leftrightarrow &~ \sum_n a_n (t^{m+n} - t^{-m-n})h \in I \tag*{by (\ref{cldef})} \\
	\Leftrightarrow &~ \left( t^m \sum_n a_n t^n - t^{-m} \sum_n a_n t^{-n}\right)h \in I \\
	\Leftrightarrow &~ (t^m r(t) - t^{-m}r(t^{-1}))h \in I \\
	\Leftrightarrow &~ (p(t) - p(t^{-1}))h \in I
\end{align*}

(ii) Now, let $I$ be a closed ideal (i.e., $I = Z(I)$) and $l \in \mathbb{Z}$. Assume $(t^jr(t) - t^{-j}r(t^{-1}))h \in I$ for $j = 0, -1$. By recalling that $r(t) = \sum_n a_n t^n$, we have that this is equivalent to \[\sum_na_n(t^{j+n} - t^{-j-n})h = \sum_na_nc_{j+n} \in I,\] for $j = 0, -1$. By Theorem \ref{ThmBasisRels}, we have that
\begin{align*}
	\sum_n a_nc_n = \left[\sum_n a_n b_n, b_0\right] \quad \text{and} \quad \sum_n a_n c_{n-1} = \left[\sum_n a_n b_n, b_1\right].
\end{align*}
It then follows from Corollary \ref{b0b1gen} that 
\begin{align*}
	&\qquad \left[\sum_n a_n b_n, \mathcal{L}^{\omega}\right] \subseteq I.
\end{align*}
Hence, we have that \[\sum_n a_n b_n \in Z(I) = I.\] This is equivalent to 
\begin{align*}
	&\sum_n a_n(t^ne + t^{-n}f) \in I\\
	\Longleftrightarrow ~ &r(t)e + r(t^{-1})f \in I \qquad \text{since $r(t) = \sum_n a_n t^n$.} 
\end{align*}
From (i) we now get that $(p(t) - p(t^{-1}))h \in I$ for all $p(t) \in r(t)k[t,t^{-1}]$. In particular, by considering the cases when $p(t) = t^{m+j}r(t)$ for $m \in \mathbb{Z}, j = 0, \pm1$, we have that $(t^{m+j}r(t) - t^{-m-j}r(t^{-1}))h \in I$ for $j = 0, \pm1$. This is equivalent to $\sum_n a_n(t^{m+j+n}-t^{-m-j-n})h = \sum_n a_n c_{m+n-j} \in I$ for $j = 0, \pm1$. By Theorem \ref{ThmBasisRels}, we have that
\begin{align*}
	\sum_n a_nc_{m+n} &= \left[\sum_n a_n b_{m+n}, b_0\right] = \left[-\sum_n a_n b_{-m-n}, b_0\right], \\
	\sum_n a_n c_{m+n-1} &= \left[\sum_n a_n b_{m+n}, b_1\right], \\
	\sum_n a_n c_{m+n+1} &= \left[-\sum_n a_n b_{-m-n}, b_1\right].
\end{align*}
From this and Corollary \ref{b0b1gen}, it follows that
\[\qquad \left[\sum_n a_n b_{m+n}, \mathcal{L}^{\omega}\right] \subseteq I \quad \text{and} \quad \left[\sum_n a_n b_{-m-n}, \mathcal{L}^{\omega}\right] \subseteq I.\]
Hence, we have that \[\Longrightarrow \sum_n a_n b_{m+n} \in Z(I) = I  \quad \text{and} \quad \sum_n a_n b_{-m-n} \in Z(I) = I.\] This is equivalent to 
\begin{align*}
	\Longleftrightarrow \sum_n a_n (t^{m+n}e + t^{-m-n}f), \sum_n a_n (t^{-m-n}e - t^{m+n}f) \in I \\
	\Longleftrightarrow \left(t^mr(t)e + t^{-m}r(t^{-1})f\right), \left(t^{-m}r(t^{-1})e + t^mr(t)f\right) \in I.
\end{align*}
Hence, for $p(t) = t^mr(t)$ ($m \in \mathbb{Z}$ arbitrary), we have that 
\begin{align*}
	p(t^{\pm1})e + p(t^{\mp1})f \in I.
\end{align*}
Then the same argument as in (i) allows us to conclude that $p(t^{\pm1})e + p(t^{\mp1})f \in I$ for all $p(t) \in r(t)k[t,t^{-1}]$. 

The equivalent relations in the second part of (ii) simply follow from the preceding ones of the lemma. 
\end{proof}

\begin{theo} \emph{(\cite[Theorem 2]{DateERoanS00})} Let $I$ be an ideal in $\mathcal{L}^{\omega}$. Then $I$ is closed if and only if $I = I_{P(t)}$ for a reciprocal polynomial $P(t)$ whose zeros at $t = \pm1$ are of even multiplicity. 
\end{theo}

\begin{proof} If $I = I_{P(t)}$ for a reciprocal polynomial $P(t)$ whose zeros at $t = \pm1$ are of even multiplicity, then we have already proven in Lemma \ref{DRLem2} that $I$ is closed. 

Next, let $I$ be a closed ideal. Denote \[\overline{I} := \{r(t) \in k[t,t^{-1}]~|~r(t)e + r(t^{-1})f \in I\}. \] By Lemma \ref{DRLem3}(ii), $\overline{I}$ is an ideal in $k[t,t^{-1}]$ invariant under the involution $r(t) \mapsto r(t^{-1})$. Let $P(t)$ be the unique monic polynomial which generates the ideal $\overline{I}\cap k[t]$ of the polynomial ring $k[t]$. Note that if $d$ is the degree of $P(t)$, then it follows from Lemma \ref{DRLem3}(ii) that $t^dP(t^{-1})$ also generates $\overline{I} \cap k[t]$. Hence, we must have that $P(t) = \pm t^dP(t^{-1})$, i.e., $P(t)$ is a reciprocal polynomial and we have that $\overline{I} = P(t)k[t,t^{-1}]$. 

The aim is to show that $I = I_{P(t)}$. Once this is proven, the result will again follow directly from Lemma \ref{DRLem2}. We begin by showing the inclusion $I_{P(t)} \subseteq I$. Let $X = p(t)e + p(t^{-1})f + q(t)h \in I_{P(t)}$. Then $p(t) \in P(t)k[t,t^{-1}] = \overline{I}$ by definition of $I_{P(t)}$ and so $p(t)e + p(t^{-1})f \in I$. Furthermore, since $q(t) \in P(t)k[t,t^{-1}] = \overline{I}$, it then follows that 
\begin{align*}
	[q(t)h, b_0] &= [q(t)h, e+f] = 2q(t)e + 2q(t^{-1})f \in I,\\
	[q(t)h, b_1] &= [q(t)h, te+t^{-1}f] = 2tq(t)e + 2t^{-1}q(t^{-1})f \in I.
\end{align*}
Hence, from Corollary \ref{b0b1gen}, we have that $q(t)h \in Z(I) = I$. Finally, since $p(t)e + p(t^{-1})f \in I$ and $q(t)h \in I$, we have that $X \in I$ and can conclude that $I_{P(t)} \subseteq I$.


So, it remains to show that $I \subseteq I_{P(t)}$.  Assume, on the contrary, that there exists an $X \in I \backslash I_{P(t)}$ and write
\begin{align*}
	X = p(t)e + p(t^{-1})f + q(t)h, \qquad q(t) = q_+(t) - q_+(t^{-1}),
\end{align*} 
where $p(t) \in k[t,t^{-1}]$ and $q_+(t) \in k[t]$. Note that if the degree of $q_+(t)$ is greater than that of $P(t)$, then by the division algorithm there exist $\overline{q}_+(t), h(t) \in k[t]$ such that \[q_+(t) = h(t)P(t) + \overline{q}_+(t),\] where the degree of $\overline{q}_+(t)$ is less than that of $P(t)$. Since $P(t)$ is reciprocal, this implies that 
\begin{align*}
	q(t)	&= h(t)P(t) + \overline{q}_+(t) - h(t^{-1})P(t^{-1}) - \overline{q}_+(t^{-1}) \\
				&= (h(t) \pm h(t^{-1})t^{-d})P(t) + \overline{q}_+(t) - \overline{q}_+(t^{-1}),
\end{align*}
where $d$ is the degree of $P(t)$. Since we are considering elements $X = p(t)e + p(t^{-1})f + q(t)h \notin I_{P(t)}$ it follows that, without loss of generality, we may assume that the polynomial $q_+(t)$ is of degree less than that of $P(t)$. Furthermore, since $q(t) = q_+(t) - q_+(t^{-1})$, we may also assume that $q_+(0)=0$. 

Let $\Delta$ be the set of elements $X$ as above, i.e.,
\begin{align*}
	\Delta = \left\{  p(t)e + p(t^{-1})f + q(t)h \in I \backslash I_{P(t)} \left| \begin{array}{c} q(t) = q_+(t) - q_+(t^{-1}) \text{ for } q_+(t) \in k[t], \\ \deg(q_+(t)) < \deg(P(t)), q_+(0)=0. \end{array} \right. \right\}
\end{align*}

Note that the polynomials $p(t)$ and $q(t)$ defining elements such as $X$ in $\Delta$ are not divisible by $P(t)$. Indeed if $q(t)$ were divisible by $P(t)$, then $q(t)h \in I_{P(t)}$, and hence we would have that $p(t)e + p(t^{-1})f \in I\backslash I_{P(t)}$. So $p(t) \in \overline{I} = P(t)k[t,t^{-1}]$, i.e., $p(t)$ would be divisible by $P(t)$. But this would mean that $X \in I_{P(t)}$, a contradiction. Hence, we have that $q(t)$ is not divisible by $P(t)$. Furthermore, it now follows that \[ [X, e + f] = [p(t)e + p(t^{-1})f + q(t)h, e + f] = 2q(t)e + 2q(t^{-1})f + (p(t)-p(t^{-1}))h \] is in $I \backslash I_{P(t)}$ because $q(t)$ is not divisible by $P(t)$. If $p(t)$ were divisible by $P(t)$, then $(p(t) - p(t^{-1}))h \in I_{P(t)}$, hence we would have that $2q(t)e + 2q(t^{-1})f \in I \backslash I_{P(t)}$. So $q(t) \in \overline{I} = P(t)k[t,t^{-1}]$, i.e., $q(t)$ would be divisible by $P(t)$, a contradiction. We have now shown that both $q(t)$ and $p(t)$ are not divisible by $P(t)$. 

Next, let $\tilde{X}$ be an element of $\Delta$ with degree of $q_+(t)$ being maximal. Since $\tilde{X} \in I$, we have that the following are also elements of $I$
\begin{align*}
	[\tilde{X},e+f] &= 2q(t)e + 2q(t^{-1})f + v(t)h,\\
	[[\tilde{X},e+f],te+t^{-1}f] &= 2tv(t)e - 2t^{-1}v(t)f + 2(t^{-1}q(t)-tq(t))h,
\end{align*}
where $v(t) = p(t) - p(t^{-1})$.  Recall that $p(t)$ and $q(t)$ are not divisible by $P(t)$. Furthermore, we claim that $v(t)$ is also not divisible by $P(t)$. To show this, we assume the contrary, i.e., that $P(t) \mid v(t)$. In this case, we have that $v(t)h \in I_{P(t)}$ and so \[2q(t)e + 2q(t^{-1})f \in I \backslash I_{P(t)}.\] It then follows that $q(t) \in \overline{I} = P(t)k[t,t^{-1}]$, i.e., $q(t)$ is divisible by $P(t)$. This is a contradiction. Hence allowing us to conclude that $v(t)$ is indeed not divisible by $P(t)$. 

Moreover, note that if $\tilde{q}(t) := tq_+(t) + t^{-1}q_+(t) \in k[t]$, then one has that $t^{-1}q(t) - tq(t^{-1}) = \tilde{q}(t) - \tilde{q}(t^{-1})$. The degree of the polynomial $\tilde{q}(t)$, and that of $\tilde{q}(t)-\tilde{q}(0)$, is greater than the degree of $q_+(t)$. Hence, $2tv(t)e - 2t^{-1}v(t)f + 2(t^{-1}q(t)-tq(t))h \notin \Delta$ and so $t^{-1}q(t) - tq(t^{-1})$ is divisible by $P(t)$. Furthermore, it follows that $P(t) \mid tv(t)$. This implies that there exists $f(t) \in k[t]$ such that \[tv(t) = f(t)P(t).\] Since $P(0) \neq 0$, we have that $t$ does not divide $P(t)$. Hence, $t$ must divide $f(t)$, i.e., $f(t) = tg(t)$ for some $g(t) \in k[t]$. So, we have that \[tv(t) = tg(t)P(t),\] which is equivalent to \[v(t) = g(t)P(t).\] But this implies that $P(t) \mid v(t)$, a contradiction. Hence, we have that $I = I_{P(t)}$ and Lemma \ref{DRLem2} then allows us to conclude that the zero-multiplicities of $P(t)$ at $t = \pm 1$ are even.
\end{proof}


\section{Alternate Involution}
\label{sec:AlternateInvolution}

The purpose of this section is to show that the Onsager algebra is also isomorphic to the Lie subalgebra of $\mathcal{L} = \mathfrak{sl}_2 \otimes k[t,t^{-1}]$ fixed by a slightly different involution which appears in the literature, for example in \cite{UglovDBIvanovIT96}. This involution is mainly used when one is looking to define the $\mathfrak{sl}_n$-analogue of the Onsager algebra. 

Consider the involution on $\mathfrak{sl}_2$ defined by 
\begin{align}
    \tilde{e} := -f, \quad \tilde{f} := -e, \quad \tilde{h} := -h .
\end{align}
This induces the involution $\sigma$ on $\mathcal{L} = \mathfrak{sl}_2 \otimes k[t,t^{-1}]$,
\begin{align}
    \sigma(x \otimes p(t)) = \tilde{x} \otimes p(t^{-1}), \quad \text{for } p(t) \in k[t,t^{-1}], x \in \mathfrak{sl}_2. \label{altinv}
\end{align}
 
In order to show that the subalgebra of $\mathcal{L}$ fixed by $\omega$, $\mathcal{L}^{\omega}$, is isomorphic to the subalgebra of $\mathcal{L}$ fixed by $\sigma$, $\mathcal{L}^{\sigma}$, one can refer to \cite[Lemma 3.3]{NeherESavageASenesiP09}. But the latter is stated in the language of schemes, hence for our purposes we will mimic the result in the context of Lie algebras as follows.

\begin{lem} Let $\tau: \mathcal{L} \rightarrow \mathcal{L}$ be defined by $u \otimes t \mapsto AuA^{-1} \otimes t^{-1}$ for $u \in \mathfrak{sl}_2$ and $A = \left( \begin{smallmatrix} i & 0 \\ 0 & 1 \end{smallmatrix} \right)$ ($i= \sqrt{-1}$). Then
\begin{itemize}
	\item[(i)] $\omega = \tau \sigma \tau^{-1}$,
	\item[(ii)] $\tau$ intertwines $\sigma$ and $\omega$; i.e., for any $x \in \mathcal{L}$, $\tau(\sigma(x)) = \omega(\tau(x))$, 
	\item[(iii)] $\mathcal{L}^{\omega} \cong \mathcal{L}^{\sigma}$ as Lie algebras. 
\end{itemize}
\end{lem}

\begin{proof}
\noindent\begin{itemize}
	\item[(i)] Begin by noting that $\tau^{-1}(u \otimes t) = A^{-1}uA \otimes t^{-1}$ and $A^{-1} = \frac{1}{i} \left( \begin{smallmatrix} 1 & 0 \\ 0 & i \end{smallmatrix} \right) = \left( \begin{smallmatrix} -i & 0 \\ 0 & 1 \end{smallmatrix} \right)$. Furthermore, we have that
		\begin{align*}
			\tau \sigma \tau^{-1} (e) &= \tau \sigma (A^{-1}eA) = \tau \sigma (A^{-1}e) = \frac{1}{i} \tau \sigma (e) = -\frac{1}{i} AfA^{-1} = Af = f = \omega(e), \\
			\tau \sigma \tau^{-1} (f)	&= \tau \sigma (A^{-1}fA) = -\frac{1}{i} \tau \sigma (f) = \frac{1}{i} \tau (e) = \frac{1}{i} AeA^{-1} = -A(ie) = e = \omega(f), \\
			\tau \sigma \tau^{-1} (h)	&= \tau \sigma(A^{-1}hA) = \frac{1}{i} \tau \sigma (ih) = \tau \sigma (h) = - \tau (h) = -AhA^{-1} = -h = \omega(h),\\
			\tau \sigma \tau^{-1} (t)	 &= \tau \sigma (t^{-1}) = \tau (t) = t^{-1} = \omega(t).													 
		\end{align*}
		Therefore $\tau \sigma \tau^{-1} = \omega$, as desired.
	\item[(ii)] Let $x$ be an arbitrary element of $\mathcal{L}$. Then \[ \omega(\tau(x)) = \tau \sigma \tau^{-1}(\tau(x)) = \tau \sigma(x). \] Therefore, $\tau$ indeed intertwines $\omega$ and $\sigma$.
	\item[(iii)] If $x \in \mathcal{L}^{\sigma}$ (i.e., $\sigma(x) = x$), then it follows from (ii) that \[ \omega(\tau(x)) = \tau(\sigma(x)) = \tau(x), \] i.e., $\tau(x) \in \mathcal{L}^{\omega}$. In the other direction, if $y \in \mathcal{L}^{\omega}$ (i.e., $\omega(y) = y$), then \[ \tau^{-1}(\omega(y)) = \tau^{-1}(y) \] and it follows from (i) that \[ \tau^{-1}(\omega(y)) = \tau^{-1}(\tau \sigma \tau^{-1}(y)) = \sigma(\tau^{-1}(y)) = \tau^{-1}(y), \] i.e., $\tau^{-1}(y) \in \mathcal{L}^{\sigma}$.
	
	Hence $\tau$ yields the desired automorphism and it follows that $\mathcal{L}^{\omega} \cong \mathcal{L}^{\sigma} (\cong \mathcal{O})$ as Lie algebras. 
\end{itemize}
\end{proof}

\begin{rmk} We include some remarks which put the lemma above in a broader context.
\noindent\begin{itemize}
	\item[(i)] Observe that the Chevalley involution $\omega$ is given by $\omega(u) = \Int \left( \begin{smallmatrix} 0 & 1 \\ 1 & 0 \end{smallmatrix} \right)$ where, for any $g \in$ GL$_2$, the map $\Int(g): \mathfrak{sl}_2 \rightarrow \mathfrak{sl}_2$ is the inner automorphism defined by $\Int(g)(u) = gug^{-1}$ (Thus $g = \left( \begin{smallmatrix} 0 & 1 \\ 1 & 0 \end{smallmatrix} \right)$ for the Chevalley involution).
	\item[(ii)] For any $g \in$ GL$_2$, the map \[\Int\left(g \left( \begin{smallmatrix} 0 & 1 \\ 1 & 0 \end{smallmatrix} \right) g^{-1}\right) = \Int(g) \omega \Int(g^{-1}) = \Int(g) \omega \Int(g)^{-1}\] is therefore another involution of $\mathfrak{sl}_2$.
	For example, if $g = \left( \begin{smallmatrix} -i & 0 \\ 0 & 1 \end{smallmatrix} \right)$ then we get \[ g \left( \begin{smallmatrix} 0 & 1 \\ 1 & 0 \end{smallmatrix} \right) g^{-1} = \left( \begin{smallmatrix} -i & 0 \\ 0 & 1 \end{smallmatrix} \right)\left( \begin{smallmatrix} 0 & 1 \\ 1 & 0 \end{smallmatrix} \right)\left( \begin{smallmatrix} i & 0 \\ 0 & 1 \end{smallmatrix} \right) = \left( \begin{smallmatrix} 0 & -i \\ i & 0 \end{smallmatrix} \right)\] and \[\Int\left( \begin{smallmatrix} 0 & -i \\ i & 0 \end{smallmatrix} \right)\left( \begin{smallmatrix} a & b \\ c & -a \end{smallmatrix} \right) = \left( \begin{smallmatrix} 0 & -i \\ i & 0 \end{smallmatrix} \right)\left( \begin{smallmatrix} a & b \\ c & -a \end{smallmatrix} \right)\left( \begin{smallmatrix} 0 & -i \\ i & 0 \end{smallmatrix} \right) = \left( \begin{smallmatrix} -a & -c \\ -b & a \end{smallmatrix} \right),\] i.e., $\Int\left( \begin{smallmatrix} 0 & -i \\ i & 0 \end{smallmatrix} \right) = \sigma$ as defined in \eqref{altinv}.
	\item[(iii)] The automorphism $\Int(g)$, $g \in$ GL$_2$, of $\mathfrak{sl}_2$ lifts to an automorphism INT($g$) of $\mathcal{L} = \mathfrak{sl}_2 \otimes k[t,t^{-1}]$ by \[ \text{INT}(g)(u \otimes p(t)) = \left(\Int(g)(u)\right) \otimes p(t^{-1}).\] 
	\item[(iv)] If $\sigma$ is an involution of $\mathfrak{sl}_2$, $g \in$ GL$_2$ and $\sigma' = \Int(g)\sigma\Int(g^{-1})$, then the corresponding involution of $\mathcal{L}$, \[u \otimes p(t) \mapsto \sigma(u) \otimes p(t^{-1}) \quad \text{and} \quad u \otimes p(t) \mapsto \sigma'(u) \otimes p(t^{-1})\] are intertwined by INT($g$), i.e., $\sigma' =$ INT$(g)\sigma$INT$(g^{-1})$ (we use $\sigma$ and $\sigma'$ to denote the corresponding involutions of $\mathcal{L} = \mathfrak{sl}_2 \otimes k[t,t^{-1}]$). As a corollary, $\mathcal{L}^{\sigma'} \cong \mathcal{L}^{\sigma}$. 
	\item[(v)] It is known that Aut $\mathfrak{sl}_2 = \{\Int(g) ~|~ g \in \text{GL}_2\}$. Suppose $\sigma \in \text{Aut }\mathfrak{sl}_2$ is an involution. Then $\sigma = \Int(g)$ for some $g \in \text{GL}_2$. Thus
	\begin{align}
		u = \sigma^2(u) = \Int(g^2)(u) = g^2ug^{-2}, \label{g2u=ug2}
	\end{align}
	for all $u \in \mathfrak{sl}_2$. Let $g^2 = \left( \begin{smallmatrix} a & b \\ c & d \end{smallmatrix} \right)$ and $ u = \left( \begin{smallmatrix} j & m \\ l & -j \end{smallmatrix} \right).$ So, it follows from (\ref{g2u=ug2}) that
	\begin{align*}
		g^2u &= ug^2 \\
		\Leftrightarrow ~ \left( \begin{matrix} a & b \\ c & d \end{matrix} \right) \left( \begin{matrix} j & m \\ l & -j \end{matrix} \right) &= \left( \begin{matrix} j & m \\ l & -j \end{matrix} \right) \left( \begin{matrix} a & b \\ c & d \end{matrix} \right) \\
		\Leftrightarrow ~ \left( \begin{matrix} aj+bl & am-bj \\ cj+dl & cm-dj \end{matrix} \right) &= \left( \begin{matrix} ja+mc & jb+md \\ la-jc & lb-jd \end{matrix} \right).
	\end{align*}
	From this we have that 
	\begin{align}
				bl &= cm, \label{1}\\
				am &= 2bj + dm, \label{2}\\ 
				dl &= -2cj + al. \label{3}
	\end{align}
	So, we have that (\ref{1}), (\ref{2}) and (\ref{3}) are satisfied for all $u = \left( \begin{smallmatrix} j & m \\ l & -j \end{smallmatrix} \right) \in \mathfrak{sl}_2$. In particular, if we consider when $m=0$ and $l=1$, it follows that $b=0$. On the other hand, if we consider when $m=1$ and $l=0$, then we have that $c=0$. And finally, since $b=c=0$, (\ref{2}) and (\ref{3}) then allow us to conclude that $a=d$. Hence, the equation (\ref{g2u=ug2}) implies that $g^2 = aE_2 = \left( \begin{smallmatrix} a & 0 \\ 0 & a \end{smallmatrix} \right)$ for some $a \in k$. One can use this to classify all involutions of $\mathfrak{sl}_2$.
\end{itemize}
\end{rmk}
\cleardoublepage

\chapter{Subalgebra of the Tetrahedron Algebra}
\label{sec:SubalgebraOfTheTetrahedronAlgebra}

Throughout this chapter, we return to $k$ being an arbitrary field of characteristic zero. 

The content of this chapter consists of discussing the so-called tetrahedron algebra and its relationship with the Onsager algebra. This proves to be useful because of the relationship between the tetrahedron algebra and the so-called three-point $\mathfrak{sl}_2$-loop algebra. Hence, we begin by considering the latter relationship in the first section and then move on to exploring the link between the tetrahedron algebra and the Onsager algebra. Finally, using this, we will describe all the ideals of the Onsager algebra.  

\section{The Tetrahedron Algebra}
\label{sec:TheTetrahedronAlgebra}

In order to study the relationship between the tetrahedron algebra and the so-called three-point $\mathfrak{sl}_2$ loop algebra, we first clearly require some definitions. 

\begin{defn}[Tetrahedron Algebra] The \emph{tetrahedron algebra}, denoted $\boxtimes$ (\cite[Definition 1.1]{HartwigBTerwilligerP07}), is the Lie algebra over $k$ with generators
\begin{align}
    \{X_{ij} | i,j \in \{0,1,2,3\}, i \neq j\} \label{tetragen}
\end{align}
and the relations
\begin{align}
    X_{ij} + X_{ji} = 0 &\quad \text{for } i \neq j, \label{tetra1} \\
    [X_{ij},X_{jk}] = 2(X_{ij} + X_{jk}) &\quad \text{for mutually distinct } i,j,k, \label{tetra2} \\
    [X_{hi},[X_{hi},[X_{hi},X_{jk}]]] = 4[X_{hi},X_{jk}] &\quad \text{for mutually distinct } h,i,j,k. \label{tetra3}
\end{align}
\end{defn}

\begin{defn}[Three-point $\mathfrak{sl}_2$ Loop Algebra] The \emph{three-point $\mathfrak{sl}_2$ loop algebra} is the Lie algebra over $k$ consisting of the $k$-vector space $\mathfrak{sl}_2 \otimes k[t,t^{-1},(1-t)^{-1}]$ and Lie bracket \[ [x \otimes a, y \otimes b] = [x,y] \otimes ab \] for $x, y \in \mathfrak{sl}_2$ and $a, b \in k[t,t^{-1},(1-t)^{-1}]$. Note that $k[t,t^{-1},(1-t)^{-1}]$ is the subalgebra of the field of rational functions $k(t)$, generated by $t$, $t^{-1}$ and $(1-t)^{-1}$.  
\end{defn}

In \cite[Theorem 11.5]{HartwigBTerwilligerP07} and \cite[Corollary 2.9]{ElduqueA07}, it is shown that $\boxtimes$ and $\mathfrak{sl}_2 \otimes k[t,t^{-1},(1-t)^{-1}]$ are isomorphic as Lie algebras under the homomorphism $\psi$ defined by
\begin{align*}
    \psi (X_{12}) = x \otimes 1, &\qquad \psi (X_{03}) = y \otimes t + z \otimes (t-1), \\
    \psi (X_{23}) = y \otimes 1, &\qquad \psi (X_{01}) = z \otimes t' + x \otimes (t'-1), \\
    \psi (X_{31}) = z \otimes 1, &\qquad \psi (X_{02}) = x \otimes t'' + y \otimes (t''-1).
\end{align*}
where $t' = 1-t^{-1}$, $t'' = (1-t)^{-1}$ and
\begin{align*}
    \left\{ x = 2e -h = \left( \begin{array}{cc} -1 & 2 \\ 0 & 1 \end{array} \right), \ y = -2f - h = \left( \begin{array}{cc} -1 & 0 \\ -2 & 1 \end{array} \right), \ z = h = \left( \begin{array}{cc} 1 & 0 \\ 0 & -1 \end{array} \right) \right\}
\end{align*}
is a basis of $\mathfrak{sl}_2$. Since this basis satisfies $[x,y] = 2(x+y)$, $[y,z] = 2(y+z)$ and $[z,x] = 2(z+x)$, we call $x, y, z$ the \emph{equitable basis} for $\mathfrak{sl}_2$ (\cite{GBenkartPTerwilliger07}).

When proving that $\psi$ is an isomorphism, Elduque considers the following elements of $\mathfrak{sl}_2 \otimes k[t,t^{-1},(1-t)^{-1}]$,
\begin{align}
        u_0 &= \frac{1}{4} \psi (X_{02} + X_{31}) = \frac{1}{4}(z \otimes 1 + x \otimes t'' + y \otimes (t'' - 1)), \label{u0} \\
        u_1 &= \frac{1}{4}\psi(X_{03} + X_{12}) = \frac{1}{4}(x \otimes 1 + y \otimes t + z \otimes (t - 1)), \label{u1} \\
        u_2 &= \frac{1}{4}\psi(X_{01} + X_{23}) = \frac{1}{4}(y \otimes 1 + z \otimes t' + x \otimes (t' - 1)), \label{u2}
\end{align}
which are proven to generate $\mathfrak{sl}_2 \otimes k[t,t^{-1},(1-t)^{-1}]$ as a Lie algebra over $k$. Simple computations show that $u_0$, $u_1$, $u_2$ satisfy the following relations:
\begin{align}
    [u_0,u_1] = -u_2t, \qquad \quad [u_1,u_2] = -u_0t', \qquad \quad [u_2,u_0] = -u_1t''. \label{uiRelations}
\end{align}
Furthermore, one can prove that $\{u_0,u_1,u_2\}$ is a basis of $\mathfrak{sl}_2 \otimes k[t,t^{-1},(1-t)^{-1}]$ as a module over $k[t,t^{-1},(1-t)^{-1}]$.

With this, we can proceed to discuss the relationship between the tetrahedron algebra and the Onsager algebra in section \ref{sec:TheTetrahedronAlgebraAndARealizationOfTheOnsagerAlgebra}. The results will then be used to describe all the ideals, in particular the closed ideals of the Onsager algebra in section \ref{sec:IdealsOfTheOnsagerAlgebra}. But first, let us note a relation between the tetrahedron algebra and $\mathfrak{sl}_2$. 

\begin{note} \label{notesl2copy} In \cite[Corollary 12.4]{HartwigBTerwilligerP07}, it is shown that for mutually distinct $h,i,j \in \{0,1,2,3\}$ the elements $X_{hi}, X_{ij}, X_{jh}$ form an `equitable' basis for a subalgebra of $\boxtimes$ that is isomorphic to $\mathfrak{sl}_2$. Note that the relation (\ref{tetra2}) is an indication of this isomorphism.
\end{note}

\section{The Tetrahedron Algebra and a Realization of the Onsager Algebra}
\label{sec:TheTetrahedronAlgebraAndARealizationOfTheOnsagerAlgebra}

\begin{lem} For mutually distinct $h,i,j,k \in \{0,1,2,3\}$, there exists a unique Lie algebra homomorphism
\begin{align*}
    \varphi : ~ & \mathcal{O} \longrightarrow \boxtimes, \\
                          & A  \longmapsto X_{hi}, \\
                          & B  \longmapsto X_{jk}.
\end{align*}
\end{lem}

\begin{proof} By definition of the tetrahedron algebra,
\begin{align*}
    [X_{hi},[X_{hi},[X_{hi},X_{jk}]]] = 4[X_{hi},X_{jk}],
\end{align*}
i.e., the elements $X_{hi}$, $X_{jk}$ satisfy the Dolan-Grady relations (\ref{DG1}) and (\ref{DG2}) which, as seen in Chapter 2, define the Onsager algebra. This proves the existence of such a homomorphism. This homomorphism is unique since $A$ and $B$ generate $\mathcal{O}$.
\end{proof}

\begin{lem} \label{LemmaJOAinter} Let $\mathcal{O}_A$ be the linear span of $\{A_m |\  m \in \mathbb{Z}\}$ and $\mathcal{O}_G$ be the linear span of $\{G_l |\  l \in \mathbb{N_+}\}$; hence $\mathcal{O} = \mathcal{O}_A \oplus \mathcal{O}_G$. If $J$ is a nontrivial ideal of $\mathcal{O}$, then $J$ intersects $\mathcal{O}_A$ nontrivially, i.e., $J \cap \mathcal{O}_A \neq \{0\}$.
\end{lem}

\begin{proof} Let $J$ be a nontrivial ideal of $\mathcal{O}$ and let $0 \neq x \in J$. Since $\mathcal{O} = \mathcal{O}_A \oplus \mathcal{O}_G$ (by linear independence of $\{A_m, G_l ~|~ m \in \mathbb{Z}, l \in \mathbb{N}_+\}$), $x$ can be written as:
\begin{align}
    x = \sum_{m \in \mathbb{Z}} c_m A_m + \sum_{n \in \mathbb{N}} d_n G_n \label{formx}
\end{align}
where only finitely many $c_m$ and $d_n$ are non-zero.

If $c_m = 0$ for all $m \in \mathbb{Z}$, we can write $x$ in the form:
\[x = \sum_{n \leq N} d_n G_n\]
where $d_N \neq 0$ and $d_n = 0$ for all $n > N$. Since $J$ is an ideal of $\mathcal{O}$, it follows that $[x,A_0] \in J$. Hence,
\begin{align*}
    [x,A_0] &= \left[ \sum_{n \leq N} d_n G_n, A_0\right] \\
                    &= \sum_{n \leq N} d_n[G_n,A_0] \\
                    &= \sum_{n \leq N} d_n (A_n - A_{-n})  \in J\backslash \{0\}.
\end{align*}
So this allows us to state that any nontrivial ideal $J$ contains an element of the form (\ref{formx}), where there exists at least one $m$ such that $c_m \neq 0$. It then follows that there exists an $M \in \mathbb{N}$ such that $c_M \neq 0$ and $c_m = 0$ for $|m| > M$. In sum, $0 \neq x \in J$ can be written as
\begin{align*}
	x = \sum_{|m| \leq M} c_m A_m + \sum_{n \leq N} d_n G_n, \tag*{$c_M \neq 0$.}
\end{align*}
Again, since $J$ is an ideal of $\mathcal{O}$, it follows that $[G_l,x] \in J$. Hence,
\begin{align*}
    [G_l,x] &= \left[G_l, \sum_{|m| \leq M} c_m A_m + \sum_{|n| \leq N} c_n G_n\right] \\
                    &= \sum_{|m| \leq M} c_m [G_l,A_m] + \sum_{|n| \leq N} c_n [G_l,G_n] \\
                    &= \sum_{|m| \leq M} c_m [G_l,A_m]  \in J.
\end{align*}
Thus, it follows that $J$ intersects $\mathcal{O}_A$ nontrivially, i.e., $J \cap \mathcal{O}_A \neq \{0\}$.
\end{proof}

\begin{lem} \emph{(\cite[Lemma 2.8]{ElduqueA07})} \label{Lemma2.8Injective} Fix mutually distinct $h,i,j,k \in \{0,1,2,3\}$. The Lie algebra homomorphism
\begin{align*}
    \phi: \mathcal{O} \rightarrow \mathfrak{sl}_2 \otimes k[t,t^{-1},(1-t)^{-1}]
\end{align*}
determined by $\phi(A) = \psi(X_{hi})$ and $\phi(B) = \psi(X_{jk})$ is injective.
\end{lem}

\begin{proof} Without loss of generality, it is sufficient to prove the lemma for $(h,i,j,k)$ $= (1,2,0,3)$. Notice that it follows from relation (\ref{Onsrel2}) that $\ad_{G_1}|_{\mathcal{O}_A} : \mathcal{O}_A \rightarrow \mathcal{O}_A$ and $\ad_{A_0}|_{\mathcal{O}_G} : \mathcal{O}_G \rightarrow \mathcal{O}_A$ are injective.

Note that if $\phi$ is not injective, then its kernel, $K$, would be a nontrivial ideal of $\mathcal{O}$. From Lemma \ref{LemmaJOAinter}, we know that $K \cap \mathcal{O}_A \neq 0$, thus $\ker(\phi|_{\mathcal{O}_A})$ is nontrivial; i.e., $\phi|_{\mathcal{O}_A}$ is not injective. Hence, in order to show that $\phi$ is injective, it suffices to show that $\phi|_{\mathcal{O}_A}$ is injective. So, since \[\{(\ad_{G_1})^m(A_0 + A_1), (\ad_{G_1})^m(A_0 - A_1) \mid m \geq 0\}\] is a basis of $\mathcal{O}_A$, it remains to show that \[\{\phi((\ad_{G_1})^m(A_0 + A_1)), \phi(\ad_{G_1})^m(A_0 - A_1)) \mid m \geq 0 \}\] is linearly independent in $\mathfrak{sl}_2 \otimes k[t,t^{-1},(1-t)^{-1}]$.

From (\ref{u2}), we note that
\begin{align*}
    4u_2t &= y \otimes t + z \otimes t't + x \otimes t(t'-1) \\
                &= y \otimes t + z \otimes (1-t^{-1})t + x \otimes t(1-t^{-1}-1) \\
                &= y \otimes t + z \otimes (t-1) - x \otimes 1,
\end{align*}
and from (\ref{u1}),
\begin{align*}
    4u_1    &= x \otimes 1 + y \otimes t + z \otimes(t-1).
\end{align*}
Hence, we have that \[\phi(A) = \psi(X_{12}) = x \otimes 1 = 2(u_1 - u_2t)\] and \[\phi(B) = \psi(X_{03}) = y \otimes t + z \otimes (t-1) = 2(u_1 + u_2t).\]

Furthermore, using \eqref{uiRelations} and recalling that $A_0 = A$ and $A_1 = B$, it follows that
\begin{align*}
    \phi(G_1) &= \frac{1}{2} \phi([A_1,A_0]) = \frac{1}{2} [\phi(A_1),\phi(A_0)] = \frac{1}{2} [\psi(X_{03}),\psi(X_{12})] \\
                        &= \frac{1}{2} [2(u_1 + u_2t),2(u_1 - u_2t)] = 2 [u_1 + u_2t, u_1 - u_2t] \\
                        &= 2 (u_0tt' + u_0tt') = 4u_0tt' = 4u_0(t-1),\\
    \phi(A_0 + A_1) &= 2(u_1 - u_2t) + 2(u_1 + u_2t) = 4u_1,\\
    \phi(A_0 - A_1) &= 2(u_1 - u_2t) - 2(u_1 + u_2t) = -4u_2t.
\end{align*}
So, we want to check if \[\{(\ad_{u_0(t-1)})^m(u_1), (\ad_{u_0(t-1)})^m(u_2t) \mid m \geq 0 \}\] is linearly independent over $k$. Since
\begin{align*}
    [u_0(t-1),u_1] &= -u_2t(t-1), \\
    [u_0(t-1),u_2t] &= u_1t''(t-1)t \\
                                    &= u_1(1-t)^{-1}(1-t)t \\
                                    &= u_1t,
\end{align*}
it follows that
\begin{align*}
    (\ad_{u_0(t-1)})^{2m} (u_1) &= u_1(t(t-1))^m = u_1(t^{2m} + \mbox{(lower terms in $t$)}), \\
    (\ad_{u_0(t-1)})^{2m+1} (u_1) &= -u_2(t(t-1))^{m+1} = -u_2(t^{2m+2} + \mbox{(lower terms in $t$)}), \\
    (\ad_{u_0(t-1)})^{2m} (u_2t) &= -u_2t(t(t-1))^m = -u_2(t^{2m+1} + \mbox{(lower terms in $t$)}),\\
    (\ad_{u_0(t-1)})^{2m+1} (u_2t) &= u_1t(t(t-1))^m = u_1(t^{2m+1} + \mbox{(lower terms in $t$)}),
\end{align*}
and these elements are linearly independent over $k$. This leads to the conclusion that $\phi$ as defined in the lemma is indeed injective.
\end{proof}

\begin{prop} \label{tetraisomOns} For mutually distinct $h,i,j,k \in \{0,1,2,3\}$ the subalgebra of $\boxtimes$ generated by $X_{hi}$, $X_{jk}$ is isomorphic to $\mathcal{O}$.
\end{prop}

\begin{proof} Let $\Omega$ denote the subalgebra of $\boxtimes$ generated by $X_{hi}$, $X_{jk}$. It follows from Lemma \ref{Lemma2.8Injective} that the epimorphism $\mathcal{O} \rightarrow \Omega$ such that $A \mapsto X_{hi}$ and $B \mapsto X_{jk}$ is injective. Hence, $\mathcal{O} \cong \Omega$. \end{proof}

\begin{rmk} In \cite[Proposition 7.8]{HartwigBTerwilligerP07}, it is shown that $\boxtimes = \Omega \oplus \Omega' \oplus \Omega'' \cong \mathcal{O} \oplus \mathcal{O} \oplus \mathcal{O}$ where $\Omega$ (respectively $\Omega'$, $\Omega''$) denotes the subalgebra of $\boxtimes$ generated by $X_{12}$ and $X_{03}$ (respectively by  $X_{23}$ and $X_{01}$, and by $X_{31}$ and $X_{02}$).
\end{rmk}

In \cite[Prop. 2.6(i)]{ElduqueA07}, it is shown that $\psi (\Omega) = u_0(t-1)k[t] \oplus u_1k[t] \oplus u_2tk[t]$. Hence, if \[v_0 = u_0(t-1), ~ v_1 = u_1 ~ \text{ and } ~ v_2 = u_2t,\] the Onsager algebra $\mathcal{O}$ can be identified with $v_0 k[t] \oplus v_1 k[t] \oplus v_2 k[t]$ where $v_0$, $v_1$ and $v_2$ freely generate $\mathcal{O}$ as a $k[t]$-algebra and, by (\ref{uiRelations}), its $k[t]$-linear Lie bracket satisfies the following relations:
\begin{align}
[v_0,v_1] = -v_2(t-1),\qquad [v_1,v_2] = -v_0, \qquad [v_2,v_0] = v_1t. \label{virels}
\end{align}

\begin{prop} $\mathcal{O}$ is generated, as a Lie algebra over $k$, by $v_0$, $v_1$ and $v_2$.
\end{prop}

\begin{proof} Consider the subalgebra V of $\mathcal{O}$ generated by $v_0$, $v_1$ and $v_2$. Note that for any $n \in \mathbb{N}$,
\begin{align*}
    (\ad_{v_2})^2(v_0t^n) = (\ad_{v_2})(v_1t^{n+1}) = v_0t^{n+1} \in V.
\end{align*}
Furthermore, $[v_2,v_0t^n] = v_1t^{n+1} \in V$ and $[v_1,v_0t^n] = v_2t^n(t-1) \in V$. Hence, all generators of the $k$-vector space, $v_ik[t]$ for $i \in \{0,1,2\}$, are included in $V$ and, since $\mathcal{O}$ is identified with $v_0 k[t] \oplus v_1 k[t] \oplus v_2 k[t]$, it follows that $\mathcal{O} = V$. 
\end{proof}

\begin{rmk}[Tetrahedron algebra in figures] Note that the structure of the tetrahedron algebra can be summarized in the following figure which the reader will easily recognize as a tetrahedron.

\begin{figure}[ht]
	\centering
		\includegraphics[width=7cm]{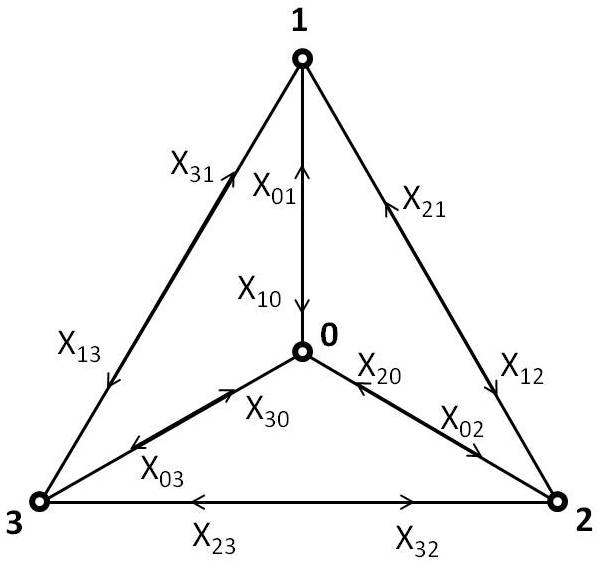}
	\caption{Tetrahedron algebra in figures.}
	\label{fig:tetrahedron}
\end{figure}

\noindent Here one notes that the relation (\ref{tetra1}) is illustrated by the fact that the labels on each edge add up to zero. Remark \ref{notesl2copy} is illustrated by the fact that the labels of every three edges, forming a face of the tetrahedron, are elements of an equitable basis of a subalgebra isomorphic to $\mathfrak{sl}_2$. And finally the labels of every two non-adjacent edges generate a subalgebra isomorphic to the Onsager algebra (hence illustrating Proposition \ref{tetraisomOns}). 
\end{rmk}

\section{Ideals of the Onsager Algebra}
\label{sec:IdealsOfTheOnsagerAlgebra}

In this section, we utilize the relationship between the tetrahedron algebra and the Onsager algebra to determine the ideals and in particular the closed ideals of the latter. We first note that for any ideal $J$ of $k[t]$, $\mathcal{O}J$ is an ideal of $\mathcal{O}$.

\begin{prop} \emph{(\cite[Proposition 5.6]{ElduqueA07})} Let $I$ be an ideal of $\mathcal{O}$ and consider the following subspace of $k[t]$:
\begin{align}
J_I = \{p(t) \in k[t] ~| ~\exists ~p_1(t), p_2(t) \in k[t], ~ v_0 p(t) + v_1 p_1(t) + v_2 p_2(t) \in I\}. \label{JI}
\end{align}
Then
\begin{enumerate}
    \item[(i)] $J_I$ is an ideal of $k[t]$.
    \item[(ii)] $I$ lies between $\mathcal{O}J_It(t-1)$ and $\mathcal{O}J_I$, i.e., $\mathcal{O}J_It(t-1) \subseteq I \subseteq \mathcal{O}J_I$.
\end{enumerate}
\end{prop}

\begin{proof}

\noindent (i) Let $p(t) \in J_I$. So there exist $p_1(t)$ and $p_2(t) \in k[t]$ such that
\begin{center}
$x = v_0p(t) + v_1p_1(t) + v_2p_2(t) \in I$.
\end{center}
But since
\begin{align*}
    (\ad_{v_2})^2(x) = (\ad_{v_2})(v_1tp(t) + v_0p_1(t)) = v_0tp(t) + v_1tp_1(t) \in I,
\end{align*}
it follows that $tp(t) \in J_I$. Since $J_I$ is clearly a subspace, it follows that $J_I$ is an ideal of $k[t]$.

\noindent (ii) Let $x = v_0p_0(t) + v_1p_1(t) + v_2p_2(t) \in I$. From (\ref{JI}), $p_0(t) \in J_I$. Furthermore, since
\begin{center}
$[v_2,x] = v_1tp_0(t) + v_0p_1(t) \in I$

and $\quad [-v_1,x] = -v_2p_0(t)(t-1) + v_0p_2(t) \in I$,
\end{center}
then $p_1(t)$ and $p_2(t)$ also belong to $J_I$. Hence $x \in v_0J_I \oplus v_1J_I \oplus v_2J_I = \mathcal{O}J_I$, and so $I \subseteq \mathcal{O}J_I$.

Moreover, since
\begin{align*}
    [v_0,[v_1,x]] &= v_1p_0(t)t(t-1) \in I, \\
    [v_2,[v_0,[v_1,x]]] &= v_0p_0(t)t(t-1) \in I, \\
    [-v_0,[v_2,x]] &= v_2p_0(t)t(t-1) \in I,
\end{align*}
it follows that $\mathcal{O}J_It(t-1) = v_0J_It(t-1) \oplus v_1J_It(t-1) \oplus v_2J_It(t-1) \subseteq I$, which leads to the conclusion that $\mathcal{O}J_It(t-1) \subseteq I \subseteq \mathcal{O}J_I$. \end{proof}

Next, consider the opposite direction. Given a non-zero ideal $J = q(t)k[t]$ of $k[t]$, what are the possible forms for an ideal $I$ of $\mathcal{O}$ with $J = J_I$?

\begin{prop} \emph{(\cite[Proposition 5.8]{ElduqueA07})} \label{Prop5.8Ideals} Let $J = q(t)k[t]$ be a non-zero ideal of k[t]. Then the ideals $I$ of $\mathcal{O}$ with $J = J_I$ are the subspaces
\begin{center}
$I = \mathcal{O}Jt(t-1) \oplus S$,
\end{center}
where $S$ is of one of the following types (for $w_i = v_iq(t), i = 0,1,2$):
\begin{enumerate}
    \item[(i)] $S = k\epsilon(w_0t + w_1t) \oplus k\delta(w_0t - w_1t) \oplus k\epsilon\delta\gamma w_2t \oplus k\epsilon'(w_0(t-1) + w_2(t-1)) \oplus k\delta'(w_0(t-1) - w_2(t-1)) \oplus k\epsilon'\delta'\gamma'w_1(t-1)$ where $\epsilon, \delta, \gamma, \epsilon', \delta', \gamma'$ are either $0$ or $1$, with $\epsilon + \delta \neq 0 \neq \epsilon' + \delta'$ (as $J = J_I$).
    \item[(ii)] $S = S_{\eta} = \Span\{w_0t, w_1t, w_0(t-1), w_2(t-1), w_2t + \eta w_1(t-1)\}$, with $0 \neq \eta \in k$.
\end{enumerate}
\end{prop}

\begin{proof} First, note that $k[t]$ decomposes into the following direct sum of subspaces:
\begin{center}
$k[t] = t(t-1)k[t] \oplus (kt \oplus k(t-1)).$
\end{center}
Hence, $\mathcal{O}$ can be written as a direct sum of vector spaces
\begin{align}
\mathcal{O} = \mathcal{O}t(t-1) \oplus \Span\{v_it,v_i(t-1) \mid i = 0,1,2\} \label{Odirsum}
\end{align}
and if $J = q(t)k[t]$ is a non-zero ideal of $k[t]$, then
\begin{align}
\mathcal{O}J = \mathcal{O}Jt(t-1) \oplus \Span\{v_iq(t)t,v_iq(t)(t-1) \mid i = 0,1,2\}. \label{OJdirsum}
\end{align}
Furthermore, for any non-zero ideal $J$ of $k[t]$,
\begin{align}
[\mathcal{O}t(t-1),\mathcal{O}J] \subseteq \mathcal{O}Jt(t-1).\label{inclOJt(t-1)}
\end{align}
So, there exist the following natural bijections:
\begin{center}
$\{$ideals $I$ of $\mathcal{O}$ with $\mathcal{O}Jt(t-1) \subseteq I \subseteq \mathcal{O}J\}$

$\updownarrow$

$\{\mathcal{O}$-submodules of $\mathcal{O}J/\mathcal{O}Jt(t-1)\}$

$\updownarrow$

$\{\mathcal{O}/\mathcal{O}t(t-1)$-submodules of $\mathcal{O}J/\mathcal{O}Jt(t-1)\},$
\end{center}
where the term ``$\mathcal{O}$-submodule" (respectively, ``$\mathcal{O}/\mathcal{O}t(t-1)$-submodule") means that we consider the adjoint action of $\mathcal{O}$ (respectively, of $\mathcal{O}/\mathcal{O}t(t-1)$) on the ideals $\mathcal{O}J$ and $\mathcal{O}Jt(t-1)$ and then take the quotient representation. 

Given an element $x \in \mathcal{O}$ (respectively, $x \in \mathcal{O}J$), let us denote by $\overline{x}$ its class modulo $\mathcal{O}t(t-1)$ (respectively, modulo $\mathcal{O}Jt(t-1)$). Thus, from (\ref{Odirsum}) and (\ref{OJdirsum}), for a nonzero ideal $J = q(t)k[t]$ of $k[t]$,
\begin{center}
$\mathcal{O}/\mathcal{O}t(t-1) = \bigoplus^2_{i=0} (k\overline{v_it} \oplus k\overline{v_i(t-1)})$,

\

$\mathcal{O}J/\mathcal{O}Jt(t-1) = \bigoplus^2_{i=0} (k\overline{w_it} \oplus k\overline{w_i(t-1)})$,
\end{center}
where $w_i = v_iq(t)$, $i=0,1,2$.

\

In order to determine the $\mathcal{O}/\mathcal{O}t(t-1)$-submodules of $\mathcal{O}J/\mathcal{O}Jt(t-1)$, one considers the action of $\mathcal{O}/\mathcal{O}t(t-1)$ on $\mathcal{O}J/\mathcal{O}Jt(t-1)$.

We begin by noting that since $[v_it, w_j(t-1)] = t(t-1)[v_i,w_j]$ and $[v_i(t-1), w_jt] = (t-1)t[v_i,w_j]$ belong to $\mathcal{O}Jt(t-1)$, then \[[\overline{v_it}, \overline{w_j(t-1)}] = [\overline{v_i(t-1)}, \overline{w_jt}] = 0.\]

It then remains to consider the action of $\overline{v_it}$ on $\overline{w_jt}$ and the action of $\overline{v_i(t-1)}$ on $\overline{w_j(t-1)}$ for all $i,j \in \{0,1,2\}$.
\begin{enumerate}
    \item $[\overline{v_it}, \overline{w_jt}]$
        \begin{enumerate}
            \item $[\overline{v_1t}, \overline{w_0t}] = [\overline{v_0t}, \overline{w_1t}] = 0$
            \item $[\overline{v_2t}, \overline{w_0t}] = \overline{w_1t^3} = \overline{w_1t}, \quad [\overline{v_0t}, \overline{w_2t}] = \overline{-w_1t}$
            \item $[\overline{v_2t}, \overline{w_1t}] = \overline{w_0t^2} = \overline{w_0t}, \quad [\overline{v_1t}, \overline{w_2t}] = \overline{-w_0t}$
        \end{enumerate}
    \item $[\overline{v_i(t-1)}, \overline{w_j(t-1)}]$
        \begin{enumerate}
            \item $[\overline{v_0(t-1)}, \overline{w_2(t-1)}] = [\overline{v_2(t-1)}, \overline{w_0(t-1)}] = 0$
            \item $[\overline{v_1(t-1)}, \overline{w_2(t-1)}] = -\overline{w_0(t-1)^2} = \overline{w_0(t-1)}, \newline [\overline{v_2(t-1)}, \overline{w_1(t-1)}] = \overline{-w_0(t-1)}$
            \item $[\overline{v_1(t-1)}, \overline{w_0(t-1)}] = \overline{w_2(t-1)^3} = \overline{w_2(t-1)}, \newline [\overline{v_0(t-1)}, \overline{w_1(t-1)}] = \overline{-w_2(t-1)}$
        \end{enumerate}
\end{enumerate}
From the above computations, one notes that $\overline{w_2t}$ generates the $\mathcal{O}$-submodule $\bigoplus^2_{i=0} k\overline{w_it}$ and that $\overline{w_1(t-1)}$ generates the $\mathcal{O}$-submodule $\bigoplus^2_{i=0} k\overline{w_i(t-1)}$. Furthermore, it follows that the eigenvalues of the action of $\overline{v_2t}$ on $\mathcal{O}J/\mathcal{O}Jt(t-1)$ are:
\begin{enumerate}
    \item[(i)] $0$, with eigenspace $k\overline{w_2t} \oplus (\bigoplus^2_{i=0} k\overline{w_j(t-1)})$,
    \item[(ii)] $1$, with eigenspace $k(\overline{w_0t + w_1t})$,
    \item[(iii)] $-1$, with eigenspace $k(\overline{w_0t - w_1t})$.
\end{enumerate}
It also follows from the above computations that the eigenvalues of the action of $\overline{v_1(t-1)}$ on $\mathcal{O}J/\mathcal{O}Jt(t-1)$ are:
\begin{enumerate}
    \item[(i)] $0$, with eigenspace $k\overline{w_1(t-1)} \oplus (\bigoplus^2_{i=0} k\overline{w_jt})$,
    \item[(ii)] $1$, with eigenspace $k(\overline{w_0(t-1) + w_2(t-1)})$,
    \item[(iii)] $-1$, with eigenspace $k(\overline{w_0(t-1) - w_2(t-1)})$.
\end{enumerate}

At this point, it is important to note that ideals $I$ of $\mathcal{O}$ with $J = J_I$ are subspaces of the form $I = \mathcal{O}Jt(t-1) \oplus S$ for some $S$. Our goal is to explore the possibilities for $S$. Since any $\mathcal{O}$-submodule  of $\mathcal{O}J/\mathcal{O}Jt(t-1)$ is the direct sum of its intersections with the previous eigenspaces, only the following four cases can occur:

\begin{enumerate}
    \item[(1)] $kw_2t$ and $kw_1(t-1)$ are both included in $S$.
    \item[(2)] Either $kw_2t$ or $kw_1(t-1)$ is included in $S$, but not both.
    \item[(3)] Neither $kw_2t$ nor $kw_1(t-1)$ are included in $S$.
    \item[(4)] A linear combination of $w_2t$ and $w_1(t-1)$ (excluding the three previous cases) is included in $S$.
\end{enumerate}

Note that since $w_2t$ generates the $\mathcal{O}$-submodule $\bigoplus^2_{i=0} kw_it$, if $kw_2t \subseteq S$ then $k(w_0t + w_1t)$ and $k(w_0t - w_1t)$ are included in $S$. Similarly, if $kw_1(t-1) \subseteq S$ then $k(w_0(t-1) + w_2(t-1))$ and $k(w_0(t-1) - w_2(t-1))$ are included in $S$.

Hence, in the case of (1), $S = k(w_0t + w_1t) \oplus k(w_0t - w_1t) \oplus kw_2t \oplus k(w_0(t-1) + w_2(t-1)) \oplus k(w_0(t-1) - w_2(t-1)) \oplus kw_1(t-1)$.

Next, consider case (2) when $kw_2t \subseteq S$. In this scenario, $S = k(w_0t + w_1t) \oplus k(w_0t - w_1t) \oplus kw_2t \oplus k\epsilon'(w_0(t-1) + w_2(t-1)) \oplus k\delta'(w_0(t-1) - w_2(t-1))$ where $\epsilon'$ and $\delta'$ are either 0 or 1.

\emph{Claim:} $\epsilon' + \delta' \neq 0$

Assume $\epsilon' + \delta' = 0$, i.e., $\epsilon' = \delta' = 0$. Hence, $S = k(w_0t + w_1t) \oplus k(w_0t - w_1t) \oplus kw_2t$. Consider an arbitrary $p(t) \in J = J_I$. Recall from (\ref{JI}) the definition \[J_I = \{p(t) \in k[t] ~| ~\exists ~p_1(t), p_2(t) \in k[t], ~ v_0 p(t) + v_1 p_1(t) + v_2 p_2(t) \in I\}.\] It follows that $p(t) = q(t)r(t)$ for some $r(t) \in k[t]$ and that there exist $p_1(t), p_2(t) \in k[t]$ such that $v_0p(t) + v_1p_1(t) + v_2p_2(t) \in I$. Since $w_0 = v_0q(t)$, we have that $w_0r(t) + v_1p_1(t) + v_2p_2(t) \in I$. This is only true if $r(t)$ is divisible by $t$; i.e., if and only if $p(t)$ is divisible by $t$. So, it follows that all elements of $J$ are multiples of $t$, i.e., $J = tq(t)k[t]$. But this contradicts the fact that $J = q(t)k[t]$. Hence proving our claim.

On the other hand, when $kw_1(t-1) \subseteq S$, it follows that $S = k\epsilon(w_0t + w_1t) \oplus k\delta(w_0t - w_1t) \oplus k(w_0(t-1) + w_2(t-1)) \oplus k(w_0(t-1) - w_2(t-1)) \oplus kw_1(t-1)$ where $\epsilon$ and $\delta$ are either $0$ or $1$. If $\epsilon + \delta = 0$, a similar argument as above would show that $J = (t-1)q(t)k[t]$, which again contradicts $J = q(t)k[t]$. Hence, in this scenario, $\epsilon + \delta \neq 0$.

For the case (3) where neither $kw_2t$ nor $kw_1(t-1)$ are included in $S$, we have that $S = k\epsilon(w_0t + w_1t) \oplus k\delta(w_0t - w_1t) \oplus k\epsilon'(w_0(t-1) + w_2(t-1)) \oplus k\delta'(w_0(t-1) - w_2(t-1))$ where $\epsilon, \delta, \epsilon', \delta'$ are either $0$ or $1$. We also note that, for the same arguments as in case (2) above, $\epsilon' + \delta' \neq 0 \neq \epsilon + \delta$. Cases (1) to (3) are then summarized in part (i) of the proposition.

And finally, in case (4), we simply have that $S = \Span\{w_0t, w_1t, w_0(t-1), w_2(t-1), w_2t + \eta w_1(t-1)\}$, with $0 \neq \eta \in k$. This case is exactly part (ii) of the proposition.
\end{proof}

In view of the equivalence between ideals $I$ of $\mathcal{O}$ with $\mathcal{O}Jt(t-1) \subseteq I \subseteq \mathcal{O}J$ and $\mathcal{O}/\mathcal{O}t(t-1)$-submodules of $\mathcal{O}J/\mathcal{O}Jt(t-1)$, it is of interest to study the Lie algebra $\mathcal{O}/\mathcal{O}t(t-1)$.

\begin{lem} The Lie algebra $B = \mathcal{O}/\mathcal{O}t(t-1) = \bigoplus^2_{i=0} (k\overline{v_it} \oplus k\overline{v_i(t-1)})$ is solvable, but not nilpotent.
\end{lem}

\begin{proof} Consider the following cases for $i,j \in \{0,1,2\}$:
\begin{enumerate}
        \item $[\overline{v_it}, \overline{v_j(t-1)}] = 0$, since $[v_it, v_j(t-1)] = t(t-1)[v_i,v_j] \in \mathcal{O}t(t-1)$, 
        \item $[\overline{v_it}, \overline{v_jt}]$ where $i \neq j$,
            \begin{enumerate}
                \item $[\overline{v_1t}, \overline{v_0t}] = [\overline{v_0t}, \overline{v_1t}] = 0$,
                \item $[\overline{v_2t}, \overline{v_0t}] = \overline{v_1t^3} = \overline{v_1t}$,
   			        \item $[\overline{v_2t}, \overline{v_1t}] = \overline{v_0t^2} = \overline{v_0t}$,
            \end{enumerate}
        \item $[\overline{v_i(t-1)}, \overline{v_j(t-1)}]$ where $i \neq j$,
            \begin{enumerate}
            		\item $[\overline{v_0(t-1)}, \overline{v_2(t-1)}] = [\overline{v_2(t-1)}, \overline{v_0(t-1)}] = 0$,
                \item $[\overline{v_1(t-1)}, \overline{v_2(t-1)}] = -\overline{v_0(t-1)^2} = \overline{v_0(t-1)}$,
                \item $[\overline{v_1(t-1)}, \overline{v_0(t-1)}] = \overline{v_2(t-1)^3} = \overline{v_2(t-1)}$.
            \end{enumerate}
    \end{enumerate}
Hence, $B^{(1)} = [B,B] = k\overline{v_1t} \oplus k\overline{v_0t} \oplus k\overline{v_0(t-1)} \oplus k\overline{v_2(t-1)}$. Then, since
\begin{align*}
    [\overline{v_1t},\overline{v_0t}] &= [\overline{v_0(t-1)},\overline{v_2(t-1)}] = 0, \\
    [\overline{v_1t},\overline{v_0t}] &= \overline{v_2t^2(t-1)} = 0 = [\overline{v_0t},\overline{v_1t}], \\
    [\overline{v_0(t-1)},\overline{v_2(t-1)}] &= -\overline{v_1t(t-1)^2} = 0 = [\overline{v_2(t-1)},\overline{v_0(t-1)}],
\end{align*}
it follows that $B^{(2)} = 0$ and hence $B$ is solvable.

From above we have that $B^1 = B^{(1)} = [B,B] = k\overline{v_1t} \oplus k\overline{v_0t} \oplus k\overline{v_0(t-1)} \oplus k\overline{v_2(t-1)}$. Next, we want to determine $B^2 = [B,B^1]$. So we consider the following:
\begin{align*}
    [\overline{v_0t}, \overline{v_1t} + \overline{v_0t} + \overline{v_0(t-1)} + \overline{v_2(t-1)}] &= -\overline{v_2t^2(t-1)} - \overline{v_1t^2(t-1)} \\ &= 0, \\
    [\overline{v_0(t-1)}, \overline{v_1t} + \overline{v_0t} + \overline{v_0(t-1)} + \overline{v_2(t-1)}] &= -\overline{v_1t(t-1)^2} = 0, \\
    [\overline{v_1t}, \overline{v_1t} + \overline{v_0t} + \overline{v_0(t-1)} + \overline{v_2(t-1)}] &= \overline{v_2t^2(t-1)} = 0, \\
    [\overline{v_1(t-1)}, \overline{v_1t} + \overline{v_0t} + \overline{v_0(t-1)} + \overline{v_2(t-1)}] &= \overline{v_2(t-1)^3} - \overline{v_0(t-1)^2} \\ &= \overline{v_2(t-1)} + \overline{v_0(t-1)}, \\
    [\overline{v_2t},\overline{v_1t} + \overline{v_0t} + \overline{v_0(t-1)} + \overline{v_2(t-1)}] &= \overline{v_0t^2} + \overline{v_1t^3}\\ &= \overline{v_0t} + \overline{v_1t}, \\
    [\overline{v_2(t-1)}, \overline{v_1t} + \overline{v_0t} + \overline{v_0(t-1)} + \overline{v_2(t-1)}] &= \overline{v_1t(t-1)^2} = 0.
\end{align*}
It then follows that $B^2 = k\overline{v_0t} \oplus k\overline{v_0(t-1)} \oplus k\overline{v_1t} \oplus k\overline{v_2(t-1)} = B^1$. Hence $B$ is not nilpotent.
\end{proof}

\begin{prop} \emph{(\cite[Proposition 5.10]{ElduqueA07})} Let $I$ be a closed ideal of $\mathcal{O}$, with $0 \neq J = J_I = q(t)k[t]$. Then, with $w_i = v_ik[t]$, $i = 0,1,2$, $I$ is one of the following ideals:
\begin{enumerate}
    \item[(i)] $I = \mathcal{O}Jt(t-1) \oplus k(w_0t \pm w_1t) \oplus k(w_0(t-1) \pm w_2(t-1)).$
    \item[(ii)] $I = \mathcal{O}Jt(t-1) \oplus (\bigoplus^2_{i=0} kw_it) \oplus k(w_0(t-1) \pm w_2(t-1))$. \newline In this case, $I = \mathcal{O}Jt \oplus k(w_0 \pm w_2)$.
    \item[(iii)] $I = \mathcal{O}Jt(t-1) \oplus k(w_0t \pm w_1t) \oplus (\bigoplus^2_{i=0} kw_i(t-1))$. \newline In this case, $I = \mathcal{O}J(t-1) \oplus k(w_0 \pm w_1)$.
    \item[(iv)] $I = \mathcal{O}J$.
\end{enumerate}
\end{prop}

\begin{proof} To prove that the closed ideals of $\mathcal{O}$ are those stated above, it is a matter of sifting through all of the ideals with $0 \neq J = J_I = q(t)k[t]$ as described in Proposition \ref{Prop5.8Ideals}.

First, recall that $\mathcal{O} = \mathcal{O}t(t-1) \oplus \Span\{v_it,v_i(t-1) | i = 0,1,2\}$. Hence, if $J = q(t)k[t]$ and $w_i = v_ik[t]$, $i = 0,1,2$, then
\begin{align*}
    [w_2t,\mathcal{O}] &= [w_2t,\mathcal{O}t(t-1)] + k[w_2t,v_0t] + k[w_2t,v_1t] + k[w_2t,v_0(t-1)] \\
                                         & \qquad \qquad \qquad \qquad + k[w_2t,v_1(t-1)] \\
                                         &= [w_2t,\mathcal{O}t(t-1)] + kw_1t^3 + kw_0t^2 + kw_1t^2(t-1) + kw_0t(t-1) \\
                                         &\subseteq \mathcal{O}Jt(t-1) + kw_1t + kw_0t
\end{align*}
In the case of Proposition \ref{Prop5.8Ideals}(ii), $[w_2t,\mathcal{O}] \subseteq I$. Hence $w_2t \in Z(I)\backslash I$ and so $I$, in this case, is not a closed ideal. This is also the case for $I$ as in Proposition \ref{Prop5.8Ideals}(i) with ($\epsilon = \delta = 1$ and $\gamma = 0$) or ($\epsilon' = \delta' = 1$ and $\gamma' = 0$).

After eliminating these cases from Proposition \ref{Prop5.8Ideals}, we are left with the four cases specified in the statement of the proposition above. In order to confirm that these are in fact closed, consider $x = v_0p_0(t) + v_1p_1(t) + v_2p_2(t) \in Z(I)$ and its action on $v_0$, $v_1$ and $v_2$:
\begin{align}
    [x, v_0] &= v_2p_1(t)(t-1) + v_1p_2(t)t \in I \subseteq \mathcal{O}J \label{xv0} \\
    [x, v_1] &= -v_2p_0(t)(t-1) + v_0p_2(t)t \in I \subseteq \mathcal{O}J \label{xv1} \\
    [x, v_2] &= -v_1p_0(t)(t-1) - v_0p_1(t)t \in I \subseteq \mathcal{O}J \label{xv2}
\end{align}
It follows from (\ref{xv1}) and (\ref{xv2}) that $p_1(t), p_2(t)\in J$ and $(t-1)p_0(t), tp_0(t) \in J$. Furthermore,
\begin{align*}
    p_0(t) = tp_0(t) - (t-1)p_0(t) \in J.
\end{align*}
Therefore, $x \in Z(I)$ and hence $Z(I) \subseteq \mathcal{O}J$. This confirms that $I = \mathcal{O}J$ is a closed ideal of $\mathcal{O}$, proving (iv).

Next, consider cases (i)-(iii). From (\ref{xv1}) and (\ref{xv2}) it follows that
\begin{center}
    $t \mid p_2(t)$ or $(t-1) \mid p_2(t)\qquad$ and  $\qquad t \mid p_1(t)$ or $(t-1) \mid p_1(t)$.
\end{center}
Furthermore, from
\begin{align*}
    [[x,v_1],v_0] = -v_1p_0(t)t(t-1) \in I, \\
    [[x,v_2],v_0] = -v_2p_0(t)t(t-1) \in I,
\end{align*}
it follows that $v_1p_0(t), v_2p_0(t) \in \mathcal{O}J$ and hence $v_0p_0(t) = [v_2p_0(t),v_1] \in \mathcal{O}J$. By recalling that
\begin{align*}
\mathcal{O}J = \mathcal{O}Jt(t-1) \oplus \Span\{v_iq(t)t,v_iq(t)(t-1) | i = 0,1,2\}.
\end{align*}
we can conclude that $x \in I$ and so $Z(I) = I$, i.e., $I$ as defined in (i)-(iii) are closed ideals.

It remains to check the last statements in (ii) and (iii). Since $Jt = Jt(t-1) \oplus kq(t)t$, it follows that $\mathcal{O}Jt = \mathcal{O}Jt(t-1) \oplus (\bigoplus^2_{i=0} kw_it)$. Furthermore, $w_i(t-1) + w_i = w_it$ and $w_it \in I \cap (\mathcal{O}Jt \oplus k(w_0 \pm w_2))$. It follows directly that $I = \mathcal{O}Jt \oplus k(w_0 \pm w_2)$, hence proving statement (ii). Similarly, using the fact that $J(t-1) = Jt(t-1) \oplus kq(t)(t-1)$, one can obtain the result stated in (iii).
\end{proof}

\cleardoublepage

\addcontentsline{toc}{chapter}{Bibliography}
\bibliographystyle{alpha}
\bibliography{reflist}

\newpage

\end{document}